\documentclass{article}

\usepackage[sort]{natbib}
\usepackage[bookmarks=true,bookmarksnumbered=true,colorlinks=true,allcolors=black]{hyperref} 
\usepackage{amsmath,amssymb,amsfonts,amsthm}

\usepackage{geometry}
\usepackage{enumitem}

\newtheorem{theorem}{Theorem}[section]
\newtheorem{lem}{Lemma}[section]
\newtheorem{prop}{Proposition}[section]

\theoremstyle{definition}
\newtheorem{Def}{Definition}[section]
\newtheorem*{rmk*}{Remark}
\newtheorem{rmk}{Remark}[section]
\newtheorem{example}{Example}[section]

\DeclareMathOperator{\MRC}{MRC}

\DeclareMathOperator{\diag}{diag}
\DeclareMathOperator{\Cov}{Cov}

\numberwithin{equation}{section}

\makeatletter
    \renewcommand*{\section}{\@startsection{section}{1}{\z@}%
    {10pt}{5pt}{\reset@font\normalsize\bfseries}}
\makeatother
    
\makeatletter
    \renewcommand*{\subsection}{\@startsection{subsection}{2}{\z@}%
    {5pt}{5pt}{\reset@font\normalsize\mdseries\itshape}}
\makeatother

\makeatletter
    \renewcommand*{\subsubsection}{\@startsection{subsubsection}{3}{\z@}%
    {5pt}{5pt}{\reset@font\normalsize\mdseries\itshape}}
\makeatother

\usepackage[pagewise,mathlines]{lineno}

\newcommand*\patchAmsMathEnvironmentForLineno[1]{%
  \expandafter\let\csname old#1\expandafter\endcsname\csname #1\endcsname
  \expandafter\let\csname oldend#1\expandafter\endcsname\csname end#1\endcsname
  \renewenvironment{#1}%
     {\linenomath\csname old#1\endcsname}%
     {\csname oldend#1\endcsname\endlinenomath}}%
\newcommand*\patchBothAmsMathEnvironmentsForLineno[1]{%
  \patchAmsMathEnvironmentForLineno{#1}%
  \patchAmsMathEnvironmentForLineno{#1*}}%
\AtBeginDocument{%
\patchBothAmsMathEnvironmentsForLineno{equation}%
\patchBothAmsMathEnvironmentsForLineno{align}%
\patchBothAmsMathEnvironmentsForLineno{flalign}%
\patchBothAmsMathEnvironmentsForLineno{alignat}%
\patchBothAmsMathEnvironmentsForLineno{gather}%
\patchBothAmsMathEnvironmentsForLineno{multline}%
}

\allowdisplaybreaks

\title{Time endogeneity and an optimal weight function in pre-averaging covariance estimation}
\author{Yuta Koike\vspace{2mm}\\ 
\parbox{12cm}
{\small\itshape Risk Analysis Research Center, The Institute of Statistical Mathematics, 10-3 Midori-cho, Tachikawa, Tokyo 190-8562, Japan and CREST, JST, Email: kyuta@ism.ac.jp
}
\date{}
}

\begin{document}


\maketitle

\begin{abstract}

We establish a central limit theorem for a class of pre-averaging covariance estimators in a general endogenous time setting. In particular, we show that the time endogeneity has no impact on the asymptotic distribution if certain functionals of observation times are asymptotically well-defined. This contrasts with the case of the realized volatility in a pure diffusion setting. We also discuss an optimal choice of the weight function in the pre-averaging.\vspace{3mm}


\end{abstract}

\renewcommand*{\baselinestretch}{1.25}\selectfont

\section{Introduction}


In the past decade an improvement in the availability of financial high-frequency data has highlighted applications of the classic asymptotic theory for the quadratic covariation of a semimartingale to the inference for the covariance structure of asset returns.
Empirical evidences, however, suggest that at ultra-high frequencies asset price processes follow a semimartingale contaminated by noise (called \textit{microstructure noise}) rather than a pure semimartingale. In addition, at ultra-high frequencies financial data are possibly recorded at irregular times, and this causes the non-synchronicity of observation times between multiple assets.

Recently various approaches have been proposed for estimating the quadratic covariation matrix of a semimartingale observed at a high frequency in a non-synchronous manner with additive observation noise. Thus far the most prominent ones are the subsampling approach by (\citealt{B2011a,Zhang2011}), the realized kernel estimation by \cite{BNHLS2011}, the pre-averaging method by (\citealt{CKP2010,CPV2013}), the quasi maximum likelihood (QML) approach by (\citealt{AFX2010,LT2014}), and the spectral method by (\citealt{BHMR2014,BW2015}). In this paper we focus on the pre-averaging method, especially the \textit{modulated realized covariance} (abbreviated \textit{MRC}) introduced in \cite{CKP2010}.\footnote{The preliminary version \cite{Koike2013hit} of this paper focuses on the pre-averaged Hayashi-Yoshida estimator, which is another covariance estimator introduced in \cite{CKP2010}.}

Specifically, we consider the following model:
\begin{equation*}
Y_t=X_t+\epsilon_t,\qquad t\geq0,
\end{equation*}
where $X=(X_t)_{t\geq0}$ is a $d$-dimensional process (\textit{latent log-price}) and $\epsilon=(\epsilon_t)_{t\geq0}$ is a $d$-dimensional error process (\textit{microstructure noise}) which is, conditionally on the process $X$, centered and serially independent. We assume that $X$ is of the form
\begin{equation*}
X_t=X_0+\int_0^t a_s\mathrm{d}s+\int_0^t\sigma_s\mathrm{d}W_s,
\end{equation*}
where $a=(a_s)_{s\geq0}$ is an $\mathbb{R}^d$-valued c\`adl\`ag process, $\sigma=(\sigma_s)_{s\geq0}$ is an $\mathbb{R}^d\otimes\mathbb{R}^{d'}$-valued c\`adl\`ag volatility, and $W$ is a $d'$-dimensional Wiener process. Our objective is the quadratic covariation matrix of $X$ over some fixed interval $[0,T]$ (hereafter an asterisk denotes the transpose of a matrix):
\begin{equation*}
[X,X]_T=\int_0^T\Sigma_t\mathrm{d}t,\qquad\Sigma_t=\sigma_t\sigma_t^*.
\end{equation*}
Let us recall the definition of the MRC estimator in the synchronous sampling case. Suppose that we have observation data $(Y_{t_i})_{i=0}^N$ with observation times $0\leq t_0<t_1<\cdots<t_{N-1}<t_N\leq T$. 
Then, we choose a weight function $g$ on $[0,1]$ and a window size $K$ with which we associate the variables called the \textit{pre-averaging of $Y$}:
\begin{equation*}
\overline{Y}_{t_i}=\sum_{j=1}^{K-1}g\left(\frac{j}{K}\right)\Delta_{t_{i+j}}Y,\qquad
\Delta_{t_i}Y=Y_{t_{i}}-Y_{t_{i-1}}.
\end{equation*}
Since the observation errors are centered and serially independent, one can expect that $\overline{Y}_{t_i}$'s are close to the latent returns. Therefore, it is natural to consider the statistic $\sum_{i=0}^{N-K+1}\overline{Y}_{t_i}\left(\overline{Y}_{t_i}\right)^*$ as an estimator of $[X,X]$. In fact, \cite{CKP2010} showed that a bias corrected version of this estimator has the consistency and the asymptotic mixed normality as long as the observation times are equidistant ($t_i=i/N$) and we consider the situation where $N$ goes to infinity. This bias corrected version of the estimator is called the MRC estimator. 

Now, our main concern is the following two questions:
\begin{enumerate}[noitemsep,label=(\alph*)]

\item \textit{What happens when the observation times are endogenous?}

\item \textit{What is an optimal choice of the weight function $g$?}

\end{enumerate}
By the term ``endogenous'' we mean that the observation times depend on the latent log-price process $X$. Indeed, this issue is a relatively new subject in this area despite its importance for both theoretical and practical perspectives. 
In fact, in a pure one-dimensional diffusion setting, \citet{Fu2010b} showed that the endogeneity of the observation times can cause a bias of the asymptotic distribution of the realized volatility $\sum_{i=1}^N(\Delta_{t_i}X)^2$, which is a natural estimator for $[X,X]_T$ in such a setting. This phenomenon was independently found by \citet{LMRZZ2014}, and they also constructed a feasible central limit theorem as well as conducted empirical work that provides evidence that time endogeneity exists in financial data. In their analysis, the skewness and kurtosis of the returns $\Delta_{t_i}X$ play an important role. In particular, \cite{LMRZZ2014} showed that the former quantity has a strong connection with the covariance between the returns $\Delta_{t_i}X$ and the durations $t_i-t_{i-1}$ (see Remark 3 of \cite{LMRZZ2014}). \citet{RW2011} discussed the effect of this covariance on the volatility inference in a semi-parametric context. 
On the other hand, \citet{LZZ2012} derived a corresponding result to the one by \cite{LMRZZ2014} in the presence of microstructure noise. More precisely, they considered the following estimator: choose two integers $p$ and $q$ such that $p<q$, and set
\begin{align*}
\hat{Y}_{t_i}=\frac{1}{p}\sum_{j=0}^{p-1}\left(Y_{t_{i+j+q}}-Y_{t_{i+j}}\right).
\end{align*}
They showed that after appropriate scaling, the estimator $\sum_{i=0}^{N-(p+q)+1}(\hat{Y}_{t_i})^2$ is (possibly biased) asymptotic mixed normal under some regularity conditions; see Theorem 2 of \cite{LZZ2012} for details. In particular, according to their theory the asymptotic distribution of the estimation error $\sqrt{\frac{N}{q}}(\frac{1}{q}\sum_{i=0}^{N-(p+q)+1}(\hat{X}_{t_i})^2-[X,X]_T)$ due to the diffusion part is characterized by the probability limit of the processes given by
\begin{equation}\label{lzz}
\frac{N}{q}\sum_{i\geq q,t_i\leq t}\left(\sum_{j=1}^{q-1}\frac{q-j}{q}\Delta_{t_{i-j}}X\right)^2(\Delta_{t_{i}}X)^2\qquad
\mathrm{and}
\qquad\frac{\sqrt{N}}{q^{3/2}}\sum_{t_{i+p+q-1}\leq t}(\hat{X}_{t_i})^3
\end{equation}
for each $t\in[0,T]$. Note that if $p=q$ their estimator corresponds to the MRC estimator while $g(x)=x\wedge(1-x)$ and $K=2p$. In this paper we concentrate on the case where $p=q$ because the estimator achieves the optimal rate of convergence under these circumstances. 

 Therefore, regarding question (a) one possible approach would be to find some counterparts of the quantities in Eq.\eqref{lzz} in the multivariate and the general weight function setting. Unfortunately, we encounter some difficulties taking this approach. Namely, (i) it is not clear what the first quantity of \eqref{lzz} corresponds to in the general weight function setting, and (ii) it is preferable to give an explicit relation between the asymptotic distribution of the estimator and the tuning parameters $g$ and $K$ in order to obtain information on the optimal choice. This is especially important for question (b). The characterization by the quantities in \eqref{lzz}, however, is not adapted to this purpose because their limiting variables will depend on the tuning parameters in an unspecified way.
For this reason we introduce another set of conditions, which is independent of the choice of the tuning parameters, for handling the time endogeneity. 
Those conditions require that certain functionals of the observation times are asymptotically well-defined, and they seem reasonable for covering important models used in financial econometrics (cf.~Remark \ref{rmk:A4}). 
Interestingly, it turns out that the time endogeneity has \textit{no} impact on the asymptotic distribution of the MRC estimator under our conditions. This is quite different from the case of the realized volatility in a pure diffusion setting and makes the derivation of feasible limit theorems easier. 

On the other hand, regarding question (b) we try to find an optimal weight function in the sense that it minimizes the asymptotic variance of the MRC estimator in the univariate and parametric setting with equidistant observation times. To accomplish this, we need to extend the class of weight functions to those with unbounded supports. This is implemented in Section \ref{setting}. After that, in Section \ref{optimal} the double exponential density is shown to be an optimal weight function. In fact, it turns out that the double exponential density is a counterpart of the optimal kernel function for the flat-top realized kernel of \citet{BNHLS2008}. Therefore, the MRC estimator with the double exponential density and the oracle window size $K$ achieves the parametric efficiency bound from \cite{GJ2001a}. We also point out that this optimal weight function has a computational advantage and discuss two related topics, comparison with other efficient estimators and what happens in the presence of jumps.
This paper is organized as follows. 
Section \ref{setting} presents the mathematical model and the construction of the MRC estimator in a more general setting. 
Section \ref{main} is devoted to the main result of this paper. 
Section \ref{discussion:A4} discusses connections between our assumption on the observation times and quantities related to the observation times appearing in the preceding studies. 
Section \ref{section:optimal} deals with question (b) and related topics.     
All proofs are given in Section \ref{secproofs}.

\subsection*{General notation}

We denote by $\mathbb{R}^{d}\otimes\mathbb{R}^{d'}$ the set of $d\times d'$ matrices. For a matrix $A\in\mathbb{R}^{d}\otimes\mathbb{R}^{d'}$, we write the entries $A^{kl}$, $1\leq k\leq d$, $1\leq l\leq d'$, and the Frobenius norm $\|A\|$, i.e.~$\|A\|^2=\sum_{k=1}^{d}\sum_{l=1}^{d'}(A^{kl})^2$. For the case of $d'=1$ we write $A^k$ instead of $A^{k1}$. 
Finally, $\mathbb{D}_T^{d\times d'}$ denotes the space of $\mathbb{R}^{d}\otimes\mathbb{R}^{d'}$-valued c\`adl\`ag functions on $[0,T]$ equipped with the Skorokhod topology.


\section{The setting}\label{setting}


We begin by constructing a suitable stochastic basis on which our noisy process $Y$ is defined. 
We fix a stochastic basis $\mathcal{B}^{(0)}=(\Omega^{(0)},\mathcal{F}^{(0)},\mathbf{F}^{(0)}=(\mathcal{F}^{(0)}_t)_{t\geq0} ,P^{(0)})$ on which our latent process $X$ is defined, such that all the constituting processes $a,\sigma$ and $W$ are adapted. 
For each $k=1,\dots,d$ the observation times for $Y^k$ are denoted by $t^k_0,t^k_1,\dots$, i.e.~the observation data $(Y^k_{t^k_i})_{t^k_i\leq T}$ are available. They are assumed to be $\mathbf{F}^{(0)}$-stopping times which implicitly depend on a parameter $n\in\mathbb{N}$ representing the observation frequency and satisfy that $t^k_i\uparrow\infty$ as $i\to\infty$ and $\sup_{i\geq0}(t^k_i\wedge t-t^k_{i-1}\wedge t)\to^p0$ as $n\to\infty$ for any $t\in\mathbb{R}_+$, with setting $t^k_{-1}=0$ for notational convenience (hereafter we will refer to such a sequence as a \textit{sampling scheme} for short).

At the observation frequency $n\in\mathbb{N}$, we construct the stochastic basis $\mathcal{B}=(\Omega,\mathcal{F},\mathbf{F}=(\mathcal{F}_t)_{t\in\mathbb{R}_+} ,P)$ where our noisy process $Y$ is defined in the following way (for notational simplicity we subtract the index $n$ from $\mathcal{B}$). First, define the sequence $(\mathcal{T}^n_i)_{i\in\mathbb{Z}_+}$ of $\mathbf{F}^{(0)}$-stopping times sequentially by $\mathcal{T}^n_0=\min_{k=1,\dots,d}t^k_0$ and $\mathcal{T}^n_i=\min_{k=1,\dots,d}\min\{t^k_j:t^k_j>\mathcal{T}^n_{i-1}\}$ for $i=1,2\dots$. Namely, $(\mathcal{T}^n_i)$ is the increasing reordering of total observation times. $\mathcal{T}^n_i$'s are indeed $\mathbf{F}^{(0)}$-stopping times because they can be rewritten as $\mathcal{T}^n_i=\min_{k=1,\dots,d}\inf_{j\geq1} \left(t^k_j\right)_{\{t^k_j>\mathcal{T}^n_{i-1}\}},$ where for an $\mathbf{F}^{(0)}$-stopping time $\tau$ and a set $A\in\mathcal{F}^{(0)}_\tau$, we define $\tau_A$ by $\tau_A(\omega^{(0)})=\tau(\omega^{(0)})$ if $\omega^{(0)}\in A$; $\tau_A(\omega^{(0)})=\infty$ otherwise (see I-1.15 of \cite{JS}). 
For each $t\in\mathbb{R}_+$, we have a transition probability $Q_t(\omega^{(0)},\mathrm{d}u)$ from $(\Omega^{(0)},\mathcal{F}^{(0)}_t)$ into $\mathbb{R}^d$ satisfying 
$\int u Q_t(\omega^{(0)},\mathrm{d}u)=0$, which will correspond to the conditional distribution of the noise at the time $t$ given $\mathcal{F}^{(0)}_t$. 
We endow the space $\Omega^{(1)}=(\mathbb{R}^d)^{\mathbb{N}}$ with the product Borel $\sigma$-field $\mathcal{F}^{(1)}$ and with the probability measure $Q(\omega^{(0)},\mathrm{d}\omega^{(1)})$ which is the product $\otimes_{i\in\mathbb{N}}Q_{\mathcal{T}^n_i(\omega^{(0)})}(\omega^{(0)},\cdot)$. Then, we define the probability space $(\Omega,\mathcal{F},P)$ by 
$\Omega=\Omega^{(0)}\times\Omega^{(1)},$ 
$\mathcal{F}=\mathcal{F}^{(0)}\otimes\mathcal{F}^{(1)},$ and 
$P(\mathrm{d}\omega^{(0)},\mathrm{d}\omega^{(1)})=P^{(0)}(\mathrm{d}\omega^{(0)})Q(\omega^{(0)},\mathrm{d}\omega^{(1)}).$ 
Here, we impose the following measurability condition to ensure the probability measure $P$ is well-defined:
\begin{equation}\label{progressive}
\textrm{The process $(Q_t(\cdot,A))_{t\in\mathbb{R}_+}$ is $\mathbf{F}^{(0)}$-progressively measurable for any Borel subset $A$ of $\mathbb{R}^d$.}
\end{equation}
Any variable or process defined on either $\Omega^{(0)}$ or $\Omega^{(1)}$ can be considered in the usual way as a variable or a process on $\Omega$. In terms of financial applications, the space $\Omega^{(0)}$ stands for latent log-price processes, while the space $\Omega^{(1)}$ stands for microstructure noise. 
Now, the error process $\epsilon=(\epsilon_t)_{t\in\mathbb{R}_+}$ is realized as $\epsilon_t=\epsilon^0_{\mathsf{N}_n(t)}$, where $(\epsilon^0_i)_{i\in\mathbb{N}}$ denotes the canonical process on $(\Omega^{(1)},\mathcal{F}^{(1)})$ and $\mathsf{N}_n(t)=\sum_{i=0}^\infty1_{\{\mathcal{T}^n_i\leq t\}}$. By construction $(\epsilon_{\mathcal{T}^n_i})_{i\in\mathbb{Z}_+}$ is, conditionally on $\mathcal{F}^{(0)}$, serially independent. Finally, the filtration $\mathbf{F}$ is defined as the one generated by $\mathbf{F}^{(0)}$ and $(\epsilon_t)_{t\in\mathbb{R}_+}$. 

Next we explain the construction of the MRC estimator in the non-synchronous sampling setting, which is briefly discussed in Section 3.6 of \citet{CKP2010}. Following \citet{BNHLS2011}, we introduce the notion of \textit{refresh time}:
\begin{Def}[Refresh time]
The refresh times $T_0,T_1,\dots$ of the sampling schemes $\{(t^k_i)\}_{k=1}^d$ are defined sequentially by $T_0=\max\{t^1_0,\dots,t^d_0\}$ and 
$T_p=\max_{k=1,\dots,d}\min\{t^k_i:t^k_i>T_{p-1}\}$ for $p=1,2,\dots$.
\end{Def}
We introduce synchronized observation times by interpolating the next-ticks into the grid $(T_p)_{p=0}^\infty$. That is, for each $k=1,\dots,d$ define the synchronized observation times $(\tau^k_p)_{p=0}^\infty$ for $Y^k$ by $\tau^k_0=t^k_0$ and
\begin{align*}
\tau^k_p=\min\{t^k_i:t^k_i>T_{p-1}\},\qquad p=1,2,\dots.
\end{align*}
Here, unlike the preceding studies, we prefer the \textit{next-tick} interpolation scheme to the \textit{previous-tick} interpolation scheme because it automatically makes the resulting synchronized observation times stopping times as we have $\tau^k_p=\inf_{i\geq1} \left(t^k_i\right)_{\{t^k_i>T_{p-1}\}}$.

Based on the synchronized data constructed in the above, we introduce the pre-averaging as follows. 
We choose a sequence $k_n$ of positive integers and a number $\theta\in(0,\infty)$ such that
\begin{equation}\label{window} 
k_n=\theta\sqrt{n} +o(n^{1/4})
\end{equation}
as $n\to\infty$. We also choose a continuous function $g:[0,1]\rightarrow\mathbb{R}$ which is piecewise $C^1$ with a piecewise Lipschitz derivative $g'$ and satisfies 
\begin{equation}\label{weightPA}
g(0)=g(1)=0\qquad\mathrm{and}\qquad\int_0^1 g(x)^2\mathrm{d}x>0.
\end{equation}
After that, for any $d$-dimensional stochastic process $V=(V^1,\dots,V^d)$ we define the quantity
\begin{equation}\label{defPA}
\overline{V}^k_{i}=\sum_{p=1}^{k_n-1}g\left(\frac{p}{k_n}\right)\left(V^k_{\tau^k_{i+p}}-V^k_{\tau^k_{i+p-1}}\right),
\end{equation}
and set $\overline{V}_i=(\overline{V}^1_i,\dots,\overline{V}^d_i)^*$. Now the MRC estimator in the non-synchronous setting is defined as 
\begin{align*}
\MRC[Y]^n_T=\frac{1}{\psi_2 k_n}\sum_{i=0}^{N^n_T-k_n+1}\overline{Y}_{i}\left(\overline{Y}_{i}\right)^*
-\frac{\psi_1}{2\psi_2k_n^2}[Y,Y]^n_T,
\end{align*}
where $N^n_t=\max\{p:T_p\leq t\}$, $\psi_1=\int_0^1g'(x)^2\mathrm{d}x$, $\psi_2=\int_0^1g(x)^2\mathrm{d}x$ and
\begin{align*}
[Y,Y]^n_t=\sum_{p=1}^{N^n_t}\Delta_{p} Y\left(\Delta_{p} Y\right)^*,\qquad
\Delta_{p} Y=\left(Y^1_{\tau^1_p}-Y^1_{\tau^1_{p-1}},\dots,Y^d_{\tau^d_p}-Y^d_{\tau^d_{p-1}}\right)^*
\end{align*}
for each $t\in[0,T]$.\footnote{We set $\sum_{i=p}^q\equiv0$ if $p>q$ by convention.}
In the synchronous and equidistant sampling case, a central limit theorem for the MRC estimator has been shown in \cite{CKP2010}. 
One of our main purposes is to develop an asymptotic distribution theory for the MRC estimator in the situation where observation times are possibly non-synchronous and endogenous.

\begin{rmk}[Pre-averaged Hayashi-Yoshida estimator]
\citet{CKP2010} also discuss another type of covariance estimator for non-synchronous and noisy observations, which is a pre-averaged version of the Hayashi-Yoshida estimator from \cite{HY2005} and thus called the \textit{pre-averaged Hayashi-Yoshida estimator}. Formally, it is defined as the $\mathbb{R}^d\otimes\mathbb{R}^d$-valued variable whose $(k,l)$-th entry is given by
\[
\frac{1}{\left(k_n\int_0^1g(x)\mathrm{d}x\right)^2}\sum_{i,j:t^k_i\vee t^l_j\leq T}\overline{Y}^k_{t^k_i}\overline{Y}^l_{t^l_j}1_{\{[t^k_i,t^k_{i+k_n})\cap[t^l_j,t^l_{j+k_n})\neq\emptyset\}},
\]
where $\overline{Y}^k_{t^k_i}=\sum_{p=1}^{k_n-1}g\left(\frac{p}{k_n}\right)(Y^k_{t^k_{i+p}}-Y^k_{t^k_{i+p-1}})$ and $\overline{Y}^l_{t^l_j}$ is defined analogously. A central limit theorem for the pre-averaged Hayashi-Yoshida estimator is given by \citet{CPV2013} when $t^k_i$'s are asymptotically regular in the sense that they satisfy conditions in Proposition 2.54 of \cite{MZ2012} (see Assumption (T) of \cite{CPV2013} for details). One reason why we do not focus on this estimator is that it is generally less efficient than the MRC estimator (see Section 6 of \cite{CKP2010} and Remark 3.5 of \cite{CPV2013}). Another reason is that it is difficult to generalize the limit theorem given by \cite{CPV2013} to more general sampling settings because the asymptotic (co)variance of the estimator complexly depends on the special form of the observation times provided by their Assumption (T).\footnote{This point can be solved by pre-synchronizing the data similarly to our case, i.e.~consider $\overline{Y}^k_{i}$ instead of $\overline{Y}^k_{t^k_i}$; see \citet{Koike2014phy} for details. See also Section 6.3 of \citet{Bibinger2012} where other advantages of such a procedure are discussed for the case of the subsampling approach.} On the other hand, the pre-averaged Hayashi-Yoshida estimator has an advantage in terms of robustness; see Remarks 3.3 and 4.5 of \cite{CPV2013}.
\end{rmk}

Another main purpose is to find an optimal weight function $g$, and to accomplish this we need to extend the definition of the MRC estimator for weight functions with unbounded supports. Specifically, we consider a function $g$ on $\mathbb{R}$ satisfying the following condition:
\begin{enumerate}[noitemsep,label={\normalfont[W]}]

\item \label{hypo:W} (i) $g$ is continuous and piecewise $C^1$ with a piecewise Lipschitz derivative $g'$.

(ii) For every $r>0$ there exists a positive constant $C_r$ such that $|g(x)|+|g'(x)|\leq C_r(1+|x|^2)^{-r}$ for any $x\in\mathbb{R}$.

(iii) $\int_{-\infty}^\infty g(x)^2\mathrm{d}x>0$.

\end{enumerate} 
Then, a na\"ive extension of \eqref{defPA} is as follows:
\begin{equation*}
\overline{V}^k_{i}=\sum_{p=-i+1}^{N^n_T-i}g\left(\frac{p}{k_n}\right)\left(V^k_{\tau^k_{i+p}}-V^k_{\tau^k_{i+p-1}}\right).
\end{equation*}
Unfortunately, this definition suffers from the end effect. In fact, summation by parts yields
\begin{equation*}
\overline{\epsilon}^k_{i}=-\sum_{p=-i+1}^{N^n_T-i-1}\left\{g\left(\frac{p+1}{k_n}\right)-g\left(\frac{p}{k_n}\right)\right\}\epsilon^k_{\tau^k_{i+p}}+g\left(\frac{N^n_T-i}{k_n}\right)\epsilon^k_{\tau^k_{N^n_T}}
-g\left(\frac{-i+1}{k_n}\right)\epsilon^k_{\tau^k_{0}},
\end{equation*}
hence the noise $\epsilon^k_{\tau^k_{0}}$ and $\epsilon^k_{\tau^k_{N^n_T}}$ at the end points will have some impact on the limiting variable of $\overline{\epsilon}^k_{i}$ unless $g$ has a bounded support. To avoid this problem, we take the averages of the first and the last $k_n$ distinct observations:  
\begin{equation*}
\mathring{V}^k_{0}=\frac{1}{k_n}\sum_{p=0}^{k_n-1}V^k_{\tau^k_p},\qquad
\mathring{V}^k_{T}=\frac{1}{k_n}\sum_{p=N^n_T-k_n+1}^{N^n_T}V^k_{\tau^k_p}.
\end{equation*}
This idea is commonly used in the literature of realized kernel estimators and called the \textit{jittering}; see e.g.~\cite{BNHLS2008} and \cite{BNHLS2011}. Now we define the adjusted returns $(\widetilde{\Delta}_{\tau^k_p}V^k)_{p=k_n}^{N^n_T-k_n+1}$ based on the data
$\mathring{V}^k_0,V^k_{\tau^k_{k_n}},V^k_{\tau^k_{k_n+1}},\dots,V^k_{\tau^k_{N^n_T-k_n-1}},V^k_{\tau^k_{N^n_T-k_n}},\mathring{V}^k_T$. Namely, set $\widetilde{\Delta}_{\tau^k_p}V^k=V^k_{\tau^k_p}-V^k_{\tau^k_{p-1}}$ for $p=k_n+1,\dots,N^n_T-k_n$ and
\begin{align*}
\widetilde{\Delta}_{\tau^k_{k_n}}V^k_{k_n}=V^k_{\tau^k_{k_n}}-\mathring{V}^k_{0},\qquad
\widetilde{\Delta}_{\tau^k_{N^n_T-k_n+1}}V^k=\mathring{V}^k_{T}-V_{\tau^k_{N^n_T-k_n}}.
\end{align*} 
After that, our adjusted version of the pre-averaging is defined by
\begin{equation}\label{defILPA}
\widetilde{V}^k_{i,T}=\sum_{p=-i+k_n}^{N^n_T-k_n+1-i}g\left(\frac{p}{k_n}\right)\widetilde{\Delta}_{\tau^k_{i+p}}V^k
=\sum_{p=k_n}^{N^n_T-k_n+1}g\left(\frac{p-i}{k_n}\right)\widetilde{\Delta}_{\tau^k_p}V^k
\end{equation}
and $\widetilde{V}_{i,T}=(\widetilde{V}^1_{i,T},\dots,\widetilde{V}^d_{i,T})^*$. Consequently, our estimator takes the following form:
\begin{align*}
\widetilde{\MRC}[Y]^n_T=\frac{1}{\psi_2 k_n}\sum_{i=k_n}^{N^n_T-k_n+1}\widetilde{Y}_{i,T}\left(\widetilde{Y}_{i,T}\right)^*-\frac{\psi_1}{2\psi_2k_n^2}[Y,Y]^n_T,
\end{align*}
where $\psi_1=\int_{-\infty}^\infty g'(x)^2\mathrm{d}x$ and $\psi_2=\int_{-\infty}^\infty g(x)^2\mathrm{d}x$. Note that if $g$ is a continuous function on $[0,1]$ which is piecewise $C^1$ with a piecewise Lipschitz derivative $g'$ and satisfies $(\ref{weightPA})$, with extending $g$ to the whole real line by setting $g(x)=0$ for $x\notin[0,1]$ we obtain a weight function $g$ satisfying the condition [W]. In this case it can easily be shown that 
$n^{1/4}\left(\MRC[Y]^n_T-\widetilde{\MRC}[Y]^n_T\right)\to^p0$ 
as $n\to\infty$ under the assumptions of Theorem \ref{mainthm}, so we can also apply the asymptotic theory developed in this paper to the original estimator $\MRC[Y]^n_T$.

\section{Main result}\label{main}

\subsection{Generalization of the framework of the synchronized observation times}

We start with generalizing the framework of the grid $(T_p)$ and the synchronized observation times $(\tau^k_p)$ for a technical reason. In fact, this generalization will be useful for the localization procedure used in the proof. 

In the remainder of this section we will suppose that the sequences $(T_p)_{p=0}^\infty$ and $(\tau^k_p)_{p=0}^\infty$ $(k=1,\dots,d)$ are given \textit{a priori} and satisfies the following condition:

\begin{enumerate}[noitemsep,label={\normalfont[H]}]

\item \label{hypo:H} (i) $(T_p)$ and $(\tau^k_p)$ $(k=1,\dots,d)$ are sampling schemes.

(ii) $\tau^k_0\leq T_0$ and $T_{p-1}<\tau^k_p\leq T_p$ for any $p\geq1$ and any $k\in\{1,\dots,d\}$.

\end{enumerate}
Apparently, the sequence $(T_p)$ of the refresh times and the sequences $(\tau^k_p)$ $(k=1,\dots,d)$ of the next-ticks into $(\tau_p)$ defined in the previous section constitute one example of such sequences. 

After that, we define the quantities $N^n_t$, \eqref{defILPA} and $[Y,Y]^n_t$ based on these schemes. Then define the process $\widetilde{\MRC}[Y]^n$ by
\begin{align*}
\widetilde{\MRC}[Y]^n_t=\frac{1}{\psi_2 k_n}\sum_{i=k_n}^{N^n_t-k_n+1}\widetilde{Y}_{i,T}\left(\widetilde{Y}_{i,T}\right)^*-\frac{\psi_1}{2\psi_2k_n^2}[Y,Y]^n_t
\end{align*}
for each $t\in[0,T]$. Here, we also extend the definition of the MRC estimator to a process for the later use. Note that the summands of the first term in the right hand side of the above definition are always defined by using all the returns on $[0,T]$. We will show a functional stable central limit theorem for the process $\widetilde{\MRC}[Y]^n$ in the following.

Note that we also need to modify the construction of the stochastic basis $\mathcal{B}$ by replacing the sequence $(\mathcal{T}^n_i)$ with the increasing reordering of $\tau^k_p$'s. This is not an essential change because $\widetilde{\MRC}[Y]^n_t$ only contains variables observed at $\tau^k_p$'s.  

\begin{rmk}
Apart from the theoretical necessity, the above generalization is meaningful in terms of applications. In fact, this allows us to use the \textit{Generalized Synchronization method},  which was introduced by \citet{AFX2010}, for the data synchronization instead of the method based on refresh times. Some advantages of such a generalization are explained in Section 3.3 of \cite{AFX2010}. 
In particular, 
this generalization implies that the MRC estimator is robust to data misplacement error, as long as these misplaced data points are within the same sampling intervals $\{(T_{p-1},T_p]\}_{p=1}^\infty$. This is important in practice because it may occur that the order of consecutive ticks is not recorded correctly.
\end{rmk} 

\subsection{Conditions}

This subsection collects the regularity conditions necessary to derive our main result. In the following $\varpi$ denotes a given positive constant. 

First, we impose the following regularity conditions on the drift and the volatility processes:
\begin{enumerate}[label={\normalfont[A\arabic*]}]

\item \label{hypo:A1} For each $j\geq1$, there is an $\mathbf{F}^{(0)}$-stopping time $\rho_j$, a bounded $\mathbf{F}^{(0)}$-progressively measurable $\mathbb{R}^d$-valued process $a(j)$, and a constant $\Lambda_j$ such that 
\begin{enumerate}[nosep,label=(\roman*)]

\item $\rho_j\uparrow\infty$ as $j\to\infty$,

\item $a(\omega^{(0)})_s=a(j)(\omega^{(0)})_s$ if $s<\rho_j(\omega^{(0)})$,

\item $E\left[\|a(j)_{t_1}-a(j)_{t_2}\|^2|\mathcal{F}_{t_1\wedge t_2}\right]
\leq \Lambda_j E\left[|t_1-t_2|^\varpi|\mathcal{F}_{t_1\wedge t_2}\right]$ for any $\mathbf{F}^{(0)}$-stopping times $t_1$ and $t_2$ bounded by $j$.

\end{enumerate}

\item \label{hypo:A2} For each $j\geq1$, there is an $\mathbf{F}^{(0)}$-stopping time $\rho_j$, a bounded, c\`adl\`ag and $\mathbf{F}^{(0)}$-adapted $\mathbb{R}^d\otimes\mathbb{R}^{d'}$-valued process $\sigma(j)$, and a constant $\Lambda_j$ such that 
\begin{enumerate}[nosep,label=(\roman*)]

\item $\rho_j\uparrow\infty$ as $j\to\infty$,

\item $\sigma(\omega^{(0)})_s=\sigma(j)(\omega^{(0)})_s$ if $s<\rho_j(\omega^{(0)})$,

\item $E\left[\|\sigma(j)_{t_1}-\sigma(j)_{t_2}\|^2|\mathcal{F}_{t_1\wedge t_2}\right]
\leq \Lambda_j E\left[|t_1-t_2|^\varpi|\mathcal{F}_{t_1\wedge t_2}\right]$ for any $\mathbf{F}^{(0)}$-stopping times $t_1$ and $t_2$ bounded by $j$.

\end{enumerate}

\end{enumerate}

\begin{rmk}

\ref{hypo:A1} and \ref{hypo:A2} hold true if $a$ and $\sigma$ are It\^o semimartingales, for example, hence they are satisfied by most practical stochastic volatility models, e.g.~the Heston model. This type of continuity condition on the coefficient processes are necessary due to the irregularity of observation times as \cite{HY2011}. 
In fact, in that paper the maximum duration $r_n(t)$ of sampling times up to the time $t$ (defined in page 2419 of that paper) is only required to satisfy the condition $r_n(t)=o_p(n^{-\xi})$ for some $\xi\in(\frac{4}{5},1)$. The discussion in Section 12 of \cite{HY2011} shows that this is because they assume that the volatility process is $(\frac{1}{2}-\lambda)$-H\"older continuous for any $\lambda>0$. In this paper we assume that the quantity corresponding to $r_n(t)$ (defined in \eqref{A4}) satisfies $r_n(t)=o_p(n^{-\xi})$ for every $\xi\in(0,1)$, so we only need a weaker continuity condition than the one of \cite{HY2011}.

\end{rmk}

Second, we impose a regularity condition on the noise process. We denote by $\Upsilon$ the covariance matrix process of the noise process, i.e.~$\Upsilon_t(\cdot)=\int zz^*Q_t(\cdot,\mathrm{d}z)$.
\begin{enumerate}[label={\normalfont[A3]}]

\item \label{hypo:A3} There is a constant $\Gamma>4$ and a sequence $(\rho_{j})_{j\geq1}$ of $\mathbf{F}^{(0)}$-stopping times increasing to infinity such that 
\[
\sup_{\omega^{(0)}\in\Omega^{(0)},t<\rho_{j}(\omega^{(0)})}\int \|z\|^\Gamma Q_t(\omega^{(0)},\mathrm{d}z)<\infty.
\]
Moreover, for each $j$ there is a bounded c\`adl\`ag $\mathbf{F}^{(0)}$-adapted $\mathbb{R}^d\otimes\mathbb{R}^d$-valued process $\Upsilon(j)_t$ and a constant $\Lambda_j$ such that
\begin{enumerate}[nosep,label=(\roman*)]

\item  $\Upsilon(j)(\omega^{(0)})_t=\Upsilon(\omega^{(0)})_t$ if $t<\rho_j(\omega^{(0)})$,

\item $E\left[\|\Upsilon(j)_{t_1}-\Upsilon(j)_{t_2}\|^2|\mathcal{F}_{t_1\wedge t_2}\right]
\leq \Lambda_j E\left[|t_1-t_2|^\varpi|\mathcal{F}_{t_1\wedge t_2}\right]$ for any $\mathbf{F}^{(0)}$-stopping times $t_1$ and $t_2$ bounded by $j$.

\end{enumerate}

\end{enumerate}

\begin{rmk}

The local boundedness of the moment process is necessary for verifying a Lyapunov-type condition and the negligibility of edge effects. The continuity of the covariance matrix process is necessary due to the same reason as for \ref{hypo:A2}.

\end{rmk}

Third, we impose the following condition on the grid and the synchronized observation times:
\begin{enumerate}[label={\normalfont[A4]}]

\item \label{hypo:A4} It holds that
\begin{equation}\label{A4}
r_n(t):=\sup_{p\geq0}(T_p\wedge t-T_{p-1}\wedge t)=o_p(n^{-\xi})
\end{equation}
as $n\to\infty$ (note that $T_{-1}=0$ by convention) for every $t>0$ and every $\xi\in(0,1)$. Moreover, for each $n$ we have an $\mathbf{F}^{(0)}$-optional positive-valued process $G^n_t$, an $\mathbf{F}^{(0)}$-optional $[0,1]^d\otimes[0,1]^d$-valued process $\chi^n_t=(\chi^{n,kl}_t)_{1\leq k,l\leq d}$ and a random subset $\mathcal{N}^n$ of $\mathbb{Z}_+$ satisfying the following conditions:

\begin{enumerate}[nosep,label={\normalfont(\roman*)},ref={\normalfont\ref{hypo:A4}(\roman*)}]

\item \label{hypo:A4i} $\{(\omega,p)\in\Omega\times\mathbb{Z}_+:p\in\mathcal{N}^n(\omega)\}$ is a measurable set of $\Omega\times\mathbb{Z}_+$. Moreover, there is a constant $\kappa\in(0,\frac{1}{2})$ such that $\#(\mathcal{N}^n\cap\{p:T_p\leq t\})=O_p(n^\kappa)$ as $n\to\infty$ for every $t>0$.

\item \label{hypo:A4ii} $E[n(T_{p+1}-T_p)\big|\mathcal{F}^{(0)}_{T_p}]=G^n_{T_p}$ and $E[1_{\{\tau^k_{p+1}=\tau^l_{p+1}\}}|\mathcal{F}^{(0)}_{T_p}]=\chi^{n,kl}_{T_p}$ for every $n$, every $\mathbb{Z}_+\setminus\mathcal{N}^n$ and any $k,l=1,\dots,d$.

\item \label{hypo:A4iii} There is a c\`adl\`ag $\mathbf{F}^{(0)}$-adapted positive valued process $G$ such that $\sup_{0\leq t\leq T}|G^n_t-G_t|=O_p(n^{-\varpi})$ as $n\to\infty$. Moreover, $G_{t-}>0$ for all $t>0$.

\item \label{hypo:A4iv} There is a c\`adl\`ag $(\mathcal{F}^{(0)}_t)$-adapted $[0,1]^d\otimes[0,1]^d$-valued process $\chi$ such that $\sup_{0\leq t\leq T}\|\chi^n_t-\chi_t\|=O_p(n^{-\varpi})$ as $n\to\infty$. 

\item \label{hypo:A4v} For each $j\geq1$ there is a c\`adl\`ag $\mathbf{F}^{(0)}$-adapted positive-valued process $G(j)$, a c\`adl\`ag $\mathbf{F}^{(0)}$-adapted $[0,1]^d\otimes[0,1]^d$-valued process $\chi(j)$, an $\mathbf{F}^{(0)}$-stopping time $\rho_j$, and a constant $\Lambda_j$ such that $\rho_j\uparrow\infty$ as $j\to\infty$ and $G(\omega^{(0)})_t=G(j)(\omega^{(0)})_t,\chi(\omega^{(0)})_t=\chi(j)(\omega^{(0)})_t$ if $t<\rho_j(\omega^{(0)})$ and 
\begin{align*}
E\left[\|G(j)_{t_1}-G(j)_{t_2}\|^2+\|\chi(j)_{t_1}-\chi(j)_{t_2}\|^2|\mathcal{F}_{t_1\wedge t_2}\right]
\leq \Lambda_j E\left[|t_1-t_2|^\varpi|\mathcal{F}_{t_1\wedge t_2}\right]
\end{align*}
for every $j$ and any $\mathbf{F}^{(0)}$-stopping times $t_1$ and $t_2$ bounded by $j$.

\end{enumerate}

\end{enumerate}

\begin{rmk}\label{rmk:A4}

(i) [A4] is motivated by multiplicative error modeling of durations, which is widely used in financial econometrics (cf.~\citet{Hautsch2012}). Namely, the sequence $D_p=T_{p}-T_{p-1}$ of durations is often modeled as $D_p=\Psi_p\eta_p$, where $\Psi_p=E[D_p|\mathcal{F}_{T_{p-1}}]$, $p=1,2,\dots$ are the conditional (expected) durations. Especially, we have $E[\eta_p]=1$, hence the process $\Psi_p$ controls the frequency of the sampling times $T_p$. Consequently, it is natural to assume an \ref{hypo:A4iii} type condition in our context, which asserts that the scaled conditional durations $G^n_{T_p}=n\Psi_p$ converges to some process $G$, ensuring the existence of the asymptotic covariance matrix of our estimator. 
We also remark that conditions like \eqref{A4} and \ref{hypo:A4iii} are widely used in studies of irregular observations in our context; see \cite{BNHLS2011}, \cite{Koike2014phy} and Chapter 14 of \cite{JP2012} for instance.


\noindent (ii) Condition \ref{hypo:A4iv} on the limiting behavior of the functional $\chi^n$ is required to deal with the ($\mathcal{F}^{(0)}$-conditional) covariance between $\epsilon^k_{\tau^k_p}$ and $\epsilon^l_{\tau^l_p}$, which is given by $\Upsilon^{kl}_{\tau^k_p}1_{\{\tau^k_p=\tau^l_p\}}$. This type of condition also appears in \citet{BM2014} due to the same reason as ours (see Assumption 3.2 (iii)-(iv) of \cite{BM2014}). Note that $\chi^{n,kl}_s\equiv1$ in the synchronous case and $\chi^{n,kl}_s\equiv1_{\{k=l\}}$ in the completely non-synchronous case, so this condition is satisfied in these two cases.

\noindent (iii) The continuity condition \ref{hypo:A4v} imposed on the limiting processes are necessary for proving that we can ignore the impact of the time endogeneity on the asymptotic distribution of the estimator. Note that this condition itself does not rule out any kind of time endogeneity.


\noindent (iv) The set $\mathcal{N}^n$ represents an exceptional set of indices for which the equations in condition \ref{hypo:A4ii} are invalid. Introducing this type of set is useful to ensure the stability of the condition under the localization procedure used in the proof; see Lemma \ref{HJYlem4.1}. It also allows the existence of outliers in the durations. For example, we can consider the situation where $T_p=\log n/n$ if $p\leq n^\kappa$ and $T_p=1/n$ otherwise. 

\noindent (v) \ref{hypo:A4} implies that $N^n_T/n$ converges to a non-zero random variable in probability (see Lemma \ref{HJYlem2.2}). In particular, this condition connects the number of (synchronized) observations with the parameter $n$ to drive our asymptotic theory.
 
\end{rmk}

To illustrate \ref{hypo:A4}, we give two simple but commonly used examples satisfying \ref{hypo:A4} when we consider the case that $(T_p)$ is defined as the refresh times of $\{(t^k_i)\}_{k=1}^d$ and $(\tau^1_p),\dots,(\tau^d_p)$ are defined as the next-tick interpolations to $(T_p)$ as in the previous section.
\begin{example}[Poisson sampling]
Let $(t^k_i)$ be a sequence of Poisson arrival times with the intensity $np_k$ for each $k$ and suppose that $(t^1_i),\dots,(t^d_i)$ are mutually independent and independent of $Y$. Then \ref{hypo:A4} is satisfied with $\mathcal{N}^n$ being empty. In fact, it is easy to show that \ref{hypo:A4iv} holds true with $\chi_t$ being the identity matrix of order $d$, while \eqref{A4} follows from Corollary 1 of \cite{RT1973}. \ref{hypo:A4iii} is satisfied with
\begin{equation}\label{poissonG}
G_s=\sum_{k=1}^d\sum_{1\leq l_1<\cdots<l_k\leq d}\frac{(-1)^{k-1}}{p_{l_1}+\cdots p_{l_k}}.
\end{equation}
This can be proven as follows. Set $p=\sum_{k=1}^dp_k$ and let $\widetilde{N}$ be a Poisson process with the intensity $np$. Let $(\eta_j)_{j=1}^\infty$ be a sequence of i.i.d.~random variables such that $P(\eta_j=k)=p_k/p$, $k=1,\dots,d$. We assume that $(\eta_j)$ is independent of $\widetilde{N}$. For each $k\in\{1,\dots,d\}$ define the process $N^{(k)}$ by $N^{(k)}_t=\sum_{j=1}^{\widetilde{N}_t}1_{\{\eta_j=k\}}$. A short calculation shows that $N^{(k)}$ is a Poisson process with the intensity $np_k$. Therefore, Theorem 6 of \cite{CA1968} implies that $N^{(1)},\dots,N^{(d)}$ are independent. This fact yields $E[n(T_{p+1}-T_p)|\mathcal{F}^{(0)}_{T_{p}}]=p^{-1}E[\min\{j:\{\eta_1,\dots,\eta_j\}=\{1,\dots,d\}\}]$. Now \eqref{poissonG} follows from Eq.(6) of \cite{VonSchelling1954}. \ref{hypo:A4v} is then obvious.

\end{example}

\begin{example}[Times generated by hitting barriers]\label{Exhit}
Let us focus on the univariate case, i.e.~$d=d'=1$. Note that in this case we have $T_i=t^1_i$. Then, a common example of endogenous observation times is a class of stopping times generated by hitting times (cf.~Section 4.4 of \cite{Fu2010b} and Example 4 of \cite{LMRZZ2014}).  
Specifically, suppose that $\sigma^2_t$ is continuous and bounded away from 0 and define
\begin{equation}\label{defhit}
t^1_0=0,\qquad t^1_{i+1}=\inf\left\{t>t^1_{i}:M_t-M_{t^1_i}=-\alpha/\sqrt{n}\textrm{ or }M_t-M_{t^1_i}= \beta/\sqrt{n}\right\}
\end{equation}
for positive constants $\alpha,\beta$, where $M_t=\int_0^t\sigma_s\mathrm{d}W_s$. This observation scheme satisfies \ref{hypo:A4} with $\mathcal{N}^n$ being empty. In fact, using a representation of a continuous local martingale with Brownian motion, we have
\begin{align*}
P\left( M_{t^1_{i+1}}-M_{t^1_i}=-\alpha/\sqrt{n}\big|\mathcal{F}^{(0)}_{t^1_i}\right)=\beta/(\alpha+\beta),\qquad
P\left( M_{t^1_{i+1}}-M_{t^1_i}=\beta/\sqrt{n}\big|\mathcal{F}^{(0)}_{t^1_i}\right)=\alpha/(\alpha+\beta).
\end{align*}
Especially, it holds that $\sup_iE[|\sqrt{n}(M_{t^1_{i+1}}-M_{t^1_i})|^r]<\infty$ for any $r>0$. Therefore, an analogous argument to the proof of Proposition 2.1 from \cite{Obloj2004} yields the following result: for each $r\geq1$ there exists a positive constant $C_r$ such that
$E[|\int_{t^1_i}^{t^1_{i+1}}\sigma^2_s\mathrm{d}s|^r]\leq C_r n^{-r}$
for every $n,i$. 
In particular, this inequality yields \eqref{A4} because $\sigma^2_t$ is bounded away from 0. Moreover, noting that $E\left[(M_{t^1_{i+1}}-M_{t^1_i})^2|\mathcal{F}^{(0)}_{t^1_i}\right]=\sigma^2_{t^1_i}E\left[t^1_{i+1}-t^1_{i}|\mathcal{F}^{(0)}_{t^1_i}\right]+o_p(n^{-1})$ as $n\to\infty$ uniformly in $i\leq N^n_T$ because of the continuity of $\sigma$, we also obtain \ref{hypo:A4iii} with $G_t=\alpha\beta/\Sigma_t$. \ref{hypo:A4iv}--(v) are obvious.

\end{example}

We further discuss about \ref{hypo:A4} in Section \ref{discussion:A4}.

\subsection{Result}

The statement of our main theorem requires the notion of \textit{stable convergence}, which is common in this area. We however need to note that in our case the stochastic basis $\mathcal{B}$, which supports our observation data, changes as $n$ varies, hence the common definition of stable convergence used in the literature (cf.~Definition 1 of \cite{PV2010}) needs to be modified here. This has been done in page 47 of \cite{JP2012} as follows. Let $(\mathcal{X},\mathcal{A},\mathbb{P})$ be a probability space and assume that we have a random element $Z_n$ taking values in a Polish space $S$ and defined on an extension $(\mathcal{X}_n,\mathcal{A}_n,\mathbb{P}_n)$ of $(\mathcal{X},\mathcal{A},\mathbb{P})$ for each $n\in\mathbb{N}\cup\{\infty\}$. In this setup the sequence $Z_n$ is said to \textit{converge stably in law} to $Z_\infty$ if $\mathbb{E}_n[Uf(Z_n)]\rightarrow \mathbb{E}_\infty[Uf(Z_\infty)]$ for any $\mathcal{A}$-measurable bounded random variable $U$ and any bounded continuous function $f$ on $S$. Then we write $Z_n\to^{d_s}Z$. The most important property of this notion is the following: For each $n\in\mathbb{N}$, let $V_n$ be a real-valued variable on $(\mathcal{X}_n,\mathcal{A}_n,\mathbb{P}_n)$, and suppose that the sequence $V_n$ converges in probability to a variable $V$ on $(\mathcal{X},\mathcal{A},\mathbb{P})$. Then we have $(Z_n,V_n)\to^{d_s}(Z_\infty,V)$ for the product topology on the space $S\times\mathbb{R}$, provided that $Z_n\to^{d_s}Z$.

Now we are ready to state the main theorem of this paper.
\begin{theorem}\label{mainthm}

Suppose that \ref{hypo:W}, \ref{hypo:H} and \ref{hypo:A1}--\ref{hypo:A4} are satisfied. Then
\begin{equation}\label{CLT}
n^{1/4}\left(\widetilde{\MRC}[Y]^n-[X,X]\right)\to^{d_s}\mathcal{W}\qquad\mathrm{in}\ \mathbb{D}_T^{d\times d}
\end{equation}
as $n\to\infty$, where $\mathcal{W}$ is an $\mathbb{R}^d\otimes\mathbb{R}^d$-valued continuous process defined on an extension of $\mathcal{B}^{(0)}$, which is conditionally on $\mathcal{F}^{(0)}$ centered Gaussian with independent increments, and with conditional covariances
\begin{align}
\widetilde{E}\left[\mathcal{W}^{kl}_t\mathcal{W}^{k'l'}_t|\mathcal{F}^{(0)}\right]&=\frac{2}{\psi_2^{2}}\int_0^t\left[\Phi_{22}\theta\left\{\Sigma^{kk'}_s\Sigma^{ll'}_s+\Sigma^{kl'}_s\Sigma^{lk'}_s\right\}G_s
+\frac{\Phi_{11}}{\theta^3}\left\{\widetilde{\Upsilon}^{kk'}_s\widetilde{\Upsilon}^{ll'}_s+\widetilde{\Upsilon}^{kl'}_s\widetilde{\Upsilon}^{lk'}_s\right\}\frac{1}{G_s}\right.\nonumber\\
&\hphantom{=\frac{2}{\psi_2^{2}}[}\left.+\frac{\Phi_{12}}{\theta}\left\{\Sigma^{kk'}_s\widetilde{\Upsilon}^{ll'}_s+\Sigma^{lk'}_s\widetilde{\Upsilon}^{kl'}_s+\Sigma^{ll'}_s\widetilde{\Upsilon}^{kk'}_s+\Sigma^{kl'}_s\widetilde{\Upsilon}^{lk'}_s\right\}\right]\mathrm{d}s\label{avar}
\end{align}
for $k,l,k,l'=1,\dots,d$ and $t\in\mathbb{R}_+$. Here, $\widetilde{\Upsilon}$ is the $\mathbb{R}^d\otimes\mathbb{R}^d$-valued process defined by $\widetilde{\Upsilon}^{kl}_s=\Upsilon^{kl}_s\chi^{kl}_s$, and 
\begin{gather*}
\Phi_{22}=\int_{0}^{\infty}\phi_{g,g}(y)^2\mathrm{d}y,\qquad
\Phi_{12}=\int_{0}^{\infty}\phi_{g,g}(y)\phi_{g',g'}(y)\mathrm{d}y,\qquad
\Phi_{11}=\int_{0}^{\infty}\phi_{g',g'}(y)^2\mathrm{d}y
\end{gather*}
with $\phi_{u,v}$ being the function on $\mathbb{R}$ defined by $\phi_{u,v}(y)=\int_{-\infty}^\infty u(x-y)v(x)\mathrm{d}x$.

\end{theorem}

\begin{rmk}\label{rmk:mainthm}

(i) The above theorem tells us that under our assumptions the observation times affect the asymptotic distribution of the MRC estimator only through the asymptotic conditional duration process $G$ and the limiting process $\chi$ measuring the degree of the non-synchronicity. In particular, the time endogeneity has no impact on the asymptotic distribution. This contrasts with the case of the realized volatility in a pure diffusion setting, where the time endogeneity can cause a bias in the asymptotic distribution as demonstrated in \cite{Fu2010b} and \cite{LMRZZ2014}. 

\noindent (ii) It is also worth pointing out that the effect of the observation times is not through the \textit{Asymptotic Quadratic Variation of Time}, unlike the case of the realized volatility as described in \cite{MZ2009} for instance. Especially, even the randomness of the durations plays no role in the asymptotic distribution of the MRC estimator in the current setup. This is again different from the case of the realized volatility, where the randomness of the durations inflates the asymptotic variance. 

\noindent (iii) Our result further suggests that the interpolation errors caused by the synchronization does not matter in the first order approximation of the estimator, which has already been pointed out in Section 3.6 of \cite{CKP2010}. This is also different from the case of the Hayashi-Yoshida estimator in a pure diffusion setting: See Section 3.2 of \cite{Bibinger2012} for details. We mention that the treatment of the time endogeneity for the Hayashi-Yoshida estimator is much more complex than ours. Recently \cite{PM2015} have dealt with this topic in a pure diffusion setting. \cite{RR2012} discuss a related topic in a setting with microstructure noise modeled by the concept of \textit{uncertainty zones}. More precisely, in their model the observations of the latent process can be estimated and they show that the Hayashi-Yoshida estimator based on these estimated observations consistently estimates the quadratic covariation. However, its asymptotic distribution is not known so far.

\noindent (iv) Here we should note that our result \textit{does not} imply that the randomness, the endogeneity and the non-synchronicity of observation times play no role in the limit of our statistical experiments. Investigating this topic apparently requires more sophisticated arguments and is beyond the scope of this paper. We only refer to the recent work of \citet{Ogihara2014noise}, which has developed the LAN property for non-synchronously observed (Gaussian) diffusion processes with noise when observation times are random but independent of the observed processes. This work has also found that the observation times affect the Fisher information only through their spot intensity process, which corresponds to the process $1/G$ in our case if the observations are synchronous.  

\noindent (v) We further remark that our condition \ref{hypo:A4} plays a crucial role to reduce the effects of the randomness of observation times. In fact, the recent work of \citet{BM2014} has pointed out the role of the \textit{long-run variation of time} in the asymptotic distribution of the (generalized) multi-scale estimator of (\citealt{Z2006,B2011a}). The well-known relation between pre-averaging and multi-scale estimators (cf.~Section 3.5 of \cite{CKP2010} and Section 2.2 of \cite{BM2014}) suggests that this would also be the case in our setting. Indeed, \ref{hypo:A4} characterizes the asymptotic long-run variation of time in terms of $G$; See Proposition \ref{LRVT}.
\end{rmk}

\begin{rmk}
In Example \ref{Exhit}, the stable convergence result of Theorem \ref{mainthm} still holds true when we replace $M$ in $(\ref{defhit})$ by $X$. This can be shown as follows. Define the process $Z$ by $Z_t=\exp\left(\int_0^t a_s/\sigma_s\mathrm{d}W_s-\frac{1}{2}\int_0^t a_s^2/\sigma^2_s\mathrm{d}s\right)$ for each $t\geq0$.
As is well known, $Z_t$ is a positive continuous local martingale. Therefore, by a localization argument we may assume that both $Z$ and $1/Z$ are bounded. In particular, $Z$ is a martingale, so we can define a probability measure $\widetilde{P}^{(0)}_T$ on $(\Omega^{(0)},\mathcal{F}^{(0)})$ by $\widetilde{P}^{(0)}_T(E)=P^{(0)}(1_E Z_T)$. $\widetilde{P}^{(0)}_T$ is obviously equivalent to the probability measure $P^{(0)}$. Set $W'_t=W_t-\int_0^t a_s/\sigma_s\mathrm{d}s$ for each $t$. Then, by the Girsanov theorem $(W'_t)_{0\leq t\leq T}$ is a standard Wiener process on $(\Omega^{(0)},\mathcal{F}^{(0)},(\mathcal{F}^{(0)}_t)_{0\leq t\leq T},\widetilde{P}^{(0)}_T)$ and it holds that $X_t=\int_0^t\sigma_s\mathrm{d}W'_s$. Hence \ref{hypo:A4} holds true under $\widetilde{P}^{(0)}_T$. Moreover, \ref{hypo:A1}--\ref{hypo:A3} are obviously satisfied under $\widetilde{P}^{(0)}_T$. Therefore, \eqref{CLT} holds true under $\widetilde{P}^{(0)}_T$. Since the stable convergence is stable under equivalent changes of probability measures, \eqref{CLT} also holds true under the original probability measure $P^{(0)}$. It is worth mentioning that the continuity condition on the drift $a$ is unnecessary in this case.
\end{rmk}

\begin{rmk}[Feasible limit theorem]
The stable convergence \eqref{CLT} allows us to consider Studentization of the MRC estimator, provided that some consistent estimators for the asymptotic conditional covariances \eqref{avar} are available. Such estimators can be constructed by a kernel-based approach as in Section 4.3 of \cite{Koike2015pthy}, for example. It would also be possible to apply other approaches such as histogram-type estimators of \cite{Bibinger2012,BM2014} or a subsampling method of \cite{CPV2013} to our case.
\end{rmk}

\begin{rmk}[Serially dependent noise]
The MRC estimator is inconsistent if the error process is serially dependent (see Lemma 1 of \citet{HP2012}). This is because the bias correction term $(\psi_1/2\psi_2k_n^2)[Y,Y]^n_T$ does not correct the bias in the presence of such serial dependence. In fact, if the serial dependence is sufficiently weak, the bias is proportional to the long-run covariance matrix of the noise. So, if the bias is correctly adjusted, the MRC estimator is still consistent, and it would even enjoy a central limit theorem where the asymptotic variance would be the same as \eqref{avar} except that the covariance matrix $\Upsilon_t$ of the noise would change to the long-run covariance matrix (see also Theorem 1 of \cite{HP2012}).    
\end{rmk}

\section{Discussion about the assumption on observation times}\label{discussion:A4}

\subsection{Connection with the tricity}

Let us focus on the univariate case (so we have $T_p=t^1_p$). One striking feature of the time endogeneity in a pure diffusion setting is that the (scaled) cubic power variation, or the \textit{tricity}
\[
\sqrt{n}\sum_{p=1}^{N^n_t}(X_{T_p}-X_{T_{p-1}})^3
\]
plays an important role in the asymptotic theory of the realized volatility. This is natural in a sense because the time endogeneity possibly causes the skewness of the returns $(X_{T_{p+1}}-X_{T_{p}})_{p\in\mathbb{Z}_+}$ even if the volatility process $\sigma$ is deterministic; see Example \ref{Exhit} for instance. More generally, for a given one-dimensional Wiener process $W$ and for any probability measure $\mu$ on $\mathbb{R}$ such that $\int x\mu(\mathrm{d}x)=0$, we can find a sequence $(S_p)_{p\in\mathbb{Z}_+}$ of stopping times such that $W_{S_{p+1}}-W_{S_p}\overset{i.i.d.}{\sim}\mu$ (cf.~Example 5 of \cite{LMRZZ2014}). 

On the other hand, our condition \ref{hypo:A4} makes the tricity of the pre-averaged data asymptotically negligible:
\begin{prop}\label{proplzz}
Under the assumptions of Theorem \ref{mainthm}, it holds that
\begin{align*}
\frac{\sqrt{n}}{k_n^{3/2}}\sum_{i=k_n}^{N^n_t-k_n+1}(\widetilde{X}_{i,T})^3\to^p0
\end{align*}
as $n\to\infty$ for any $t\in[0,T]$ (recall that $\widetilde{X}_{i,T}$ is defined by \eqref{defILPA}).
\end{prop}
This result gives some intuition of why the time endogeneity is less important in a noisy diffusion setting. Indeed, it can be shown that the pre-averaged data is asymptotically centered Gaussian in some sense; see Lemma 6.7 of \cite{Koike2014jclt}.

\subsection{Connection with the long-run variation of time}

As was stated in Remark \ref{rmk:mainthm}(iv), \citet{BM2014} have introduced the functional 
\[
\mathfrak{S}_{n,m}(t)=\frac{n}{m}\sum_{p=1}^{N^n_t}(T_p-T_{p-1})\sum_{q=1}^{m\wedge p}(T_{p-q+1}-T_{p-q}),\qquad
t\in\mathbb{R}_+,~m=1,2,\dots
\]
to derive a central limit theorem for the generalized multi-scale estimator. Our assumption on observation times characterizes the limiting process of this functional as follows:
\begin{prop}\label{LRVT}
Under \ref{hypo:A4}, suppose further that $\#(\mathcal{N}^n\cap\{p:T_p\leq t\})=O_p(1)$ as $n\to\infty$ for every $t>0$. Then $\mathfrak{S}_{n,m}(t)\to\int_0^tG_s\mathrm{d}s$ as $n\to\infty$ for every $t$, provided that $m\to\infty$ and $m=o(n)$.
\end{prop}

\section{Optimal weight function and related topics}\label{section:optimal}

\subsection{Optimal weight function}\label{optimal}

We turn to question (b). Noting that $\phi''_{g,g}=-\phi_{g',g'}$, in the univariate and equidistant sampling case our estimator has the same asymptotic variance as that of the flat-top realized kernel with the kernel function $\phi_{g,g}$ and the bandwidth $k_n$. Here, the flat-top realized kernel with the kernel function $K$ and the bandwidth $H$ is defined by
\begin{equation*}
RK(Y)=\gamma_0(Y)+\sum_{h=1}^{N^n_T-1}K\left(\frac{h-1}{H}\right)\left\{\gamma_h(Y)+\gamma_{-h}(Y)\right\},\qquad\gamma_h(Y)=\sum_{j=h+1}^{N^n_T}\Delta_{j}Y\Delta_{j-h}Y.
\end{equation*}
According to Proposition 1 of \cite{BNHLS2008}, in the parametric setting, i.e.~both $\sigma$ and $\Upsilon$ are constant, the asymptotic variance of $RK(Y)$ is minimized by the kernel $K_{\mathrm{opt}}(x)=(1+x)e^{-x}$ with the oracle bandwidth $H=(\sqrt{\Upsilon}/\sigma)\sqrt{N^n_T}$. Therefore, if there exists a function $g$ on $\mathbb{R}$ satisfying [W] and $\phi_{g,g}=K_{\mathrm{opt}}$, such a function $g$ is an optimal weight function. Fortunately, we can find such a $g$ by a simple Fourier analysis and it is given by $g(x)=e^{-|x|}$. In other words, the (twice) double exponential density function is an optimal weight function for our estimator. In this case our estimator achieves the parametric efficiency bound $8\sigma^3\sqrt{\Upsilon}$ of the asymptotic variance from \cite{GJ2001a} with the oracle tuning parameter $k_n=(\sqrt{\Upsilon}/\sigma)\sqrt{N^n_T}$.

Despite its efficiency, the optimal kernel $K_{\mathrm{opt}}$ is not preferable in practice due to its computational disadvantage. That is, since the support of $K_{\mathrm{opt}}$ is unbounded, it requires $n$ (all) realized autocovariances $\gamma_h(Y)$ to be computed. As a consequence, the order of the computation for $RK(Y)$ becomes $O(n^2)$. In contrast, our optimal weight function has a nice feature in terms of the computation. Let us define the sequences $(y^+_p)_{p=k_n}^{N^n_T-k_n+1}$ and $(y^-_p)_{p=k_n}^{N^n_T-k_n+1}$ recursively by $y^+_{k_n}=\widetilde{\Delta}_{k_n}Y$, $y^-_{k_n}=\widetilde{\Delta}_{N^n_T-k_n+1}Y$ and  
\begin{equation*}
y^+_p=e^{-1/k_n}y^+_{p-1}+\widetilde{\Delta}_{p}Y,\qquad
y^-_p=e^{-1/k_n}y^-_{p-1}+\widetilde{\Delta}_{N^n_T-p+1}Y,\qquad
p=k_n+1,\dots,N^n_T-k_n+1.
\end{equation*}
Then it can easily be seen that $\widetilde{Y}_{i,T}=y^+_i+y^-_{N^n_T-i+1}-\widetilde{\Delta}_{i}Y$, hence we can compute $(\widetilde{Y}_{i,T})_{i=k_n}^{N^n_T-k_n+1}$ with the order $O(n)$. Consequently, the order of the computation of our estimator is $O(n)$, which is, in general, even less than that of the MRC estimator with a weight function with a bounded support.

\subsection{Comparison with other approaches}

We shall compare the pre-averaging approach with two existing nonparametric volatility estimation methods which also achieve the parametric efficiency bound: the QML approach from \citet{Xiu2010} and the spectral method from \citet{Reiss2011}. In terms of implementation, the QML approach has two advantages over the others. Namely, it contains no tuning parameter and it always ensures the positivity of the estimated value. On the other hand, the spectral approach has an advantage that it is also non-parametrically asymptotically efficient in the sense that it achieves an asymptotic lower bound for estimating integrated volatilities in settings with non-constant volatilities (see \cite{Reiss2011} for details). Another advantage of the spectral approach is that it can be extended to an efficient estimator for multivariate volatility matrices in a non-synchronous observation setting (the \textit{local method of moment (LMM) estimator} from \citet{BHMR2014}). Selection of the tuning parameter $\theta$ in our estimator also has a theoretical issue. Namely, the optimal $\theta$ contains unknown parameters and it is not clear whether we may plug-in some estimated values into them. This issue can presumably be solved by modifying the estimator to an \textit{adaptive} version, which has already been done in the case that $g$ has a bounded support; see Section 7.6.2 of \citet{AJ2014} for details.

An advantage of the pre-averaging approach over these two approaches is that it enables us systematically to extend functionals of semimartingale increments in a noisy observation setting. 
It is known that such functionals serve as statistical analyses of jumps very much (cf.~Chapter 10 of \citet{AJ2014}), so the pre-averaging approach is expected to be more appropriate than the others in terms of handling jumps, and this is indeed one of the original motivations to introduce the concept of pre-averaging by \citet{PV2009}. 
In fact, it is not obvious how to handle jumps separately from diffusion parts in the QML approach. For the spectral method, a threshold method originally proposed by \citet{M2001} can be applied to separating jumps from the spectral volatility estimator, as shown by \citet{BW2015}. However, \cite{BW2015} have also shown that the spectral estimator from \cite{Reiss2011} is not a rate-optimal estimator for the \textit{entire} quadratic variation. As we will briefly see in the next subsection, the pre-averaging approach can handle the effect of jumps in volatility inferences more efficiently.

\subsection{Jumps}\label{sec:jumps}

We shall briefly discuss how much the pre-averaging procedure can improve the estimation of the quadratic variation in the presence of jumps. Specifically, we assume that our observations are generated by the process $Z_t=Y_t+J_t$ instead of $Y_t$, where $J$ is a c\`adl\`ag process defined on $\mathcal{B}^{(0)}$ and of the form $J_t=\sum_{k=1}^{L_t}\Delta J_{S_k}$ with $L_t$ being a point process with the jump times $S_1<S_2<\cdots$. Moreover, for the sake of brevity, we concentrate our attention on the following simplified situation: $d=d'=1$, $T=1$, $t^1_i=i/n$, $a_s\equiv0$, $\sigma$ and $\Upsilon$ are constants, $\epsilon\overset{i.i.d.}{\sim}\mathbf{N}(0,\Upsilon)$. 

To indicate the dependence of quantities on the weight function $g$ explicitly, in the following we will write $\widetilde{Z}(g)_{i,1}$ instead of $\widetilde{Z}_{i,1}$, for example. We introduce threshold pre-averaging estimators for the (squared) volatility $\sigma^2$ and the sum $\sum_{k=1}^{L_1}(\Delta J_{S_k})^2$ of the squared jumps as follows:
\begin{align*}
\left\{\begin{array}{l}
\widehat{IV}_n(g,\rho_n)=\frac{1}{\psi(g)_2 k_n}\sum_{i=k_n}^{n-k_n+1}\left(\widetilde{Z}(g)_{i,1}\right)^21_{\left\{\left|\widetilde{Z}(g)_{i,1}\right|\leq\rho_n\right\}}-\frac{\psi(g)_1}{\psi(g)_2k_n^2}[Z,Z]^n_1,\\
\widehat{JV}_n(g,\rho_n)=\frac{1}{\psi(g)_2 k_n}\sum_{i=k_n}^{n-k_n+1}\left(\widetilde{Z}(g)_{i,1}\right)^21_{\left\{\left|\widetilde{Z}(g)_{i,1}\right|>\rho_n\right\}},
\end{array}\right.
\end{align*}
where $\rho_n$ is a sequence of positive numbers tending to 0 as $n\to\infty$. Then we obtain the following result:
\begin{prop}\label{jclt}
In addition to the above assumptions, suppose that $g_1,g_2$ satisfy \ref{hypo:W}, $\rho_{n}=cn^{-w}$ for some $c>0$ and $w\in(\frac{1}{8},\frac{1}{4})$, $P(S_1=0)=P(S_{L_1}=1)=0$ and $(S_k)_{k\geq1}$ is independent of $W$. Then
$n^{1/4}(\widehat{IV}_n(g_1,\rho_{n})-\sigma^2,\widehat{JV}_n(g_2,\rho_{n})-\sum_{k=1}^{L_1}(\Delta J_{S_k})^2)\to^{d_s}(v_C(g_1,\theta)\zeta_C,v_J(g_2.\theta)\zeta_J)$ as $n\to\infty$, where $\zeta_C$ and $\zeta_J$ are mutually independent standard normal variables which are defined on an extension of $\mathcal{B}^{(0)}$ and independent of $\mathcal{F}^{(0)}$, and
\[
\left\{\begin{array}{l}
v_C(g,\theta)^2=\frac{4}{\psi(g)_2^2}\left(\Phi(g)_{22}\theta\sigma^4+2\frac{\Phi(g)_{12}}{\theta}\sigma^2\Upsilon+\frac{\Phi(g)_{11}}{\theta^3}\Upsilon^2\right),\\
v_J(g,\theta)^2=\frac{8}{\psi(g)_2^2}\left(\Phi(g)_{22}\theta\sigma^2+\frac{\Phi(g)_{12}}{\theta}\Upsilon\right)\sum_{k=1}^{L_1}(\Delta J_{S_k})^2.
\end{array}\right.
\]
\end{prop}
Note that, for the case that $g$ has a bounded support, central limit theorems for the MRC estimator have been derived in fairly general settings by (\citealt{JPV2010,Koike2014jclt}), and the derivation of Proposition \ref{jclt} is pursued completely analogous to these papers.

From Proposition \ref{jclt} our adjusted MRC estimator is also a rate-optimal estimator for the entire quadratic variation. However, in terms of efficiency it is better to use different weight functions between the estimation of the continuous and the jump parts. This is because the optimal choices of $\theta$ for minimizing $v_C(g,\theta)$ and $v_J(g,\theta)$ do not coincide for any $g$ satisfying \ref{hypo:W}. Namely, the estimator $\widehat{QV}_n(g_1,g_2,\rho_{n}):=\widehat{IV}_n(g_1,\rho_{n})+\widehat{JV}_n(g_2,\rho_{n})$ could be a more efficient estimator for the quadratic variation than usual MRC estimators in the presence of jumps. For example, if we set $g_1(x)=e^{-|x|}$, then $g_2(x)=e^{-\sqrt{5}|x|}$ makes the optimal choices of $\theta$ for minimizing $v_C(g_1,\theta)$ and $v_J(g_2,\theta)$ coincide. In this case the minimum value of $v_J(g_2,\theta)$ becomes $4\sqrt{5}\sigma\sqrt{\Upsilon}$.

\begin{rmk}[Comparison with Bibinger-Winkelmann's spectral jump estimator]
\citet{BW2015} have overcome the aforementioned problem of estimating jumps in the spectral approach by a clever adjustment which exploits a trigonometric identity. Their adjusted estimator for $\sum_{k=1}^{L_1}(\Delta J_{S_k})^2$, which is given by Eq.(16) of \cite{BW2015}, enjoys a central limit theorem with the optimal rate $n^{-1/4}$. In the current situation the asymptotic variance of this estimator is given by
\[
\Xi=2\left(\int_0^\infty\frac{\mathrm{d}z}{\left(\sigma^2+\pi^2z^2\Upsilon\right)^2}\right)^{-2}\int_0^\infty\frac{\sigma^2+4\pi^2z^2\Upsilon}{\left(\sigma^2+\pi^2z^2\Upsilon\right)^4}\mathrm{d}z\sum_{k=1}^{L_1}(\Delta J_{S_k})^2
\]
according to Theorem 2 of \cite{BW2015}. The integrals in the expression can be calculated using the formula 
\[
\int_0^\infty\frac{\mathrm{d}z}{(\sigma^2+\pi^2z^2\Upsilon)^{n+1}}=\frac{(2n-1)!!}{2\sigma^{2n+1}\sqrt{\Upsilon}(2n)!!},\qquad
n=1,2,\dots,
\]
and we obtain $\Xi=9\sigma\sqrt{\Upsilon}\sum_{k=1}^{L_1}(\Delta J_{S_k})^2$ which is slightly greater than $4\sqrt{5}\sigma\sqrt{\Upsilon}\sum_{k=1}^{L_1}(\Delta J_{S_k})^2$. So $\widehat{JV}_n(g,\rho_n)$ could be more efficient in the ideal situation where we can choose the optimal $\theta$. 
\end{rmk}

To evaluate the absolute efficiency of estimating $\sigma^2+\sum_{k=1}^{L_1}(\Delta J_{S_k})^2$, we need to derive a reasonable asymptotic lower bound for estimating this quantity. For this purpose we further simplified our model as follows:
\begin{equation}\label{jump_model}
\underline{Z}_i=\sigma W_{i/n}+\sum_{k=1}^K\gamma_k1_{\{S_k\leq i/n\}}+\epsilon_i,\qquad
i=1,\dots,n,
\end{equation}
where we assume that $\Upsilon>0$, $K\in\mathbb{N}$ and $0<S_1<\cdots<S_K<1$ are known and deterministic, and consider the problem of estimating the (deterministic) parameter $\vartheta=(\sigma,\gamma_1,\dots,\gamma_K)\in(0,\infty)\times\mathbb{R}^K$ from observations generated by \eqref{jump_model}. Note that simplification of making the number of jumps and jump times deterministic is commonly used for establishing asymptotic lower bounds for estimating jumps in the absence of noise (cf.~\citet{CDG2014} and Section 4 of \citet{LTT2014}). 
\begin{prop}\label{LAN}
For model \eqref{jump_model}, we have the LAN property at any $\vartheta$ with rate $n^{-1/4}$ and asymptotic Fisher information matrix $(2\sigma\sqrt{\Upsilon})^{-1}E_{K+1}$, where $E_{K+1}$ is the identity matrix of order $K+1$.
\end{prop}

Proposition \ref{LAN} implies that an asymptotic lower bound for estimating $\sigma^2+\sum_{k=1}^K\gamma_k^2$ is given by $8\sigma^3\sqrt{\Upsilon}+8\sigma\sqrt{\Upsilon}\sum_{k=1}^K\gamma_k^2$. In particular, the above choice of the weight function $g_1$ does not attain this bound. So the next question is whether there is a weight function $g$ satisfying \ref{hypo:W} and $v_J(g,\theta)=8\sigma\sqrt{\Upsilon}$ for some $\theta>0$. Unfortunately, however, we have the following negative result.
\begin{prop}\label{negative}
There is no function $g$ satisfying \ref{hypo:W} and $v_J(g,\theta)=8\sigma\sqrt{\Upsilon}$ for some $\theta>0$.
\end{prop}

Finally, we remark that the asymptotic lower bound $8\sigma\sqrt{\Upsilon}\sum_{k=1}^K\gamma_k^2$ for estimating $\sum_{k=1}^K\gamma_k^2$ is achievable if we know $\sigma$ in addition to $\Upsilon$ and $S_k$'s:
\begin{prop}\label{MLE}
Consider the vector $\mathbf{z}_n:=(\underline{Z}_1,\dots,\underline{Z}_n)^*$ of observations generated from model \eqref{jump_model}. Let $\widehat{\gamma}^n$ be the $K$-dimensional random vector whose $k$-th component is equal to the $\lceil nS_k\rceil$-component of $2\sigma\sqrt{\Upsilon}n^{-\frac{1}{2}}V_n(\sigma)^{-1}D_n\mathbf{z}_n$, where the $n\times n$ matrices $V_n(\sigma)$ and $D_n$ are defined by \eqref{eq:matrices}. Then we have $n^{1/4}(\widehat{\gamma}^n-\gamma)\xrightarrow{d}\mathbf{N}(0,2\sigma\sqrt{\Upsilon}E_K)$ as $n\to\infty$, where $\gamma=(\gamma_1,\dots,\gamma_K)^*$. 
\end{prop}

\begin{rmk}
Although the estimator $\widehat{\gamma}^n$ constructed in Proposition \ref{MLE} is infeasible in practice because $\sigma$, $\Upsilon$, \dots are usually unknown, it is interesting in the sense that the form of the estimator suggests that a feasible efficient estimator might be obtained by plugging appropriately estimated values in unknown parameters. We leave this topic to future research. 
\end{rmk}




\section{Proofs}\label{secproofs}


\subsection{Asymptotic behavior of $N^n_t$}

The aim of this subsection is to prove the following result:
\begin{lem}\label{HJYlem2.2}
\ref{hypo:A4} implies that $N^n_t/n\to^p\int_0^t1/G_s\mathrm{d}s$ as $n\to\infty$ for every $t$.
\end{lem}
To prove this result, we introduce some preliminary results which we will also use later. Throughout the section, we fix constants $\gamma>0$ and $\xi\in(0,1)$ such that
\begin{equation}\label{est.xi}
\xi>\left(\frac{3}{4}+\gamma\right)\vee\left\{\frac{3}{4}+\frac{1}{\Gamma}+\gamma\left(1-\frac{2}{\Gamma}\right)\right\}\vee\left(\frac{4+\varpi}{4+2\varpi}+\gamma\right)\vee\left(\kappa+\frac{1}{2}+2\gamma\right)\vee(1-\varpi+\gamma),
\end{equation}
and set $\bar{r}_n=n^{-\xi}$ and $d_n=\lceil n^{-\gamma}\rceil$.


First, we remark the following result, which is more or less known and repeatedly used throughout the section:
\begin{lem}\label{useful}
Consider a sequence $(\mathcal{I}^n_j)_{j\in\mathbb{Z}_+}$ of filtrations and a sequence $(\zeta^n_j)_{j\in\mathbb{N}}$ of random variables adapted to the filtration $(\mathcal{I}^n_j)$ for each $n$. Let $\mathbb{T}$ be a non-empty set and suppose that a non-negative integer-valued variable $N^n(t)$ is given for each $n\in\mathbb{N}$ and each $t\in\mathbb{T}$. Suppose also that there is an element $t_0\in\mathbb{T}$ such that $N^n(t_0)$ is an $(\mathcal{I}^n_j)$-stopping time and $N^n(t)\leq N^n(t_0)$ for all $t\in\mathbb{T}$. If $\sum_{j=1}^{N^n(t_0)}E\left[\left|\zeta^n_j\right|^2\big|\mathcal{I}^n_{j-1}\right]\to^p 0$, then 
$\sup_{t\in\mathbb{T}}\left|\sum_{j=1}^{N^n(t)}\left\{\zeta^n_j-E\left[\zeta^n_j\big|\mathcal{I}^n_{j-1}\right]\right\}\right|\to^p 0.$
\end{lem}
The proof of this lemma is essentially the same as that of Lemma 2.3 from \cite{Fu2010b}, so we omit it. 

Next we show that we may assume that the following strengthened version of \ref{hypo:A4}:
\begin{enumerate}[label={\normalfont[SA4]}]

\item \label{hypo:SA4} We have \ref{hypo:A4}, and for every $n$ it holds that
\begin{equation}\label{SA4}
\sup_{p\geq0}(T_p-T_{p-1})\leq\bar{r}_n.
\end{equation}

\end{enumerate}

The following lemma is a version of Lemma 4.1 from \citet{HJY2011}:
\begin{lem}\label{HJYlem4.1}
Assume \ref{hypo:A4}. One can find sampling schemes $(\widetilde{T}_{p})$ and $(\widetilde{\tau}^k_{p})$ $(k=1,\dots,d)$ such that
\begin{enumerate}[noitemsep,label=(\roman*),font=\normalfont]


\item $(\widetilde{T}_{p})$ and $(\widetilde{\tau}^k_{p})$ satisfy \ref{hypo:SA4} with the same limiting processes $G$ and $\chi$ as those of the original sampling schemes,

\item There is a subset $\Omega^{(0)}_{n}$ of $\Omega^{(0)}$ such that $\lim_nP^{(0)}(\Omega^{(0)}_{n})=1$. Moreover, on $\Omega^{(0)}_{n}$ we have $T_p\wedge T=\widetilde{T}_p\wedge T$ and $\tau^k_p\wedge T=\widetilde{\tau}^k_{p}\wedge T$ for all $k,p$.

\end{enumerate}
\end{lem}

\begin{proof}
Set $R_n=\inf\{s:r_n(s)>\bar{r}_n\}$. Since $(r_n(s))_{s\geq0}$ is an $\mathbf{F}^{(0)}$-adapted continuous nondecreasing process, $R_n$ is an $\mathbf{F}^{(0)}$-stopping time. Moreover, $\Omega^{(0)}_{n}:=\{R_n>T\}$ satisfies $\lim_nP^{(0)}(\Omega^{(0)}_{n})=1$ by \eqref{A4}. Now we define $(\widetilde{T}_{p})_{p=-1}^\infty$ sequentially by $\widetilde{T}_{-1}=0$ and
\begin{equation*}
\widetilde{T}_{p}=
\left\{\begin{array}{ll}
T_p\wedge R_n, & \textrm{if $T_{p-1}<R_n$},\\
\widetilde{T}_{p-1}+n^{-1}, & \textrm{otherwise}.
\end{array}\right.
\end{equation*}
Since we can rewrite $\widetilde{T}_p$ as
\begin{equation}\label{Ttilderep}
\widetilde{T}_p=\left(T_p\wedge R_n\right)_{\{T_{p-1}<R_n\}}\wedge\left(\widetilde{T}_{p-1}\vee R_n+n^{-1}\right)_{\{T_{p-1}\geq R_n\}},
\end{equation}
$\widetilde{T}_p$ is an $\mathbf{F}^{(0)}$-stopping time. Then it is obvious that $(\widetilde{T}_{p})$ is a sampling scheme and satisfies \eqref{SA4}.  After that, for each $k$ we define $(\widetilde{\tau}^k_{p})_{p=-1}^\infty$ sequentially by $\widetilde{\tau}^k_{-1}=0$ and
\begin{equation*}
\widetilde{\tau}^k_{p}=
\left\{\begin{array}{ll}
\tau^k_p\wedge R_n, & \textrm{if $T_{p-1}<R_n$},\\
\widetilde{T}_{p}, & \textrm{otherwise}.
\end{array}\right.
\end{equation*}
Since $\tau^k_p$ has a similar representation to Eq.\eqref{Ttilderep}, it is an $\mathbf{F}^{(0)}$-stopping time. Moreover, it is evident that $(\widetilde{T}_{p})$ and $(\widetilde{\tau}^k_{p})$ satisfy \ref{hypo:H} and (ii).

Next, for each $n\geq1$ and any $k,l=1,\dots,d$ we define the processes $\widetilde{G}^n$ and $\widetilde{\chi}^{n}$ by
\begin{align*}
\widetilde{G}^n_t=G^n_t1_{[0,R_n)}(t)+1_{[R_n,\infty)}(t),\qquad
\widetilde{\chi}^{n,kl}_{t}=\chi^{n,kl}_t1_{[0,R_n)}(t)+1_{[R_n,\infty)}(t).
\end{align*}
These processes are obviously $\mathbf{F}^{(0)}$-optional. Moreover, by construction $(\widetilde{T}_{p+1}-\widetilde{T}_{p})$ is equal to $(T_{p+1}-T_p)$ on the set $\{T_{p+1}<R_n\}$, and to $n^{-1}$ on the set $\{T_{p}\geq R_n\}$. Therefore, setting $\widetilde{\mathcal{N}}^n=\mathcal{N}^n\cup\{p\in\mathbb{Z}_+:T_{p}<R_n\leq T_{p+1}\}$, we have $\widetilde{G}^n_{\widetilde{T}_{p}}=E\left[n(\widetilde{T}_{p+1}-\widetilde{T}_{p})\big|\mathcal{F}^{(0)}_{\widetilde{T}_{p}}\right]$ for every $p\in\mathbb{Z}_+-\widetilde{\mathcal{N}}^n$. Similarly, we also have $\widetilde{\chi}^{n,kl}_{\widetilde{T}_{p}}=P\left(\widetilde{\tau}^k_{p+1}=\widetilde{\tau}^l_{p+1}\big|\mathcal{F}^{(0)}_{\widetilde{T}_{p}}\right)$ for every $p\in\mathbb{Z}_+-\widetilde{\mathcal{N}}^n$. Moreover, since $R_n\to\infty$ as $n\to\infty$ by \eqref{A4}, we have $\lim_nP(\sup_{0\leq t\leq T}|\widetilde{G}^n_t-G^n_t|>0)=0$ and $\lim_nP(\sup_{0\leq t\leq T}\|\widetilde{\chi}^{n}_t-\chi^{n}_t\|>0)=0$. This implies that $(\widetilde{T}_{p})$ and $(\widetilde{\tau}^k_{p})$ satisfy (i), and thus the proof is completed.
\end{proof}

\begin{proof}[\upshape{\textbf{Proof of Lemma \ref{HJYlem2.2}}}]
A standard localization argument, based on Lemma \ref{HJYlem4.1}, allows us to assume the strengthened version \ref{hypo:SA4} of \ref{hypo:A4}. 

We begin by proving $N^n_t=O_p(n)$. \eqref{SA4}, \ref{hypo:A4i}--(ii) and \eqref{est.xi} yield
\begin{equation}\label{eq.stop.tight}
N^n_{t}=\sum_{p=2}^{N^n_{t}+1}\frac{E\left[n(T_p-T_{p-1})\big|\mathcal{F}^{(0)}_{T_{p-1}}\right]}{G^n_{T_{p-1}}}+o_p(n).
\end{equation}
On the other hand, \eqref{SA4} again imply that
$E\left[\sum_{p=2}^{N^n_{t}+1}E\left[n(T_p-T_{p-1})\big|\mathcal{F}^{(0)}_{T_{p-1}}\right]\right]
\leq nt+n\bar{r}_n$, hence $\sum_{p=2}^{N^n_{t}+1}E\left[n(T_p-T_{p-1})\big|\mathcal{F}^{(0)}_{T_{p-1}}\right]=O_p(n)$. Now, since $\sup_{0\leq s\leq T}(1/G^n_s)=O_p(1)$ by \ref{hypo:A4iii}, \eqref{eq.stop.tight} yields $N^n_t=O_p(n)$.

Next, \eqref{SA4} as well as the tightness of $N^n_t/n$ and $\sup_{0\leq s\leq t}(1/G^n_s)$ imply that
\begin{equation*}
\sum_{p=2}^{N^n_{t}+1}\frac{E\left[(T_p-T_{p-1})^2\big|\mathcal{F}^{(0)}_{T_{p-1}}\right]}{(G^n_{T_{p-1}})^2}
\leq\bar{r}_n^2\left(\sup_{0\leq s\leq t}\frac{1}{G^n_s}\right)^2N^n_t=o_p(1),
\end{equation*}
hence Lemma \ref{useful} yields
\begin{equation*}
\sum_{p=2}^{N^n_{t}+1}\frac{E\left[(T_p-T_{p-1})\big|\mathcal{F}^{(0)}_{T_{p-1}}\right]}{G^n_{T_{p-1}}}
=\sum_{p=2}^{N^n_{t}+1}\frac{(T_p-T_{p-1})}{G^n_{T_{p-1}}}+o_p(1)
=\int_0^t\frac{1}{G_s}\mathrm{d}s+o_p(1).
\end{equation*}
Combining this with Eq.\eqref{eq.stop.tight}, we obtain the desired result.
\end{proof}

\subsection{Proof of Theorem \ref{mainthm}}


\subsubsection{Outline of the proof}

First, we note that we may also strengthen conditions \ref{hypo:A1}--\ref{hypo:A3} due to a standard localization procedure which is described in detail e.g.~in Lemma 4.4.9 of \cite{JP2012} as follows:
\begin{enumerate}[noitemsep,label={\normalfont[SA\arabic*]}]

\item \label{hypo:SA1} $a_t$ is bounded, and there is a constant $\Lambda$ such that
\begin{equation}\label{SA5}
E\left[\|a_{t_1}-a_{t_2}\|^2|\mathcal{F}_{t_1\wedge t_2}\right]
\leq \Lambda E\left[|t_1-t_2|^\varpi|\mathcal{F}_{t_1\wedge t_2}\right]
\end{equation}
for any bounded $\mathbf{F}^{(0)}$-stopping times $t_1$ and $t_2$.

\item \label{hypo:SA2} $\sigma_t$ is bounded, and there is a constant $\Lambda$ such that
\begin{equation}\label{SA3}
E\left[\|\sigma_{t_1}-\sigma_{t_2}\|^2|\mathcal{F}_{t_1\wedge t_2}\right]
\leq \Lambda E\left[|t_1-t_2|^\varpi|\mathcal{F}_{t_1\wedge t_2}\right]
\end{equation}
for any bounded $\mathbf{F}^{(0)}$-stopping times $t_1$ and $t_2$.

\item \label{hypo:SA3} There is a constant $\Gamma>4$ and a constant $\Lambda$ such that the process $\int\|z\|^\Gamma Q_t(\mathrm{d}z)$ is bounded and 
\[
E\left[\|\Upsilon_{t_1}-\Upsilon_{t_2}\|^2|\mathcal{F}_{t_1\wedge t_2}\right]
\leq \Lambda E\left[|t_1-t_2|^\varpi|\mathcal{F}_{t_1\wedge t_2}\right]
\]
for any bounded $\mathbf{F}^{(0)}$-stopping times $t_1$ and $t_2$. Moreover, $\Upsilon_t$ is c\`adl\`ag.

\end{enumerate}


Next we introduce some notation. Set $I_p=[T_{p-1},T_p)$ for every $p\in\mathbb{Z}_+$. 
For any process $V$ and any (random) interval $I=[S,T)$, we define the random variable $V(I)$ by $V(I)=V_T-V_S$. We also set $|I|=T-S$. For any real-valued function $u$ on $\mathbb{R}$, we set $u^n_p=u(p/k_n)$ for $p\in\mathbb{Z}$. For any $d$-dimensional processes $U$, $V$, any $k,l\in\{1,\dots,d\}$ and any $u,v\in\{g,g'\}$, we define the process $\Xi^{(k,l)}_{u,v}(U,V)^n$ by
\begin{align*}
\Xi^{(k,l)}_{u,v}(U,V)^n_t=\frac{1}{\psi_2 k_n}\sum_{i=k_n}^{N^n_t-k_n+1}\breve{U}(u)^k_i\breve{V}(v)^l_i,
\end{align*}
where 
$\breve{U}(u)^k_{i}=\sum_{p=k_n}^{N^n_T-k_n} u^n_{p-i}U^k(I_p)$ and $\breve{V}(v)^l_{i}$ is defined analogously. Moreover, we define the processes $A$ and $M$ by $A_t=\int_0^ta_s\mathrm{d}s$ and $M_t=\int_0^t\sigma_s\mathrm{d}W_s$ respectively, and also define the $d$-dimensional process $\mathfrak{E}$ by
\begin{align*}
\mathfrak{E}^k_t=-\frac{1}{k_n}\sum_{p=1}^\infty\epsilon^k_{\tau^k_p}1_{\{\tau^k_p\leq t\}},\qquad
t\in\mathbb{R}_+,\quad
k=1,\dots,d.
\end{align*}
It can easily be seen that $\mathfrak{E}$ is a purely discontinuous locally square-integrable martingale on $\mathcal{B}$ under \ref{hypo:SA3}. 

Now we turn to the outline of the proof. In the first step we show that the errors from end effects and interpolations to the synchronized sampling times are asymptotically negligible:
\begin{prop}\label{synchronization}
Assume \ref{hypo:W}, \ref{hypo:H} and \ref{hypo:SA1}--\ref{hypo:SA4}. Then 
$\sup_{0\leq t\leq T}\left\|\widetilde{\MRC}[Y]^{n}_t-\mathbf{\Xi}[X]^{n}_t+\frac{\psi_1}{2\psi_2k_n^2}[Y,Y]^{n}_t\right\|=o_p(n^{-1/4})$ 
as $n\to\infty$, where $\mathbf{\Xi}[X]^{n}$ is the $\mathbb{R}^d\otimes\mathbb{R}^d$-valued process such that
\begin{equation*}
\mathbf{\Xi}[X]^{n,kl}=\Xi^{(k,l)}_{g,g}(X,X)^n+\Xi^{(k,l)}_{g,g'}(X,\mathfrak{E})^n
+\Xi^{(l,k)}_{g,g'}(X,\mathfrak{E})^n
+\Xi^{(k,l)}_{g',g'}(\mathfrak{E},\mathfrak{E})^n,\qquad
k,l=1,\dots,d.
\end{equation*}
\end{prop}


In the next step we prove a martingale approximation of the error process. For any $d$-dimensional processes $U,V$, any $k,l\in\{1,\dots,d\}$ and any real-valued functions $u,v$ on $[0,1]$, we define the processes $\mathbb{M}^{(k,l)}_{u,v}(U,V)^n$ and $\mathbb{L}^{(k,l)}_{u,v}(U,V)^n$ by
\begin{align*}
\mathbb{M}^{(k,l)}_{u,v}(U,V)^n_t=\sum_{q=k_n}^{N^n_t+1}C^n_{u,v}(U)^k_qV^l(I_q),\qquad
\mathbb{L}^{(k,l)}_{u,v}(U,V)^n_t=\mathbb{M}^{(k,l)}_{u,v}(U,V)^n_t+\mathbb{M}^{(l,k)}_{v,u}(V,U)^n_t,
\end{align*}
where
\begin{align*} 
C^n_{u,v}(U)^k_q=\sum_{p=(q-d_n)\vee k_n}^{q-1}c^n_{u,v}(p,q)U^k(I_p),\qquad
c^n_{u,v}(p,q)=\frac{1}{\psi_2k_n}\sum_{i=k_n}^\infty u^n_{p-i}v^n_{q-i}.
\end{align*}
Here, let us recall that the number $d_n$ is given by $d_n=\lceil n^{-\gamma}\rceil$ and $\gamma$ satisfies \eqref{est.xi}. 
Moreover, define the $\mathbb{R}^d\otimes\mathbb{R}^d$-valued process $\mathbf{L}[M]^n$ by
\begin{align*}
\mathbf{L}[M]^{n,kl}=\mathbb{L}^{(k,l)}_{g,g}(M,M)^n+\mathbb{L}^{(k,l)}_{g,g'}(M,\mathfrak{E})^n
+\mathbb{L}^{(l,k)}_{g,g'}(M,\mathfrak{E})^n
+\mathbb{L}^{(k,l)}_{g',g'}(\mathfrak{E},\mathfrak{E})^n.
\end{align*}

\begin{prop}\label{approx}
Under the assumptions of Proposition \ref{synchronization}, 
$\sup_{0\leq t\leq T}\left\|\mathbf{\Xi}[X]^n_t-[X,X]_t-\frac{\psi_1}{2\psi_2k_n^2}[Y,Y]^{n}_t-\mathbf{L}[M]^n_t\right\|=o_p(n^{-1/4})$ 
as $n\to\infty$.
\end{prop}


The above two propositions suggest that it suffices to prove the following stable limit theorem in $\mathbb{D}_T^{d\times d}$:
\begin{equation}\label{CLT:Lprocess}
n^{1/4}\mathbf{L}[M]^n\to^{d_s}\mathcal{W}.
\end{equation}
For the proof we apply Jacod's stable limit theorem, and especially the version from \cite{JP2012} (note that condition \eqref{progressive} ensures that $\mathcal{B}$ is a very good filtered extension of $\mathcal{B}^{(0)}$, i.e.~the variable $Q(\cdot,A)$ is $\mathcal{F}_t^{(0)}$-measurable for all $A\in\mathcal{F}_t$ and all $t\in\mathbb{R}_+$). Set 
\[
\zeta^{(k,l)}_{u,v}(U,V)^n_q=n^{1/4}\{C^n_{u,v}(U)^k_qV^l(I_q)+C^n_{v,u}(V)^l_qU^k(I_q)\}
\]
for $U,V\in\{M,\mathfrak{E}\}$, $u,v\in\{g,g'\}$, $k,l\in\{1,\dots,d\}$ and $q\geq k_n$. Then we define the $\mathbb{R}^d\otimes\mathbb{R}^d$-valued random variable $\zeta^n_q=(\zeta^{n,kl}_q)_{1\leq k,l\leq d}$ by 
 $\zeta^{n,kl}_q=\zeta^{(k,l)}_{g,g}(M,M)^n_q+\zeta^{(k,l)}_{g,g'}(M,\mathfrak{E})^n_q+\zeta^{(k,l)}_{g',g}(\mathfrak{E},M)^n_q+\zeta^{(k,l)}_{g',g'}(\mathfrak{E},\mathfrak{E})^n_q$. Since $n^{1/4}\mathbf{L}[M]^n_t=\sum_{q=k_n}^{N^n_t+1}\zeta^n_q$ and $\zeta^{n}_q$ is $\mathcal{F}_{T_q}$-measurable and satisfies $E[\zeta^{n,kl}_q|\mathcal{F}_{T_{q-1}}]=0$, in the light of Theorem 2.2.15 of \cite{JP2012} it suffices to verify the following conditions:
\begin{align}
&\textstyle\sum_{q=k_n}^{N^n_t+1}E\left[\zeta^{n,kl}_q\zeta^{n,k'l'}_q\big|\mathcal{F}_{T_{q-1}}\right]\to^p\int_0^t\mathfrak{V}_s^{klk'l'}\mathrm{d}s,\label{eq.energy}\\
&\textstyle\sum_{q=k_n}^{N^n_t+1}E\left[\left|\zeta^{n,kl}_q\right|^4\big|\mathcal{F}_{T_{q-1}}\right]\to^p0,\label{eq.Lindeberg}\\
&\textstyle\sum_{q=k_n}^{N^n_t+1}E\left[\zeta^{n,kl}_qW^{j}(I_q)\big|\mathcal{F}_{T_{q-1}}\right]\to^p0,\label{eq.unbias}\\
&\textstyle\sum_{q=k_n}^{N^n_t+1}E\left[\zeta^{n,kl}_qN(I_q)\big|\mathcal{F}_{T_{q-1}}\right]\to^p0\label{eq.orthogonal}
\end{align}
for any $t>0$, $k,l,k,l'\in\{1,\dots,d\}$, $j\in\{1,\dots,d'\}$ and any bounded $\mathbf{F}^{(0)}$-martingale $N$ orthogonal to $W$. Here, $\mathfrak{V}_s^{klk'l'}$ is the integrand in the right hand side of \eqref{avar}. 

Eq.\eqref{eq.energy} follows from the following lemma:
\begin{lem}\label{lemenergy}
Under the assumptions of Proposition \ref{synchronization}, it holds that 
\begin{align*}
&\textstyle\sum_{q=k_n}^{N^n_t+1}E\left[\zeta^{(k,l)}_{g,g}(M,M)^n_q\zeta^{(k',l')}_{g,g}(M,M)^n_q\big|\mathcal{F}_{T_{q-1}}\right]\to^p2\theta\frac{\Phi_{22}}{\psi_2^2}\int_0^t\left\{\Sigma^{kk'}_s\Sigma^{ll'}_s+\Sigma^{kl'}_s\Sigma^{lk'}_s\right\}G_s\mathrm{d}s,\\
&\textstyle\sum_{q=k_n}^{N^n_t+1}E\left[\zeta^{(k,l)}_{g,g'}(M,\mathfrak{E})^n_q\zeta^{(k',l')}_{g,g'}(M,\mathfrak{E})^n_q\big|\mathcal{F}_{T_{q-1}}\right]
\to^p2\frac{\Phi_{12}}{\theta\psi_2^2}\int_0^t\Sigma^{kk'}_s\widetilde{\Upsilon}^{ll'}_s\mathrm{d}s,\\
&\textstyle\sum_{q=k_n}^{N^n_t+1}E\left[\zeta^{(k,l)}_{g',g'}(\mathfrak{E},\mathfrak{E})^n_q\zeta^{(k',l')}_{g',g'}(\mathfrak{E},\mathfrak{E})^n_q\big|\mathcal{F}_{T_{q-1}}\right]
\to^p2\frac{\Phi_{11}}{\theta^3\psi_2^2}\int_0^t\left\{\widetilde{\Upsilon}^{kk'}_s\widetilde{\Upsilon}^{ll'}_s+\widetilde{\Upsilon}^{kl'}_s\widetilde{\Upsilon}^{lk'}_s\right\}\frac{1}{G_s}\mathrm{d}s,\\
&\textstyle\sum_{q=k_n}^{N^n_t+1}E\left[\zeta^{(k,l)}_{g,g}(M,M)^n_q\zeta^{(k',l')}_{g,g'}(M,\mathfrak{E})^n_q\big|\mathcal{F}_{T_{q-1}}\right]\to^p0,\quad
\sum_{q=k_n}^{N^n_t+1}E\left[\zeta^{(k,l)}_{g,g}(M,M)^n_q\zeta^{(k',l')}_{g',g'}(\mathfrak{E},\mathfrak{E})^n_q\big|\mathcal{F}_{T_{q-1}}\right]\to^p0,\\
&\textstyle\sum_{q=k_n}^{N^n_t+1}E\left[\zeta^{(k,l)}_{g,g'}(M,\mathfrak{E})^n_q\zeta^{(k',l')}_{g',g'}(\mathfrak{E},\mathfrak{E})^n_q\big|\mathcal{F}_{T_{q-1}}\right]\to^p0
\end{align*}
as $n\to\infty$ for all $k,l,k',l'$ and all $t\in[0,T]$.
\end{lem}

On the other hand, Eqs.\eqref{eq.Lindeberg}--\eqref{eq.orthogonal} follow from the following lemma. 
\begin{lem}\label{lemremain}
Let $k,l\in\{1,\dots,d\}$, $u,v\in\{g,g'\}$, $U,V\in\{M,\mathfrak{E}\}$ and $t\in[0,T]$. Under the assumptions of Proposition \ref{synchronization}, the following statements hold true:
\begin{enumerate}[noitemsep,label={\normalfont(\alph*)}]

\item $n\sum_{q=k_n}^{N^n_{t}+1}E\left[\left|C^n_{u,v}(U)^k_qV^l(I_q)\right|^4\big|\mathcal{F}_{T_{q-1}}\right]\to^p0$ as $n\to\infty$,

\item $n^{1/4}\sum_{q=k_n}^{N^n_{t}+1}E\left[C^n_{u,v}(U)^k_qV^l(I_q)W^{j}(I_q)\big|\mathcal{F}_{T_{q-1}}\right]\to^p0$ for every $j=1,\dots,d'$,

\item $n^{1/4}\sum_{q=k_n}^{N^n_{t}+1}E\left[C^n_{u,v}(U)^k_qV^l(I_q)N(I_q)\big|\mathcal{F}_{T_{q-1}}\right]\to^p0$ as $n\to\infty$ for any one-dimensional square-integrable martingale $N$ on $\mathcal{B}^{(0)}$ orthogonal to $M$.

\end{enumerate}
\end{lem}

\subsubsection{Proof of Proposition \ref{synchronization}}

Throughout the discussions, for (random) sequences $(x_n)$ and $(y_n)$, $x_n\lesssim y_n$ means that there exists a (non-random) constant $K\in[0,\infty)$ such that $x_n\leq Ky_n$ for large $n$. Also, we denote by $E_0$ the conditional expectation given $\mathcal{F}^{(0)}$, i.e.~$E_0[\cdot]:=E[\cdot|\mathcal{F}^{(0)}]$. Moreover, for each $\delta\in(0,T)$, we set $\beta_\delta=\sup_{0\leq h\leq \delta}\left(\|X_h-X_0\|+\|X_T-X_{T-h}\|\right)$. Eqs.(2.1.33)--(2.1.34) from \cite{JP2012} and \ref{hypo:SA1}--\ref{hypo:SA2} imply that, for any $r\geq1$, there is a constant $K_r$ such that
\begin{equation}\label{estimate:beta}
E\left[\left(\beta_\delta\right)^r\right]\leq K_r\delta^{r/2}
\end{equation}
for any $\delta\in(0,T)$. 
 
\begin{lem}\label{lemma:estimate}
Under the assumptions of Proposition \ref{synchronization}, the following statements hold true:
\begin{enumerate}[label={\normalfont(\alph*)}]

\item For any $r\in[0,\Gamma]$, there is a constant $K_r>0$ such that
\begin{align}
&E\left[\left\|\widetilde{X}_{i,T}\right\|^r+\left\|\breve{X}(g)_i\right\|^r\big|\mathcal{F}^{(0)}_{T_{(i-d_n+1)_+}}\right]
\leq K_r\left\{E\left[\left(\beta_{(k_n+1)\bar{r}_n}\right)^r\big|\mathcal{F}^{(0)}_{T_{(i-d_n+1)_+}}\right]+\left(k_n\bar{r}_n\right)^{r/2}\right\},\label{eq:avg1}\\
&E_0\left[\left\|\widetilde{\epsilon}_{i,T}\right\|^r+\|\breve{\mathfrak{E}}(g')_i\|^r\right]
\leq K_r k_n^{-r/2}\label{eq:avg2}
\end{align}
for every $i$.

\item There is a constant $K>0$ such that
\begin{align}
&E\left[\left\|\widetilde{X}_{i,T}-\breve{X}(g)_i\right\|^2\big|\mathcal{F}^{(0)}_{T_{(i-d_n+1)_+}}\right]
\leq K\left\{k_n^{-1}\bar{r}_n+\left(|g^n_{N^n_T-k_n+1-i}|^2+|g^n_{k_n-i}|^2\right)E\left[\left(\beta_{(k_n+1)\bar{r}_n}\right)^2\big|\mathcal{F}^{(0)}_{T_{(i-d_n+1)_+}}\right]\right\}\label{eq:syn1}
\end{align}
for every $i$.

\end{enumerate}
\end{lem}

\begin{proof}
(a) First, by \eqref{SA4} we have $|\widetilde{\Delta}_{\tau^k_{k_n}}X^k|+|\widetilde{\Delta}_{\tau^l_{N^n_T-k_n+1}}X^k|\leq2\beta_{(k_n+1)\bar{r}_n}$. Moreover, the Burkholder-Davis-Gundy (henceforth BDG) inequality, \eqref{SA4} and [W] yield
\begin{align*}
E\left[\left|\sum_{p=k_n+1}^{N^n_T-k_n}g^n_{p-i}X^k(I^k_p)\right|^r\big|\mathcal{F}_{T_{(i-d_n+1)_+}}\right]
\leq E\left[\max_{m\in\mathbb{Z}_+}\left|\sum_{p=k_n+1}^{m}g^n_{p-i}X^k(I^k_p)\right|^r\big|\mathcal{F}_{T_{(i-d_n+1)_+}}\right]
\lesssim\left(k_n\bar{r}_n\right)^{r/2}.
\end{align*}
This inequality also holds true when we replace $X^k(I^k_p)$ with $X^k(I_p)$, hence we obtain \eqref{eq:avg1}. 

Next, summation by parts yields
\begin{equation}\label{eq:epsilon}
\widetilde{\epsilon}^k_{i,T}
=-\sum_{p=k_n}^{N^n_T-k_n}\Delta(g)^n_{p-i}\epsilon^k_{\tau^k_p}
+g^n_{N^n_T-k_n+1-i}\mathring{\epsilon}^k_T
-g^n_{k_n-i}\mathring{\epsilon}^k_0,
\end{equation}
hence the equation $\Delta(g)^n_{p-i}=\int_{(p-i)/k_n}^{(p-i+1)/k_n}g'(x)\mathrm{d}x$, [W] and the BDG inequality yield $E_0\left[\left\|\widetilde{\epsilon}_{i,T}\right\|^r\right]\lesssim k_n^{-r/2}$. On the other hand, since $\mathfrak{E}^k(I_p)=-k_n^{-1}\sum_q\epsilon^k_{\tau^k_q}1_{\{T_{p-1}<\tau^k_q\leq T_p\}}=-k_n^{-1}\epsilon^k_{\tau^k_p}$, the BDG inequality again yields $E_0\left[\left\|\breve{\mathfrak{E}}(g')_i\right\|^r\right]\lesssim k_n^{-r/2}$, hence we obtain \eqref{eq:avg2}. 

(b) Summation by parts yields
\begin{align*}
&\widetilde{X}_{i,T}-\breve{X}(g)_i\\
=&-\sum_{p=k_n}^{N^n_T-k_n}\Delta(g)^n_{p-i}\left(X^k_{\tau^k_p}-X^k_{T_p}\right)
+g^n_{N^n_T-k_n+1-i}\left(\mathring{X}^k_{T}-X^k_{T_{N^n_T-k_n}}\right)
-g^n_{k_n-i}\left(\mathring{X}^k_0-X^k_{T_{k_n-1}}\right),
\end{align*}
hence \eqref{eq:syn1} can be shown in a similar manner to the proof of \eqref{eq:avg1} using the Lipschitz continuity of $g$. 
\end{proof}

\begin{proof}[\upshape{\bfseries{Proof of Proposition \ref{synchronization}}}]
Fix $\alpha>0$, and define the $\mathbf{F}^{(0)}$-stopping time $R^n_\alpha$ by
\begin{equation}\label{def:Rn}
R^n_{\alpha}=\inf\{t:n^{-1}N^n_t>\alpha\}.
\end{equation}
Since $\Delta N^n_t\leq 1$ for every $t$, it holds that
\begin{equation}\label{SC3}
N^n_{t\wedge R^n_\alpha}\leq\alpha n+1
\end{equation}
for all $t\geq0$. Moreover, by Lemma \ref{HJYlem2.2} we also have
\begin{equation}\label{alpha.infinity}
\limsup_{\alpha\to\infty}\limsup_{n\to\infty}P\left(R^n_\alpha\leq T\right)=0.
\end{equation}
In particular, by the Markov inequality and \eqref{alpha.infinity} it is enough to prove 
\[
E\left[\sup_{0\leq t\leq T\wedge R^n_\alpha}\left|\widetilde{\MRC}[Y]^{n,kl}_t-\mathbf{\Xi}[X]^{n,kl}_t+\frac{\psi_1}{\psi_2k_n^2}[Y,Y]^{n,kl}_t\right|\right]=o(n^{-1/4})\qquad\text{ for any }\alpha>0.
\] 
In view of Lemma \ref{lemma:estimate}, for this it suffices to show the following equations for any $k,l=1,\dots,d$ and any $\alpha>0$:
\begin{align}
&\textstyle\sup_{0\leq t\leq T\wedge R^n_\alpha}\left|\frac{1}{k_n}\sum_{i=k_n}^{N^n_t-k_n+1}\left\{\widetilde{\epsilon}^k_{i,T}-\breve{\mathfrak{E}}(g')^k_i\right\}\widetilde{X}^l_{i,T}\right|=o_p(n^{-1/4}),\label{revise3}\\
&\textstyle\sup_{0\leq t\leq T\wedge R^n_\alpha}\left|\frac{1}{k_n}\sum_{i=k_n}^{N^n_t-k_n+1}\left\{\widetilde{\epsilon}^k_{i,T}-\breve{\mathfrak{E}}(g')^k_i\right\}\breve{X}(g)^l_i\right|=o_p(n^{-1/4}),\label{revise4}\\
&\textstyle\sup_{0\leq t\leq T\wedge R^n_\alpha}\left|\frac{1}{k_n}\sum_{i=k_n}^{N^n_t-k_n+1}\left\{\widetilde{\epsilon}^k_{i,T}-\breve{\mathfrak{E}}(g')^k_i\right\}\widetilde{\epsilon}^l_{i,T}\right|=o_p(n^{-1/4}),\label{revise5}\\
&\textstyle\sup_{0\leq t\leq T\wedge R^n_\alpha}\left|\frac{1}{k_n}\sum_{i=k_n}^{N^n_t-k_n+1}\left\{\widetilde{\epsilon}^k_{i,T}-\breve{\mathfrak{E}}(g')^k_i\right\}\breve{\mathfrak{E}}(g')^l_i\right|=o_p(n^{-1/4}).\label{revise6}
\end{align}
Since \eqref{revise4} (resp.~\eqref{revise6}) can be shown in a similar manner to \eqref{revise3} (resp.~\eqref{revise5}), we only prove \eqref{revise3} and \eqref{revise5}.

First, thanks to \ref{hypo:W}(i), there are points $-\infty=:x_0<x_1<\cdots<x_\Lambda<x_{\Lambda+1}:=\infty$ such that $g$ is of $C^1$ and $g'$ is Lipschitz continuous on $(x_\lambda,x_{\lambda+1})$ for every $\lambda=0,1\,\dots,\Lambda$. We denote by $\mathcal{P}_n$ the set of all integers $p$ such that $x_\lambda\in[p/k_n,(p+1)/k_n]$ for some $\lambda\in\{1,\dots,\Lambda\}$. We evidently have $\#\mathcal{P}_n\leq2\Lambda$. Also, let us set $\Delta^2(g)^n_p=k_n\Delta(g)^n_p-(g')^n_p$. Then the following claims hold true: (I) $\sup_p|\Delta^2(g)^n_p|<\infty$, (II) $\sup_{p\notin\mathcal{P}_n}|k_n\Delta^2(g)^n_p|<\infty$ and (III) $\sup_{p:|p|>d_n}|n^K\Delta^2(g)^n_p|<\infty$ for any $K>0$. 
In fact, (I) is a consequence of the Lipschitz continuity of $g$ and the boundedness of $g'$. (II) follows from the identity $\Delta^2(g)^n_p=k_n\int_{p/k_n}^{(p+1)/k_n}\left\{g'(x)-g'(p/k_n)\right\}\mathrm{d}x$ and the fact that $g'$ is Lipschitz continuous on $(x_{\lambda},x_{\lambda+1})$ for every $\lambda$. (III) is a consequence of \ref{hypo:W}(ii).

Now, (II) and (III) imply that there is a constant $C>0$ such that
\begin{align*}
E_0\left[\left|\frac{1}{k_n}\sum_{\begin{subarray}{c}
p=k_n\\
p-i\notin\mathcal{P}_n
\end{subarray}}^{N^n_T-k_n}\Delta^2(g)^n_{p-i}\epsilon^k_{\tau^k_p}\right|^2\right]
\leq Ck_n^{-4}d_n
\end{align*}
for every $i$. Therefore, noting the identity
\begin{align*}
\widetilde{\epsilon}^k_{i,T}-\breve{\mathfrak{E}}(g')_i
=-\frac{1}{k_n}\sum_{p=k_n}^{N^n_T-k_n}\Delta^2(g)^n_{p-i}\epsilon^k_{\tau^k_p}
+g^n_{N^n_T-k_n+1-i}\mathring{\epsilon}^k_T
-g^n_{k_n-i}\mathring{\epsilon}^k_0,
\end{align*}
which follows from \eqref{eq:epsilon} and the definition of $\breve{\mathfrak{E}}(g')_i$, \eqref{revise3} and \eqref{revise5} follow once we show that
\begin{align}
&\sup_{0\leq t\leq T\wedge R^n_\alpha}\left|\frac{1}{k_n}\sum_{i=k_n}^{N^n_t-k_n+1}\left\{\frac{1}{k_n}\sum_{\begin{subarray}{c}
p=k_n\\
p-i\in\mathcal{P}_n
\end{subarray}}^{N^n_T-k_n}\Delta^2(g)^n_{p-i}\epsilon^k_{\tau^k_p}\right\}\widetilde{X}^l_{i,T}\right|=o_p(n^{-1/4}),\label{revise3-2}\\
&\sup_{0\leq t\leq T\wedge R^n_\alpha}\left|\frac{1}{k_n}\sum_{i=k_n}^{N^n_t-k_n+1}\left\{\frac{1}{k_n}\sum_{\begin{subarray}{c}
p=k_n\\
p-i\in\mathcal{P}_n
\end{subarray}}^{N^n_T-k_n}\Delta^2(g)^n_{p-i}\epsilon^k_{\tau^k_p}\right\}\left\{\sum_{q=k_n}^{N^n_T-k_n}\Delta(g)^n_{q-i}\epsilon^l_{\tau^l_q}
\right\}\right|=o_p(n^{-1/4}).\label{revise5-2}
\end{align}

First we prove \eqref{revise3-2}. Since we have
\begin{align*}
\frac{1}{k_n}\sum_{i=k_n}^{N^n_t-k_n+1}\left\{\frac{1}{k_n}\sum_{\begin{subarray}{c}
p=k_n\\
p-i\in\mathcal{P}_n
\end{subarray}}^{N^n_T-k_n}\Delta^2(g)^n_{p-i}\epsilon^k_{\tau^k_p}\right\}\widetilde{X}^l_{i,T}
=\frac{1}{k_n^2}\sum_{\begin{subarray}{c}
p=-N^n_t+2k_n-1\\
p\in\mathcal{P}_n
\end{subarray}}^{N^n_T-2k_n}\Delta^2(g)^n_{p}\sum_{i=(k_n-p)\vee k_n}^{\nu_n(t,p)}\epsilon^k_{\tau^k_{i+p}}\widetilde{X}^l_{i,T},
\end{align*}
where $\nu_n(t,p)=(N^n_T-k_n-p)\wedge(N^n_t-k_n+1)$, the Davis inequality and (I) imply that
\begin{align*}
&E_0\left[\sup_{0\leq t\leq T\wedge R^n_\alpha}\left|\frac{1}{k_n}\sum_{i=k_n}^{N^n_t-k_n+1}\left\{\frac{1}{k_n}\sum_{\begin{subarray}{c}
p=k_n\\
p-i\in\mathcal{P}_n
\end{subarray}}^{N^n_T-k_n}\Delta^2(g)^n_{p-i}\epsilon^k_{\tau^k_p}\right\}\widetilde{X}^l_{i,T}\right|\right]
\lesssim \frac{2\Lambda}{k_n^2}\sqrt{\sum_{i=k_n}^{N^n_{T\wedge R^n_\alpha}-k_n+1}\left|\widetilde{X}^l_{i,T}\right|^2}.
\end{align*}
Hence \eqref{revise3-2} holds true by Lemma 6.6 and \eqref{SC3}. 

Next we prove \eqref{revise5-2}. We decompose the target quantity as
\begin{align*}
&\frac{1}{k_n}\sum_{i=k_n}^{N^n_t-k_n+1}\left\{\frac{1}{k_n}\sum_{\begin{subarray}{c}
p=k_n\\
p-i\in\mathcal{P}_n
\end{subarray}}^{N^n_T-k_n}\Delta^2(g)^n_{p-i}\epsilon^k_{\tau^k_p}\right\}\left\{\sum_{q=k_n}^{N^n_T-k_n}\Delta(g)^n_{q-i}\epsilon^l_{\tau^l_q}\right\}\\
&=\frac{1}{k_n^2}\sum_{\begin{subarray}{c}
p=-N^n_t+2k_n-1\\
p\in\mathcal{P}_n
\end{subarray}}^{N^n_T-2k_n}\Delta^2(g)^n_{p}\sum_{i=(k_n-p)\vee k_n}^{\nu_n(t,p)}\epsilon^k_{\tau^k_{i+p}}\left\{\sum_{q=k_n}^{i+p-1}\Delta(g)^n_{q-i}\epsilon^l_{\tau^l_q}+\Delta(g)^n_{p}\epsilon^l_{\tau^l_{i+p}}+\sum_{q=i+p+1}^{N^n_T-k_n}\Delta(g)^n_{q-i}\epsilon^l_{\tau^l_q}\right\}\\
&=:\mathbb{I}_t+\mathbb{II}_t+\mathbb{III}_t.
\end{align*}
We can prove $\sup_{0\leq t\leq T\wedge R^n_\alpha}|\mathbb{I}_t|=o_p(n^{-1/4})$ similarly to the proof of \eqref{revise3-2}, while it can easily be seen $\sup_{0\leq t\leq T\wedge R^n_\alpha}|\mathbb{II}_t|=O_p(k_n^{-3}n)=o_p(n^{-1/4})$. Now we prove $\sup_{0\leq t\leq T\wedge R^n_\alpha}|\mathbb{III}_t|=o_p(n^{-1/4})$. For this it suffices to show that (i) the process $(n^{1/4}\mathbb{III}_t)_{t\in[0,T]}$ is C-tight, and (ii) $\mathbb{III}_t=o_p(n^{-1/4})$ for every $t\in[0,T]$.

We begin with proving (i). For $0\leq s\leq t\leq T$, the Schwarz inequality yields 
\begin{align}
|\mathbb{III}_t-\mathbb{III}_s|
&\leq\frac{1}{k_n^2}\sum_{p\in\mathcal{P}_n}|\Delta^2(g)^n_{p}|\sqrt{N^n_t-N^n_s}\left\{\sum_{i=k_n+1}^{\nu_n(T,p)}\left|\epsilon^k_{\tau^k_{i+p}}\sum_{q=i+p+1}^{N^n_T-k_n}\Delta(g)^n_{q-i}\epsilon^l_{\tau^l_q}\right|^2\right\}^{1/2}.\label{c-tight}
\end{align}
In particular, since $\mathbb{III}_0=0$, noting the identity $\Delta(g)^n_{q-i}=\int_{(q-i)/k_n}^{(q-i+1)/k_n}g'(x)\mathrm{d}x$, we obtain
\begin{equation}\label{c-tight1}
E_0\left[\sup_{0\leq t\leq T}|n^{1/4}\mathbb{III}_t|\right]
\lesssim \frac{n^{1/4}}{k_n^2}\cdot2\Lambda\cdot\frac{N^n_T}{\sqrt{k_n}}
\lesssim2\Lambda\frac{N^n_T}{n}.
\end{equation}
On the other hand, setting $w_T(f,\delta)=\sup\{|f(t)-f(s)|:t,s\in[0,T],|t-s|\leq\delta\}$ for a function $f:[0,T]\to\mathbb{R}$ and a number $\delta>0$, \eqref{c-tight} yields
\begin{align}
E_0[w_T(n^{1/4}\mathbb{III}_t,\delta)]
\lesssim 2\Lambda\sqrt{w_T(n^{-1}N^n,\delta)}\left\{n^{-1}N^n_T\right\}^{1/2}.\label{c-tight2}
\end{align}
Since the process $n^{-1}N^n$ is C-tight by Lemma \ref{HJYlem2.2} and Theorem VI-3.37 of \cite{JS}, claim (i) follows from \eqref{c-tight1}--\eqref{c-tight2} and Proposition VI-3.26 of \cite{JS}.

Next, in order to prove (ii), we rewrite $\mathbb{III}_t$ as
\begin{align*}
\mathbb{III}_t
=\frac{1}{k_n^2}\sum_{\begin{subarray}{c}
p=-N^n_t+2k_n-1\\
p\in\mathcal{P}_n
\end{subarray}}^{N^n_T-2k_n}\Delta^2(g)^n_{p}\sum_{q=(k_n-p)\vee k_n+p+1}^{N^n_T-k_n}\epsilon^l_{\tau^l_q}\sum_{i=(k_n-p)\vee k_n}^{\nu_n(t,p)\wedge(q-p-1)}\Delta(g)^n_{q-i}\epsilon^k_{\tau^k_{i+p}}.
\end{align*}
Then the Davis inequality yields $E_0[|\mathbb{III}_t|]
\lesssim k_n^{-2}\cdot2\Lambda\sqrt{N^n_T/k_n}
=o_p(n^{-1/4}),$ 
which implies that claim (ii) holds true. Consequently, we obtain \eqref{revise5-2} and the proof of the proposition is completed.
\end{proof}

\subsubsection{Proof of Proposition \ref{approx}}\label{sec.app.cont}

Next we prove some auxiliary results.
\begin{lem}\label{sup.epsilon}
Under \ref{hypo:SA3}, $\sup_{0\leq p\leq N^n_t+1}|\epsilon^k_{\tau^k_p}|=o_p(n^{1/4})$ for any $t>0$ and any $k=1,\dots,d$.
\end{lem}

\begin{proof}
Fix $\eta>0$. By the Markov inequality, \ref{hypo:SA3} and Lemma \ref{HJYlem2.2} we have
\begin{align*}
P\left(n^{-1/4}\sup_{0\leq p\leq N^n_t+1}|\epsilon^k_{\tau^k_p}|>\eta\big|\mathcal{F}^{(0)}\right)
&\leq\eta^{-\Gamma}n^{-\Gamma/4}E_0\left[\sup_{0\leq p\leq N^n_t+1}|\epsilon^k_{\tau^k_p}|^\Gamma\right]
\leq\eta^{-\Gamma}n^{-\Gamma/4}\sum_{p=0}^{N^n_t+1}E_0\left[\left|\epsilon^k_{\tau^k_{p}}\right|^\Gamma\right]\\
&\lesssim\eta^{-\Gamma}n^{-\Gamma/4}(N^n_t+2)=o_p(1),
\end{align*}
hence the desired result holds true.
\end{proof}

\begin{lem}\label{supC}
Suppose either that $V$ is a $d$-dimensional c\`adl\`ag process or that $V=\mathfrak{E}$ and \ref{hypo:SA3} holds true. Then 
\[
\sup_{1\leq q\leq N^n_t+1}|C^n_{u,v}(V)^k_{q}|=O_p(n^\gamma)
\] 
as $n\to\infty$ for any $t>0$, $u,v\in\{g,g'\}$ and $k=1,\dots,d$.
\end{lem}

\begin{proof}
First consider the former case. In this case, summation by parts yields
\begin{align*}
C^n_{u,v}(V)^k_{q}
&=\sum_{p=(q-d_n)\vee k_n}^{q-2}\left\{c^n_{u,v}(p,q)-c^n_{u,v}(p+1,q)\right\}V^k_{T_p}\\
&\quad
+c^n_{u,v}(q-1,q)V^k_{T_{q-1}}-c^n_{u,v}((q-d_n)\vee k_n,q)V^k_{T_{(q-d_n)\vee k_n-1}},
\end{align*}
hence the (piecewise) Lipschitz continuity of $u,v$ implies that $\sup_{1\leq q\leq N^n_t+1}|C^n_{u,v}(V)^k_{q}|\lesssim n^\gamma\sup_{0\leq s\leq t}|V^k_s|=O_p(n^\gamma)$.

Next consider the latter case. In this case, the BDG inequality, \ref{hypo:SA3} and Lemma \ref{HJYlem2.2} yield
\begin{align*}
E_0\left[\sup_{1\leq q\leq N^n_t+1}|C^n_{u,v}(\mathfrak{E})^k_{q}|^4\right]
\leq\sum_{q=1}^{N^n_t+1}E_0\left[|C^n_{u,v}(\mathfrak{E})^k_{q}|^4\right]
\lesssim(N^n_t+1)k_n^{-4}d_n^2=O_p(n^{2\gamma}),
\end{align*}
hence the Markov inequality implies that $\sup_{1\leq q\leq N^n_t+1}|C^n_{u,v}(\mathfrak{E})^k_{q}|=O_p(n^\gamma)$.
\end{proof}

Now we turn to the main body of the proof of Proposition \ref{approx}. 
\begin{lem}\label{Xirep}
Under the assumptions of Proposition \ref{approx}, it holds that
\begin{equation}\label{aim:Xirep}
\sup_{0\leq t\leq T}\left|\Xi^{(k,l)}_{u,v}(U,V)^n_t-\mathbb{L}^{(k,l)}_{u,v}(U,V)^n_t-\psi_2^{-1}\phi_{u,v}(0)[U^k,V^l]_t\right|=o_p(n^{-1/4})
\end{equation}
as $n\to\infty$ for any $k,l\in\{1,\dots,d\}$, $U,V\in\{X,\mathfrak{E}\}$ and $u,v\in\{g,g'\}$.
\end{lem}

\begin{proof}
Similarly to the proof of Proposition \ref{synchronization}, it suffices to prove \eqref{aim:Xirep} with replacing $\sup_{0\leq t\leq T}$ by $\sup_{0\leq t\leq T\wedge R^n_\alpha}$, where $R^n_\alpha$ is defined by \eqref{def:Rn}. 

First we show that $\Xi^{(k,l)}_{u,v}(U,V)^n_t=\widetilde{\Xi}^{(k,l)}_{u,v}(U,V)^n_t+o_p(n^{-1/4})$ uniformly in $t\in[0,T\wedge R^n_\alpha]$, where
\begin{equation}\label{Xi.tilde}
\widetilde{\Xi}^{(k,l)}_{u,v}(U,V)^n_t=\frac{1}{\psi_2k_n}\sum_{i=k_n}^\infty\mathbb{U}_{i,t}\mathbb{V}_{i,t},
\end{equation}
$\mathbb{U}_{i,t}=\sum_{p=k_n}^{N^n_t+1}u^n_{p-i}U^k(I_p)$ and $\mathbb{V}_{i,t}$ is defined analogously. Thanks to \ref{hypo:W}, we have
\begin{align*}
\Xi^{(k,l)}_{u,v}(U,V)^n_t-\widetilde{\Xi}^{(k,l)}_{u,v}(U,V)^n_t
&=\frac{1}{\psi_2 k_n}\sum_{i=N^n_t-d_n+1}^{N^n_t-k_n+1}\breve{U}(u)^k_i\breve{V}(v)^l_i
+\frac{1}{\psi_2 k_n}\sum_{i=N^n_t-d_n+1}^\infty\mathbb{U}_{i,t}\mathbb{V}_{i,t}
+o_p(n^{-1/4})\\
&=:\mathbb{A}_{1,t}+\mathbb{A}_{2,t}+o_p(n^{-1/4})
\end{align*}
uniformly in $t\in[0,T\wedge R^n_\alpha]$. The H\"older inequality, Lemma \ref{lemma:estimate}, \eqref{estimate:beta}, \eqref{SC3} and \eqref{est.xi} imply that
\begin{align*}
 E\left[\sup_{0\leq t\leq T\wedge R^n_\alpha}\left|\mathbb{A}_{1,t}\right|\right]
&\leq\frac{1}{\psi_2k_n}E\left[\sup_{0\leq t\leq T\wedge R^n_\alpha}\sum_{i=N^n_t-d_n+1}^{N^n_t-k_n+1}\left|\breve{U}(u)^k_i\right|\left|\breve{V}(v)^l_i\right|\right]\\
&\leq\psi_2^{-1}k_n^{-1} d_n^{1-2/\Gamma}E\left[\sup_{0\leq t\leq T\wedge R^n_\alpha}\left\{\sum_{i=k_n}^{N^n_t+d_n}\left|\breve{U}(u)^k_i\right|^{\Gamma/2}\left|\breve{V}(v)^l_i\right|^{\Gamma/2}\right\}^{2/\Gamma}\right]\\
&\lesssim k_n^{-1} d_n^{1-2/\Gamma}n^{2/\Gamma}k_n\bar{r}_n
=O(n^{1/\Gamma+1/2-\xi+\gamma(1-2/\Gamma)})=o(n^{-1/4}),
\end{align*}
hence we obtain $\sup_{0\leq t\leq T\wedge R^n_\alpha}\left|\mathbb{A}_{1,t}\right|=o_p(n^{-1/4})$. On the other hand, noting that $\mathbb{A}_{2,t}=\frac{1}{\psi_2 k_n}\sum_{i=N^n_t-d_n+1}^{N^n_t+d_n}\mathbb{U}_{i,t}\mathbb{V}_{i,t}+o_p(n^{-1/4})$ uniformly in $t\in[0,T\wedge R^n_\alpha]$ due to \ref{hypo:W}, we similarly deduce $\sup_{0\leq t\leq T\wedge R^n_\alpha}\left|\mathbb{A}_{2,t}\right|=o_p(n^{-1/4})$.

Next, a direct computation shows $\widetilde{\Xi}^{(k,l)}_{u,v}(U,V)^n_t=\sum_{p,q=k_n}^{N^n_t+1}c_{u,v}^n(p,q)U^k(I_p)V^l(I_q)$, hence \ref{hypo:W} implies that
\begin{align*}
\widetilde{\Xi}^{(k,l)}_{u,v}(U,V)^n_t
&=\mathbb{L}^{(k,l)}_{u,v}(U,V)^n_t
+\sum_{p=k_n}^{N^n_t+1} c^n_{u,v}(p,p)U^k(I_{p})V^l(I_{p})+o_p(n^{-1/4})
\end{align*}
uniformly in $t\in[0,T\wedge R^n_\alpha]$.
Therefore, the proof is completed once we prove
\begin{equation}\label{Xirep.aim4}
\sup_{0\leq t\leq T\wedge R^n_\alpha}\left|\mathbb{B}_{t}-\psi_2^{-1}\phi_{u,v}(0)[U^k,V^l]_t\right|=o_p(n^{-1/4}),
\end{equation}
where $\mathbb{B}_t=\sum_{p=k_n}^{N^n_t+1} c^n_{u,v}(p,p)U^k(I_{p})V^l(I_{p})$. 
If $U=A$ or $V=A$, $(\ref{Xirep.aim4})$ holds true since $E[\sup_{0\leq t\leq T\wedge R^n_\alpha}|\mathbb{B}_{t}|]\lesssim\sqrt{\bar{r}_n}T=o(n^{-1/4})$ and $[U^k,V^l]=0$. Otherwise,  $U^kV^l-[U^k,V^l]$ is an $(\mathcal{F}_t)$-martingale, hence a standard martingale argument yields 
$\mathbb{B}_{t}=\sum_{p=k_n}^{N^n_t+1} c^n_{u,v}(p,p)[U^k,V^l](I_p)+o_p(n^{-1/4})$ 
uniformly in $t\in[0,T\wedge R^n_\alpha]$. Moreover, since $c^n_{u,v}(p,p)=\psi_2^{-1}\phi_{u,v}(0)+O_p(k_n^{-1})$ uniformly in $p\geq d_n$ and $\sum_{p=k_n}^{d_n} c^n_{u,v}(p,p)[U^k,V^l](I_p)=o_p(n^{-1/4})$ as well as $[U^k,V^l]_{T_{d_n}}=o_p(n^{-1/4})$, we obtain 
$\mathbb{B}_{t}=\psi_2^{-1}\phi_{u,v}(0)[U^k,V^l]_t+o_p(n^{-1/4})$ uniformly in $t\in[0,T\wedge R^n_\alpha]$ due to Lemma \ref{sup.epsilon}. Thus we complete the proof.
\end{proof}


In the remaining tasks to prove Proposition \ref{approx}, the most sophisticated part is the proof of the negligibility of the term $\mathbb{L}^{(k,l)}_{u,v}(A,M)^n$. If the process $a_s$ is a constant and $T_p$'s are independent of $M$, $\mathbb{L}^{(k,l)}_{u,v}(A,M)^n$ is a martingale with respect to an appropriate filtration, so this is an easy task. 
Dropping the assumption that $a_s$ is a constant is not difficult. Here the problem is that $T_p$ could depend on $M$. 
In fact, in a pure diffusion setting this dependence could cause the non-negligibility of the approximation error of the realized covariance due to the drift term (see \cite{Fu2010b} or \cite{LMRZZ2014} for details). Unlike such a setting, we can prove the negligibility of such a term without ruling out the dependence between $(T_p)$ and $M$, as long as \ref{hypo:A4} is satisfied: 
\begin{lem}\label{HYlem13}
Suppose that $V\in\{A,M,\mathfrak{E}\}$, $u,v\in\{g,g'\}$ and $k,l\in\{1,\dots,d\}$. Under the assumptions of Proposition \ref{approx}, we have $\sup_{0\leq t\leq T}|\mathbb{M}_{u,v}^{(k,l)}(V,A)^n_t|=o_p(n^{-1/4})$ and $\sup_{0\leq t\leq T}|\mathbb{M}_{u,v}^{(k,l)}(A,V)^n_t|=o_p(n^{-1/4})$.
\end{lem}

\begin{proof}
For the proof we may replace $\sup_{0\leq t\leq T}$ by $\sup_{0\leq t\leq T\wedge R^n_\alpha}$ similarly to the above. 

First, since \ref{hypo:SA2} and $(\ref{SA4})$ yield $\sup_{0\leq t\leq T\wedge R^n_\alpha}\left|\mathbb{M}^{(k,l)}_{u,v}(A,A)^n_t\right|\lesssim d_n\bar{r}_n=o_p(n^{-1/4})$, the lemma holds true for $V=A$. Therefore, it suffices to consider the case that $V\in\{M,\mathfrak{E}\}$. In this case $E\left[C^n_{u,v}(A)^k_qV^l(I_q)\big|\mathcal{F}_{T_{q-1}}\right]=0$ and
\begin{align*}
\sqrt{n}\sum_{q=k_n}^{N^n_{T\wedge R^n_\alpha}+1}E\left[\left|C^n_{u,v}(A)^k_qV^l(I_q)\right|^2\big|\mathcal{F}_{T_{q-1}}\right]
=O_p\left(\sqrt{n}(d_n\bar{r}_n)^2\right)=O_p(n^{3/2+2\gamma-2\xi})=o_p(1)
\end{align*}
by \eqref{SA4}, \ref{hypo:SA2}--\ref{hypo:SA3} and \eqref{est.xi}, hence Lemma \ref{useful} yields $\sup_{0\leq t\leq T\wedge R^n_\alpha}|\mathbb{M}_{u,v}^{(k,l)}(A,V)^n_t|=o_p(n^{-1/4})$.

Now we prove $\sup_{0\leq t\leq T\wedge R^n_\alpha}|\mathbb{M}_{u,v}^{(k,l)}(V,A)^n_t|=o_p(n^{-1/4})$. First, by \eqref{SA4}, \ref{hypo:SA2}--\ref{hypo:SA3} and the Doob inequality, there is a constant $K$ such that
\begin{equation}\label{maest}
E\left[\left|C^n_{u,v}(V)^k_q\right|^2|\mathcal{F}_{T_{(q-d_n-1)_+}}\right]\leq Kd_n\bar{r}_n
\end{equation}
for any $q,n$. Combining this estimate with \eqref{SA4}, \eqref{SC3} and \eqref{SA5}, we obtain
\begin{align*}
&E\left[\sup_{0\leq t\leq T\wedge R^n_\alpha}\left|\sum_{q=k_n}^{N^n_t+1} C^n_{u,v}(V)^k_q\left\{A^l(I_q)-a^l_{T_{q-1}}|I_q|\right\}\right|\right]\\
\leq&E\left[\sum_{q=k_n}^{N^n_{T\wedge R^n_\alpha}+1}\left|C^n_{u,v}(V)^k_q\right|\left\{\int_{T_{q-1}}^{T_{q-1}+2\bar{r}_n}E\left[\left|a^l_s-a^l_{T_{q-1}}\right||\mathcal{F}_{T_{q-1}}\right]\mathrm{d}s\right\}\right]
\lesssim n\sqrt{d_n\bar{r}_n}\bar{r}_n^{1+\varpi/2}
=o(n^{-1/4}).
\end{align*}
Therefore, we have
\begin{equation}\label{shiftA}
\sup_{0\leq t\leq T\wedge R^n_\alpha}\left|\mathbb{M}^{(k,l)}_{u,v}(V,A)^n_t-\sum_{q=k_n}^{N^n_t+1} C^n_{u,v}(V)^k_q a^l_{T_{q-1}}|I_q|\right|=o_p(n^{-1/4}).
\end{equation}

Next we show that
\begin{equation}\label{usefularg}
\sup_{0\leq t\leq T\wedge R^n_\alpha}\left|\mathbb{M}^{(k,l)}_{u,v}(V,A)^n_t-n^{-1}\sum_{q=k_n}^{N^n_t+1}C^n_{u,v}(V)^k_q a^l_{T_{q-1}}G_{T_{q-1}}\right|=o_p(n^{-1/4}).
\end{equation}
\eqref{SA4}, the boundedness of $a$, \eqref{maest} and \eqref{SC3} yield
\begin{align*}
E\left[\sqrt{n}\sum_{q=k_n}^{N^n_{T\wedge R^n_\alpha}+1}\left|C^n_{u,v}(V)^k_q a^l_{T_{q-1}}|I_q|\right|^2\right]
&\lesssim \sqrt{n}\cdot n\cdot d_n\bar{r}_n\cdot \bar{r}_n^2
=O(n^{2+\gamma-3\xi})=o(1).
\end{align*}
Therefore, Lemma \ref{useful} implies that
\begin{equation*}
\sup_{0\leq t\leq T\wedge R^n_\alpha}\left|\mathbb{M}^{(k,l)}_{u,v}(V,A)^n_t-n^{-1}\sum_{q=k_n}^{N^n_t+1}C^n_{u,v}(V)^k_q a^l_{T_{q-1}}E\left[n|I_q|\big|\mathcal{F}^{(0)}_{T_{q-1}}\right]\right|=o_p(n^{-1/4}),
\end{equation*}
hence by \ref{hypo:A4}, Lemma \ref{supC}, the boundedness of $a$, \eqref{SA4} and \eqref{est.xi} we obtain 
\begin{equation*}
\sup_{0\leq t\leq T\wedge R^n_\alpha}\left|\mathbb{M}^{(k,l)}_{u,v}(V,A)^n_t-n^{-1}\sum_{q=k_n}^{N^n_t+1}C^n_{u,v}(V)^k_q a^l_{T_{q-1}}G^n_{T_{q-1}}\right|=o_p(n^{-1/4}).
\end{equation*}
Furthermore, since we have
\begin{equation*}
\sup_{0\leq t\leq T\wedge R^n_\alpha}\left|n^{-1}\sum_{q=k_n}^{N^n_t+1}C^n_{u,v}(V)^k_q a^l_{T_{q-1}}\left(G^n_{T_{q-1}}-G_{T_{q-1}}\right)\right|
\lesssim\sup_{0\leq t\leq T}\left|G^n_{t}-G_{t}\right|\cdot n^{-1}\sum_{q=k_n}^{N^n_{T\wedge R^n_\alpha}+1}\left|C^n_{u,v}(V)^k_q\right|
\end{equation*}
and $n^{-1}\sum_{q=k_n}^{N^n_{T\wedge R^n_\alpha}+1}\left|C^n_{u,v}(V)^k_q\right|=O_p(\sqrt{d_n\bar{r}_n})$ due to \eqref{maest} and \eqref{SC3}, we obtain \eqref{usefularg} by \ref{hypo:A4} and \eqref{est.xi}.

Now we show that
\begin{equation}\label{contG}
\sup_{0\leq t\leq T\wedge R^n_\alpha}\left|\mathbb{M}^{(k,l)}_{u,v}(V,A)^n_t-n^{-1}\sum_{q=k_n}^{N^n_t+1}C^n_{u,v}(V)^k_qF_{T_{(q-d_n-1)_+}}\right|=o_p(n^{-1/4}),
\end{equation}
where $F=a^lG$. First, by $(\ref{SA5})$, \ref{hypo:SA1}, \ref{hypo:A4v} as well as a standard localization procedure, for each $j\geq1$ there are a bounded $(\mathcal{F}^{(0)}_t)$-progressively measurable process $F(j)$, $(\mathcal{F}^{(0)}_t)$-stopping time $\rho_j$ and a constant $K_j$ such that $\rho_j\uparrow\infty$ as $j\to\infty$, $F_t=F(j)_t$ if $t<\rho_j$ and 
$E\left[|F(j)_{t_1}-F(j)_{t_2}|^2|\mathcal{F}_{t_1\wedge t_2}\right]
\leq K_jE\left[|t_1-t_2|^\varpi|\mathcal{F}_{t_1\wedge t_2}\right]$ 
for any $(\mathcal{F}^{(0)}_t)$-stopping times $t_1,t_2$ bounded by $T$. Then, for a fixed $j$, the Schwarz inequality, $(\ref{maest})$, the boundedness of $a$, \eqref{SA4} and \eqref{SC3} yield
\begin{align*}
&E\left[n^{-1}\sup_{0\leq t\leq T\wedge R^n_\alpha}\left|\sum_{q=k_n}^{N^n_t+1}C^n_{u,v}(V)^k_q\left(F_{T_{q-1}}-F_{T_{(q-d_n-1)_+}}\right)\right|;T<\rho_j\right]\\
&\lesssim\sqrt{d_n\bar{r}_n}(d_n\bar{r}_n)^{\varpi/2}
=O_p\left(n^{-\frac{1+\varpi}{2}\left(\xi-\frac{1}{2}-\gamma\right)}\right)=o_p(n^{-1/4}).
\end{align*}
Since $\lim_{j\to\infty}P(\rho_j\leq T)=0$, we conclude that \eqref{contG} holds true by the Markov inequality.

After all, it suffices to show that $\sup_{0\leq t\leq T\wedge R^n_\alpha}|\mathbb{A}_t|\to^p0$ as $n\to\infty$, where
$\mathbb{A}_t=n^{-3/4}\sum_{q=k_n}^{N^n_t+1}C^n_{u,v}(V)^k_qF_{T_{(q-d_n-1)_+}}.$ 
Set $H^p=n^{-3/4}\sum_{q=p+1}^{p+d_n}c^n_{u,v}(p,q)F_{T_{(q-d_n-1)_+}}$. Then, by construction $H^p$ is $\mathcal{F}_{T_{p-1}}$-measurable and we have $\mathbb{A}_t=\sum_{p=k_n}^{N^n_t}H^pV^k(I_p)$. Therefore, 
by Lemma \ref{useful} it is enough to prove 
$\sum_{p=k_n}^{N^n_{T\wedge R^n_\alpha}+1}E\left[\left|H^pV^k(I_p)\right|^2\big|\mathcal{F}_{T_{p-1}}\right]\to^p0,$ 
which follows from \eqref{est.xi} and the fact that $|H^p|\lesssim n^{-3/4}d_n\sup_{0\leq t\leq T_{p-1}}|F_t|$ uniformly in $p$. Thus we complete the proof.
\end{proof}

\begin{proof}[\upshape{\bfseries{Proof of Proposition \ref{approx}}}]
Note that $\phi_{g,g'}(0)=\phi_{g',g}(0)=0$ due to integration by parts and \ref{hypo:W}. Therefore, in the light of Lemmas \ref{Xirep}--\ref{HYlem13} as well as \eqref{alpha.infinity} the proof is completed once we show that
\begin{equation}\label{EEconv}
n^{1/4}\sup_{0\leq t\leq T}\left|\frac{1}{2k_n^2}[Y,Y]^{n,kl}_t-[\mathfrak{E}^k,\mathfrak{E}^l]_t\right|\to^p0
\end{equation}
as $n\to\infty$. First, it can easily be shown that
\begin{align*}
\frac{1}{2k_n^2}[Y,Y]^{n,kl}_t=\frac{1}{2k_n^2}\sum_{p=1}^{N^n_t}(\epsilon^k_{\tau^k_p}\epsilon^l_{\tau^l_{p}}+\epsilon^k_{\tau^k_{p-1}}\epsilon^l_{\tau^l_{p-1}})+O_p(n^{-1/2})
\end{align*}
uniformly in $t\in[0,T]$. On the other hand, we can write $[\mathfrak{E}^k,\mathfrak{E}^l]_t=k_n^{-2}\sum_{p=1}^{N^n_t+1}\epsilon^k_{\tau^k_p}\epsilon^l_{\tau^l_p}1_{\{\tau^k_p=\tau^l_p\leq t\}}$, hence
\begin{align*}
&\sup_{0\leq t\leq T}\left|\frac{1}{2k_n^2}\sum_{p=1}^{N^n_t}(\epsilon^k_{\tau^k_p}\epsilon^l_{\tau^l_{p}}+\epsilon^k_{\tau^k_{p-1}}\epsilon^l_{\tau^l_{p-1}})-[\mathfrak{E}^k,\mathfrak{E}^l]_t\right|\\
\leq&\sup_{0\leq t\leq T}\left|\frac{1}{k_n^2}\sum_{p=1}^{N^n_t}\epsilon^k_{\tau^k_p}\epsilon^l_{\tau^l_{p}}1_{\{\tau^k_p\neq\tau^l_p\}}\right|+\frac{1}{2k_n^2}\left|\epsilon^k_{\tau^k_0}\epsilon^l_{\tau^l_{0}}\right|+\frac{1}{k_n^2}\sup_{0\leq p\leq N^n_T+1}\left|\epsilon^k_{\tau^k_{p}}\epsilon^l_{\tau^l_{p}}\right|
=:\Gamma_1+\Gamma_2+\Gamma_3.
\end{align*}
The Doob inequality yields $\Gamma_1=O_p(n^{-1/2})$, while we obviously have $\Gamma_2=O_p(n^{-1})$. Furthermore, Lemma \ref{sup.epsilon} implies that $\Gamma_3=o_p(k_n^{-2}n^{1/2})=o_p(n^{-1/4})$.
This yields $(\ref{EEconv})$.
\end{proof}

\subsubsection{Proof of Lemma \ref{lemenergy}}\label{sec.contCLT}

Let $(U,u),(V,v),(\check{U},\check{u}),(\check{V},\check{v})\in\{(M,g),(\mathfrak{E},g')\}$ and set 
$$\mathfrak{V}^n_t=\sqrt{n}\sum_{q=k_n}^{N^n_t+1}C^n_{u,v}(U)^k_qC^n_{\check{u},\check{v}}(\check{U})^{k'}_qE\left[V^l(I_q)\check{V}^{l'}(I_q)\big|\mathcal{F}_{T_{q-1}}\right].$$ 
It suffices to compute the limiting variable of $\mathfrak{V}^n_t$ explicitly.

Set $H_q=\sqrt{n}C^n_{u,v}(U)^k_qC^n_{\check{u},\check{v}}(\check{U})^{k'}_q$. Then, for any $r\in[1,2]$ there is a positive constant $K_r$ such that
\begin{align}
E\left[\left|H_{q}\right|^r\big|\mathcal{F}_{T_{(q-d_n-1)_+}}\right]&\leq K_r\left(\sqrt{n}d_n\bar{r}_n\right)^{r}\label{maestc}
\end{align}
for every $q$ by the Schwarz and BDG inequalities, \ref{hypo:SA2}--\ref{hypo:SA3} and \eqref{SA4}. This estimate will often be used in the following. Moreover, we can rewrite $\mathfrak{V}^n_t$ as $\mathfrak{V}^n_t=\sum_{q=k_n}^{N^n_t+1}H_qE\left[[V^l,\check{V}^{l'}](I_q)\big|\mathcal{F}_{T_{q-1}}\right]$ since $V^l\check{V}^{l'}-[V^l,\check{V}^{l'}]$ is an $(\mathcal{F}_t)$-martingale.

Now we separately consider the following three cases:

\noindent\textit{Case} 1: $V=\check{V}=M$.
We fix $\alpha>0$ for a while. First, since the boundedness of $\sigma$, \eqref{SA4}, \eqref{maestc} and \eqref{SC3} yield
\begin{align*}
E\left[\sum_{q=k_n}^{N^n_{t\wedge R^n_\alpha}+1}\left|H_q\right|^2E\left[\left|[M^l,M^{l'}](I_q)\right|^2\big|\mathcal{F}_{T_{q-1}}\right]\right]
&\lesssim\bar{r}_n^2E\left[\sum_{q=k_n}^{N^n_{t\wedge R^n_\alpha}+d_n+1}\left|H_q\right|^2\right]\\
&\lesssim\bar{r}_n^2\cdot n\cdot n(d_n\bar{r}_n)^2
=O(n^{3+2\gamma-4\xi})=o(1),
\end{align*}
Lemma \ref{useful} implies that $\mathfrak{V}^n_{t\wedge R^n_\alpha}=\sum_{q=k_n}^{N^n_{t\wedge R^n_\alpha}+1}H_q[M^l,M^{l'}](I_q)+o_p(1)$. Next, since $[M^l,M^{l'}]_t=\int_0^t\Sigma^{ll'}_s\mathrm{d}s$, by a similar argument to the proof of \eqref{shiftA} (using \eqref{maestc} instead of \eqref{maest}) we can show that
\begin{align*}
E\left[\left|\sum_{q=k_n}^{N^n_{t\wedge R^n_\alpha}+1} H_{q}\left\{[M^l,M^{l'}](I_q)-\Sigma^{ll'}_{T_{q-1}}|I_q|\right\}\right|\right]
\lesssim n\cdot \sqrt{n}d_n\bar{r}_n\cdot\bar{r}_n^{1+\varpi/2}
=O(n^{2+\gamma-(2+\varpi/2)\xi}).
\end{align*}
Hence $(\ref{est.xi})$ yields 
\begin{equation}\label{NNshift}
\mathfrak{V}^n_{t\wedge R^n_\alpha}=\sum_{q=k_n}^{N^n_{t\wedge R^n_\alpha}+1} H_{q}\Sigma^{ll'}_{T_{q-1}}|I_q|+o_p(1).
\end{equation}
Moreover, similar arguments to the proofs of \eqref{usefularg} and \eqref{contG} (using \eqref{maestc} instead of \eqref{maest}) yield 
\begin{equation}\label{lhg}
\mathfrak{V}^n_{t\wedge R^n_\alpha}=n^{-1}\sum_{q=k_n}^{N^n_{t\wedge R^n_\alpha}+1}H_{q}F_{T_{(q-d_n-1)_+}}+o_p(1),
\end{equation}
where $F=\Sigma^{ll'}G$. 
$(\ref{lhg})$ yields 
$\mathfrak{V}^n_{t\wedge R^n_\alpha}=\sum_{p,p'=k_n}^{N^n_{t\wedge R^n_\alpha}}\widetilde{H}_{p,p'}U^{k}(I_p)\check{U}^{k'}(I_{p'})+o_p(1),$ 
where
\begin{align*}
\widetilde{H}_{p,p'}=n^{-1/2}\sum_{q=p\vee p'+1}^{p\wedge p'+d_n}c^n_{u,v}(p,q)c^n_{\check{u},\check{v}}(p',q)F_{T_{(q-d_n-1)_+}}.
\end{align*}
Therefore, we have the following decomposition:
\begin{align*}
\mathfrak{V}^n_{t\wedge R^n_\alpha}
&=\left[\sum_{k_n\leq p<p'\leq N^n_{t\wedge R^n_\alpha}}+\sum_{k_n\leq p'<p\leq N^n_{t\wedge R^n_\alpha}}+\sum_{k_n\leq p=p'\leq N^n_{t\wedge R^n_\alpha}}\right]\widetilde{H}_{p,p'}U^{k}(I_p)\check{U}^{k'}(I_{p'})+o_p(1)\\
&=:\mathbb{I}+\mathbb{II}+\mathbb{III}+o_p(1).
\end{align*}
We first prove $\mathbb{I}=o_p(1)$. Fix $L>0$, and we further decompose $\mathbb{I}$ as
\begin{align*}
\mathbb{I}=\sum_{k_n\leq p<p'\leq N^n_{t\wedge R^n_\alpha}}\left(\widetilde{H}_{p,p'}1_{\{|\widetilde{H}_{p,p'}|\leq L\}}+\widetilde{H}_{p,p'}1_{\{|\widetilde{H}_{p,p'}|>L\}}\right)U^{k}(I_p)\check{U}^{k'}(I_{p'})
=:\mathbb{I}'(L)+\mathbb{I}''(L).
\end{align*}
First we show $\mathbb{I}'(L)=o_p(1)$ as $n\to\infty$. Since $\widetilde{H}_{p,p'}1_{\{|\widetilde{H}_{p,p'}|\leq L\}}U^k(I_p)$ is $\mathcal{F}_{T_{p'-1}}$-measurable for $p<p'$ and $E[\check{U}^{k'}(I_{p'})|\mathcal{F}_{T_{p'-1}}]=0$ and we have
\begin{align*}
\left|\mathbb{I}'(L)\right|\leq\sup_{k_n\leq j\leq N^n_{t\wedge R^n_\alpha}+1}\left|\sum_{p'=k_n}^{j}\sum_{p:k_n\leq p< p'}\widetilde{H}_{p,p'}1_{\{|\widetilde{H}_{p,p'}|\leq L\}}U^k(I_p)\check{U}^{k'}(I_{p'})\right|,
\end{align*}
by the Lenglart inequality it suffices to prove
\begin{align*}
\Delta_n:=E\left[\sum_{p'=k_n}^{N^n_{t\wedge R^n_\alpha}+1}\left|\sum_{p:k_n\leq p< p'}\widetilde{H}_{p,p'}1_{\{|\widetilde{H}_{p,p'}|\leq L\}}U^k(I_p)\check{U}^{k'}(I_{p'})\right|^2\right]\to0.
\end{align*}
The boundedness of $\sigma$ and $\Upsilon$, \eqref{SA4}, \eqref{SC3} and the fact that $\widetilde{H}_{p,p'}=0$ if $|p-p'|\geq d_n$ as well as $\widetilde{H}_{p,p'}$ is $\mathcal{F}_{T_{p\wedge p'-1}}$-measurable yield
\begin{align*}
\Delta_n
&\lesssim\bar{r}_nE\left[\sum_{p'=k_n}^{N^n_{t\wedge R^n_\alpha}+d_n+1}\sum_{p:k_n\vee(p'-d_n+1)\leq p< p'}\left|\widetilde{H}_{p,p'}1_{\{|\widetilde{H}_{p,p'}|\leq L\}}U^k(I_p)\right|^2\right]
\lesssim L^2nd_n\bar{r}_n^2,
\end{align*}
hence we obtain the desired result. Next we show $\lim_{L\to\infty}\limsup_nP(|\mathbb{I}''(L)|>0)=0$. 
First, since $|c^n_{\check{u},\check{v}}(p',q)|\lesssim1$ and $\sum_{q=-\infty}^\infty|c^n_{u,v}(p,q)|
\leq\frac{k_n}{\psi_2}\left(\frac{1}{k_n}\sum_{i=-\infty}^\infty|u^n_{i}|\right)\left(\frac{1}{k_n}\sum_{q=-\infty}^\infty|v^n_{q}|\right)\lesssim k_n$ by [W], there is a constant $K>0$ such that $|\widetilde{H}_{p,p'}|\leq K\sup_{0\leq s\leq t}|F_s|$ if $1\leq p<p'\leq N^n_{t\wedge R^n_\alpha}$. So, noting that $|\mathbb{I}''(L)|\leq\sum_{k_n\leq p<p'\leq N^n_{t\wedge R^n_\alpha}}|\widetilde{H}_{p,p'}|1_{\{|\widetilde{H}_{p,p'}|>L\}}|U^{k}(I_p)\check{U}^{k'}(I_{p'})|$, we obtain $\limsup_nP(|\mathbb{I}''(L)|>0)\leq P(\sup_{0\leq s\leq t}|F_s|>L/K)$. 
This yields the desired result because $F$ is c\`adl\`ag. Consequently, we conclude that $\mathbb{I}=o_p(1)$ as $n\to\infty$.

By symmetry we also have $\mathbb{II}=o_p(1)$ as $n\to\infty$. Now we consider $\mathbb{III}$. First, a similar argument to the proof of \eqref{contG} yields
\begin{align*}
\mathbb{III}=\sum_{p=k_n}^{N^n_{t\wedge R^n_\alpha}}\left[n^{-1/2}\sum_{q=p+1}^{p+d_n}c^n_{u,v}(p,q)c^n_{\check{u},\check{v}}(p,q)\right]F_{T_{p-1}}U^{k}(I_p)\check{U}^{k'}(I_{p})+o_p(1).
\end{align*}
Moreover, we have $n^{-1/2}\sum_{q=p+1}^{p+d_n}c^n_{u,v}(p,q)c^n_{\check{u},\check{v}}(p,q)=\theta\psi_2^{-2}\int_0^\infty\phi_{u,v}(y)\phi_{\check{u},\check{v}}(y)\mathrm{d}y+O_p(k_n^{-1})$ uniformly in $p\geq d_n$ by \ref{hypo:W}, hence we obtain
\begin{align*}
\mathbb{III}=\left(\frac{\theta}{\psi_2^2}\int_0^\infty\phi_{u,v}(y)\phi_{\check{u},\check{v}}(y)\mathrm{d}y\right)\sum_{p=k_n}^{N^n_{t\wedge R^n_\alpha}}F_{T_{p-1}}U^{k}(I_p)\check{U}^{k'}(I_{p})+o_p(1).
\end{align*}

Now combining these results with \eqref{alpha.infinity}, we conclude that
\begin{align*}
\mathfrak{V}^n_t=\left(\frac{\theta}{\psi_2^2}\int_0^\infty\phi_{u,v}(y)\phi_{\check{u},\check{v}}(y)\mathrm{d}y\right)\sum_{p=k_n}^{N^n_{t}}F_{T_{p-1}}U^{k}(I_p)\check{U}^{k'}(I_{p})+o_p(1).
\end{align*}
Therefore, in the light of Theorem VI-6.22 of \cite{JS}, the limiting variable of $\mathfrak{V}^n_t$ can be computed explicitly once we show that
\begin{align}
&\textstyle\sup_{0\leq t\leq T}\left|\sum_{p=k_n}^{N^n_{t}}M^{k}(I_p)M^{k'}(I_{p})-[M^k,M^{k'}]_t\right|\to^p0,\label{MMlimit}\\
&\textstyle\sup_{0\leq t\leq T}\left|\sum_{p=k_n}^{N^n_{t}}M^{k}(I_p)\mathfrak{E}^{k'}(I_{p})\right|\to^p0,\label{MElimit}\\
&\textstyle\sup_{0\leq t\leq T}\left|\sum_{p=k_n}^{N^n_{t}}\mathfrak{E}^{k}(I_p)\mathfrak{E}^{k'}(I_{p})-\int_0^t\frac{\widetilde{\Upsilon}^{kk'}_s}{G_s}\mathrm{d}s\right|\to^p0\label{EElimit}
\end{align}
for any $T>0$. 
\eqref{MMlimit}--\eqref{MElimit} can be verified by standard martingale arguments based on Lemma \ref{useful}. On the other hand, another standard martingale argument yields  
$\sum_{p=1}^{N^n_{t}}\mathfrak{E}^{k}(I_p)\mathfrak{E}^{k'}(I_{p})
=\frac{1}{k_n^2}\sum_{p=1}^{N^n_t}\Upsilon^{kk'}_{\tau^k_p}1_{\{\tau^k_p=\tau^{k'}_p\}}+o_p(1)$ 
uniformly in $t\in[0,T]$. Then, using the boundedness and the c\`adl\`ag property of $\Upsilon$ as well as Lemma \ref{HJYlem2.2}, we can easily show that
\begin{align*}
\sum_{p=1}^{N^n_{t}}\mathfrak{E}^{k}(I_p)\mathfrak{E}^{k'}(I_{p})
=\frac{1}{k_n^2}\sum_{p=1}^{N^n_t+1}\Upsilon^{kk'}_{T_{p-1}}1_{\{\tau^k_p=\tau^{k'}_p\}}+o_p(1)
\end{align*}
uniformly in $t\in[0,T]$, hence by Lemma \ref{useful} and \ref{hypo:A4} we obtain 
$\sum_{p=1}^{N^n_{t}}\mathfrak{E}^{k}(I_p)\mathfrak{E}^{k'}(I_{p})
=\frac{1}{k_n^2}\sum_{p=1}^{N^n_t+1}\widetilde{\Upsilon}^{kk'}_{T_{p-1}}+o_p(1)$ 
uniformly in $t\in[0,T]$. Now \eqref{EElimit} follows from Theorem VI-6.22 of \cite{JS} and Lemma \ref{HJYlem2.2}) (note that the convergence $N^n_t/n\to^p\int_0^t1/G_s\mathrm{d}s$ holds true uniformly in $t\in[0,T]$ because the limiting process is nondecreasing).

\noindent\textit{Case} 2: $V=\check{V}=\mathfrak{E}$. Again fix $\alpha>0$. In this case we have $E\left[[V^l,V^{l'}](I_q)\big|\mathcal{F}_{T_{q-1}}\right]=E\left[\Upsilon^{ll'}_{\tau^l_p}1_{\{\tau^l_p=\tau^{l'}_p\}}\big|\mathcal{F}_{T_{q-1}}\right]$, hence a similar argument to the proof of \eqref{NNshift} yields $\mathfrak{V}^n_{t\wedge R^n_\alpha}=\frac{1}{k_n^2}\sum_{q=k_n}^{N^n_{t\wedge R^n_\alpha}+1} H_{q}\Upsilon^{ll'}_{T_{q-1}}1_{\{\tau^l_q=\tau^{l'}_q\}}+o_p(1)$, and a similar argument to the proof of \eqref{usefularg} (using \eqref{maestc} instead of \eqref{maest}) implies that $\mathfrak{V}^n_{t\wedge R^n_\alpha}=\frac{1}{k_n^2}\sum_{q=k_n}^{N^n_{t\wedge R^n_\alpha}+1} H_{q}\widetilde{\Upsilon}^{ll'}_{T_{q-1}}+o_p(1)$. Now we can apply similar arguments to those of Case 1 after the equation \eqref{lhg}, and thus we obtain 
\begin{align*}
\mathfrak{V}^n_t=\left(\frac{1}{\psi_2^2\theta}\int_0^\infty\phi_{u,v}(y)\phi_{\check{u},\check{v}}(y)\mathrm{d}y\right)\sum_{p=k_n}^{N^n_t}\widetilde{\Upsilon}^{ll'}_{T_{p-1}}U^k(I_p)\check{U}^{k'}(I_p)+o_p(1).
\end{align*}
Now the proof is completed in a similar manner to the previous case.

\noindent\textit{Case} 3: $V\neq\check{V}$. In this case we have $[V^l,V^{l'}]=0$, hence it holds that $\mathfrak{V}^n_t\to^p0$.

Consequently, we complete the proof.\hfill$\Box$

\subsubsection{Proof of Lemma \ref{lemremain}}

(a) By \eqref{SA4}, \ref{hypo:SA2}--\ref{hypo:SA3}, the BDG inequality and \eqref{SC3} we have  
\begin{align*}
E\left[n\sum_{q=k_n}^{N^n_{t\wedge R^n_\alpha}+1}E\left[\left|C^n_{u,v}(U)^k_qV^l(I_q)\right|^4\big|\mathcal{F}_{T_{q-1}}\right]\right]
&\lesssim n\bar{r}_n^2E\left[\sum_{q=k_n}^{N^n_{t\wedge R^n_\alpha}+d_n+1}\left|C^n_{u,v}(U)^k_q\right|^4\right]\\
&\lesssim n\bar{r}_n^2\cdot n(d_n\bar{r}_n)^2
=O(n^{3+2\gamma-4\xi})=o(1).
\end{align*}
Therefore, the Markov inequality and \eqref{alpha.infinity} yield the desired result.

(b) Since $E[\mathfrak{E}^l(I_q)W^{j}(I_q)|\mathcal{F}_{T_{q-1}}]=0$, it is enough to consider the case that $V=M$. In this case a standard martingale argument yields
\begin{align*}
n^{1/4}\sum_{q=k_n}^{N^n_{t}+1}E\left[C^n_{u,v}(U)^k_qV^l(I_q)W^{j}(I_q)\big|\mathcal{F}_{T_{q-1}}\right]
&=n^{1/4}\sum_{q=k_n}^{N^n_{t}+1}C^n_{u,v}(U)^k_q[M^l,W^{j}](I_q)+o_p(1)\\
&=n^{1/4}\mathbb{M}^{(k,l)}_{u,v}(U,B)_{t}+o_p(1),
\end{align*}
where $B=([M^l,W^{j}])_{1\leq l\leq d}$. Therefore, noting that $B^l_t=\int_0^t\sigma^{lj}_s\mathrm{d}s$ and $\sigma$ satisfies \eqref{SA3}, Lemma \ref{HYlem13} yields the desired result.

(c) Since $N$ is orthogonal to $W$ and defined on $\mathcal{B}^{(0)}$, we have $E\left[C^n_{u,v}(U)^k_qV^l(I_q)N(I_q)\big|\mathcal{F}_{T_{q-1}}\right]=0$, which yields the desired result. \hfill$\Box$

\subsection{Proofs of the results from Section \ref{discussion:A4}}

\subsubsection{Proof of Proposition \ref{proplzz}}\label{proofproplzz}

By a localization procedure we may replace \ref{hypo:A1}--\ref{hypo:A4} by \ref{hypo:SA1}--\ref{hypo:SA4}, respectively. First, Lemma \ref{lemma:estimate}, \eqref{SA4} and the boundedness of $a$ yield
\begin{align*}
\frac{\sqrt{n}}{k_n^{3/2}}\sum_{i=k_n}^{N^n_t-k_n+1}(\widetilde{X}_{i,T})^3
=\frac{\sqrt{n}}{k_n^{3/2}}\sum_{i=k_n}^{N^n_t-k_n+1}(\widetilde{M}_{i,T})^3+o_p(1).
\end{align*}
Then, similarly to the proofs of Proposition \ref{synchronization} and \eqref{Xi.tilde}, we can deduce
\begin{align*}
\frac{\sqrt{n}}{k_n^{3/2}}\sum_{i=k_n}^{N^n_t-k_n+1}(\widetilde{X}_{i,T})^3
=\frac{\sqrt{n}}{k_n^{3/2}}\sum_{i=k_n}^\infty(\widehat{M}_{i,t})^3+o_p(1),
\end{align*}
where $\widehat{M}_{i,t}=\sum_{p=k_n}^\infty g^n_{p-i} M(I_p(t))$ and $I_p(t)=[T^{p-1}\wedge t,T_p\wedge t)$. 
Now, by It\^o's formula we deduce
\begin{align*}
\frac{\sqrt{n}}{k_n^{3/2}}\sum_{i=k_n}^\infty(\widehat{M}_{i,t})^3
=3\frac{\sqrt{n}}{k_n^{3/2}}\sum_{i=k_n}^\infty\int_0^t(\widehat{M}_{i,s})^2\mathrm{d}\widehat{M}_{i,s}
+3\frac{\sqrt{n}}{k_n^{3/2}}\sum_{i=k_n}^\infty\int_0^t\widehat{M}_{i,s}\mathrm{d}[\widehat{M}_{i,\cdot},\widehat{M}_{i,\cdot}]_s
=:\mathbf{I}_t+\mathbf{II}_t.
\end{align*}
Since $\mathbf{I}$ is a locally square-integrable martingale and its predictable quadratic variation satisfies
\begin{align*}
E\left[\langle\mathbf{I}\rangle_t\right]&=9\frac{n}{k_n^{3}}\sum_{i,j=k_n}^\infty E\left[\int_0^t(\widehat{M}_{i,s})^2(\widehat{M}_{j,s})^2\mathrm{d}\langle\widehat{M}_{i,\cdot},\widehat{M}_{j,\cdot}\rangle_s\right]
\lesssim \frac{n\bar{r}_n}{k_n}\sum_{i=k_n}^\infty E\left[\int_0^t(\widehat{M}_{i,s})^4\mathrm{d}s\right]
\lesssim n\bar{r}_n\cdot k_n\bar{r}_n=o_p(1),
\end{align*}
the Lenglart inequality yields $\mathbf{I}_t=o_p(1)$. On the other hand, by using associativity and linearity we obtain
\begin{align*}
\mathbf{II}_t
=3\psi_2\frac{\sqrt{n}}{k_n^{1/2}}\sum_{p,q=k_n}^\infty c_{g,g^2}(p,q)\int_{I_q(t)}M(I_p(s))\mathrm{d}[M,M]_s
=\frac{3\psi_2}{\sqrt{\theta}}n^{1/4}\mathbb{M}^{(1,1)}_{g,g^2}(M,[M,M])^n_t+o_p(1),
\end{align*}
hence Lemma \ref{HYlem13} yields $\mathbf{II}_t=o_p(1)$, and thus we complete the proof.\hfill $\Box$

\subsubsection{Proof of Proposition \ref{LRVT}}

Application of the Davis and Lenglart inequalities deduces $\mathfrak{S}_{n,m}(t)=\sum_{p=m+1}^{N^n_t+1}|I_p|\frac{1}{m}\sum_{q=1}^{m}E\left[n|I_{p-q+1}|\big|\mathcal{F}^{(0)}_{T_{p-q}}\right]+o_p(1)$. Then we obtain $\mathfrak{S}_{n,m}(t)=\sum_{p=1}^{N^n_t}G^n_{T_{p-1}}|I_p|+o_p(1)$ due to \ref{hypo:A4}, which yields the desired result. \hfill$\Box$

\subsection{Proofs of the results from Section \ref{sec:jumps}}

In the following we set $i(k)=\lceil nS_k\rceil$ for every $k\in\mathbb{N}$.

\subsubsection{Proof of Proposition \ref{jclt}}

\begin{lem}\label{lemma:threshold}
Suppose that $g$ is a function satisfying \ref{hypo:W}. Then, under the assumptions of Proposition \ref{jclt},
\[
\sum_{i=k_n}^{n-k_n+1}P\left(\left|\widetilde{Y}(g)_{i,1}\right|>\frac{\rho_n}{2}\right)\to0\qquad
\textrm{as }n\to\infty.
\]
\end{lem}

\begin{proof}
Take a constant $r$ such that $r>4/(1-4w)$. Then, the Markov inequality and Lemma \ref{lemma:estimate} imply that
\begin{align*}
\sum_{i=k_n}^{n-k_n+1}P\left(\left|\widetilde{Y}(g)_{i,1}\right|>\rho_n\right)
\leq\frac{1}{\rho_n^r}\sum_{i=k_n}^{n-k_n+1}E\left[\left|\widetilde{Y}(g)_{i,1}\right|^r\right]
\lesssim\frac{n^{1-r/4}}{\rho_n^r},
\end{align*}
hence the assumption on $\rho_n$ yields the desired result.
\end{proof}

\begin{proof}[\upshape{\bfseries{Proof of Proposition \ref{jclt}}}]
We start with proving the following equations:
\begin{align*}
\widehat{IV}_n(g_1,\rho_{n})&=\mathbb{H}(Y,Y;g_1)^n-\frac{\psi(g)_1}{\psi(g)_2k_n^2}[Z,Z]^n_1+o_p(n^{-1/4}),\\
\widehat{JV}_n(g_2,\rho_{n})&=2\mathbb{H}(Y,J;g_2)^n+\mathbb{H}(J,J;g_2)^n+o_p(n^{-1/4}),
\end{align*}
where $\mathbb{H}(U,V;g)^n=\frac{1}{\psi(g)_2 k_n}\sum_{i=k_n}^{n-k_n+1}\widetilde{U}(g)_{i,1}\widetilde{V}(g)_{i,1}$. For this, it suffices to show
\begin{align*}
\frac{1}{k_n}\sum_{i=k_n}^{n-k_n+1}\left(\widetilde{Y}(g)_{i,1}\right)^21_{\left\{\left|\widetilde{Z}(g)_{i,1}\right|>\rho_n\right\}}=o_p(n^{-1/4}),\qquad
\frac{1}{k_n}\sum_{i=k_n}^{n-k_n+1}\widetilde{U}(g)_{i,1}\widetilde{J}(g)_{i,1}1_{\left\{\left|\widetilde{Z}(g)_{i,1}\right|\leq\rho_n\right\}}=o_p(n^{-1/4})
\end{align*}
for $U\in\{Y,J\}$. First, Lemma \ref{lemma:threshold} implies that
\begin{align*}
\frac{1}{k_n}\sum_{i=k_n}^{n-k_n+1}\left(\widetilde{Y}(g)_{i,1}\right)^21_{\left\{\left|\widetilde{Z}(g)_{i,1}\right|>\rho_n\right\}}=\frac{1}{k_n}\sum_{i=k_n}^{n-k_n+1}\left(\widetilde{Y}(g)_{i,1}\right)^21_{\left\{\left|\widetilde{Z}(g)_{i,1}\right|>\rho_n,\left|\widetilde{Y}(g)_{i,1}\right|\leq\rho_n/2\right\}}+o_p(n^{-1/4}),\\
\frac{1}{k_n}\sum_{i=k_n}^{n-k_n+1}\widetilde{U}(g)_{i,1}\widetilde{J}(g)_{i,1}1_{\left\{\left|\widetilde{Z}(g)_{i,1}\right|\leq\rho_n\right\}}=\frac{1}{k_n}\sum_{i=k_n}^{n-k_n+1}\widetilde{U}(g)_{i,1}\widetilde{J}(g)_{i,1}1_{\left\{\left|\widetilde{Z}(g)_{i,1}\right|\leq\rho_n,\left|\widetilde{Y}(g)_{i,1}\right|\leq\rho_n/2\right\}}+o_p(n^{-1/4}).
\end{align*}
Next, take $\eta\in(0,2w-\frac{1}{4})$ and set $\mathcal{I}_n=\{i\in\{k_n,\dots,n-k_n+1\}:S_k\in(\frac{i-n^{1/2+\eta}}{n},\frac{i+n^{1/2+\eta}}{n})\textrm{ for some }k=1,\dots,L_1\}$. Such an $\eta$ exists because $w>1/8$. Then, noting that $|\widetilde{J}(g)_{i,1}|$ is sufficiently small if $i\notin\mathcal{I}_n$ because of \ref{hypo:W}, we have
\begin{align*}
\frac{1}{k_n}\sum_{i=k_n}^{n-k_n+1}\left(\widetilde{Y}(g)_{i,1}\right)^21_{\left\{\left|\widetilde{Z}(g)_{i,1}\right|>\rho_n\right\}}
=\frac{1}{k_n}\sum_{i\in\mathcal{I}_n}\left(\widetilde{Y}(g)_{i,1}\right)^21_{\left\{\left|\widetilde{Z}(g)_{i,1}\right|>\rho_n,\left|\widetilde{Y}(g)_{i,1}\right|\leq\rho_n/2\right\}}+o_p(n^{-1/4}),\\
\frac{1}{k_n}\sum_{i=k_n}^{n-k_n+1}\widetilde{U}(g)_{i,1}\widetilde{J}(g)_{i,1}1_{\left\{\left|\widetilde{Z}(g)_{i,1}\right|\leq\rho_n\right\}}
=\frac{1}{k_n}\sum_{i\in\mathcal{I}_n}\widetilde{U}(g)_{i,1}\widetilde{J}(g)_{i,1}1_{\left\{\left|\widetilde{Z}(g)_{i,1}\right|\leq\rho_n,\left|\widetilde{Y}(g)_{i,1}\right|\leq\rho_n/2\right\}}+o_p(n^{-1/4}).
\end{align*}
Now, since it holds that
\begin{align*}
\frac{1}{k_n}\sum_{i\in\mathcal{I}_n}\left(\widetilde{Y}(g)_{i,1}\right)^21_{\left\{\left|\widetilde{Z}(g)_{i,1}\right|>\rho_n,\left|\widetilde{Y}(g)_{i,1}\right|\leq\rho_n/2\right\}}
\leq\frac{\rho_n^2\#\mathcal{I}_n}{4k_n}
\lesssim n^{\eta-2w}
\end{align*}
and
\begin{align*}
\frac{1}{k_n}\left|\sum_{i\in\mathcal{I}_n}\widetilde{U}(g)_{i,1}\widetilde{J}(g)_{i,1}1_{\left\{\left|\widetilde{Z}(g)_{i,1}\right|\leq\rho_n,\left|\widetilde{Y}(g)_{i,1}\right|\leq\rho_n/2\right\}}\right|
\leq\frac{9\rho_n^2\#\mathcal{I}_n}{4k_n}
\lesssim n^{\eta-2w},
\end{align*}
we obtain the desired equations.

Next, by simple calculations we can easily deduce that $\mathbb{H}(J,J;g_2)^n=\sum_{k=1}^{L_1}(\Delta J_{S_k})^2+o_p(n^{-1/4})$ and 
\begin{align*}
\mathbb{H}(Y,J;g_2)^n
&=\sum_{k=1}^{L_1}\left[\sum_{p=i(k)-d_n}^{i(k)+d_n}\left\{c^n_{g_2,g_2}(p,i(k))X(I_p)-\frac{1}{k_n}c^n_{g_2',g_2}(p,i(k))\epsilon_{\frac{p}{n}}\right\}\right]\Delta J_{S_k}+o_p(n^{-1/4}).
\end{align*}
Therefore, we can prove the desired result in a similar manner to the proof of Proposition 6.2 from \cite{Koike2014jclt}, which is based on Propositions 6.6--6.7 and Lemma 6.7 from \cite{Koike2014jclt}.  
\end{proof}

\subsubsection{Proof of Proposition \ref{LAN}}\label{proof:LAN}

We begin by introducing some notation. We denote by $P_{n,\vartheta}$ the law of the vector $\mathbf{z}_n:=(\underline{Z}_1,\dots,\underline{Z}_n)^*$ from \eqref{jump_model} with $\vartheta=(\sigma,\gamma_1,\dots,\gamma_K)$. Define the $n\times n$ matrices $D_n$ and $V_n(\sigma)$ by
\begin{equation}\label{eq:matrices}
D_n^{ij}=\left\{
\begin{array}{ll}
1  & \textrm{if }i=j, \\
-1 & \textrm{if } i=j+1,  \\
0 &  \textrm{otherwise},
\end{array}
\right.\qquad
\textrm{and}\qquad
V_n(\sigma)^{ij}=\left\{
\begin{array}{ll}
\frac{\sigma^2}{n}+\Upsilon  & \textrm{if }i=j=1, \\
\frac{\sigma^2}{n}+2\Upsilon & \textrm{if } 2\leq i=j\leq n, \\
-\Upsilon & \textrm{if } |i-j|=1,\\
0 &  \textrm{otherwise}.
\end{array}
\right.
\end{equation}
Then the law of $D_n\mathbf{z}_n$ under $P_{n,\vartheta}$ is given by $\mathbf{N}(\sum_{k=1}^K\gamma_k\mathbf{e}_{i(k)},V_n(\sigma))$, where $\mathbf{e}_1,\dots,\mathbf{e}_n$ denote the canonical basis of $\mathbb{R}^n$ (recall that $i(k)$ is defined by $i(k)=\lceil nS_k\rceil$). Next, define the $n\times n$ orthogonal matrix $U_n$ by $U_n^{ij}=\frac{2}{\sqrt{2n+1}}\cos\left[\frac{2\pi}{2n+1}\left(i-\frac{1}{2}\right)\left(j-\frac{1}{2}\right)\right]$. Then by Lemma 1 of \citet{KS2013} $U_n$ diagonalizes $V_n(\sigma)$ as $U_nV_n(\sigma)U_n=\Lambda_n(\sigma):=\diag(\lambda^n_1(\sigma),\dots,\lambda^n_n(\sigma))$, where $\lambda^n_i(\sigma)=\frac{\sigma^2}{n}+4\Upsilon\sin^2\left[\frac{\pi}{2}\left(\frac{2i-1}{2n+1}\right)\right]$. Therefore, setting $\mathbf{z}'_n=(Z_1',\dots,Z_n')^*=\Lambda_n(\sigma)^{-\frac{1}{2}}U_n(D_n\mathbf{z}_n-\sum_{k=1}^K\gamma_k\mathbf{e}_{i(k)})$, we have $Z'_i\overset{i.i.d.}{\sim}\mathbf{N}(0,1)$ under $P_{n,\vartheta}$.

\begin{lem}\label{lemma:RL}
Let $(\tau_n)$ be a sequence of real numbers tending to some $\tau\in(-1,0)\cup(0,1)$. Then
\begin{equation}\label{cosine}
\frac{1}{2n+1}\sum_{j=1}^n\frac{1}{\lambda^n_j(\sigma)}\cos\left[\pi\tau_n\left(2j-1\right)\right]=O(1)\qquad
\textrm{as }n\to\infty.
\end{equation}
\end{lem}

\begin{proof}
We imitate a proof of the Riemann-Lebesgue lemma. Since the quantity in the left side of \eqref{cosine} coincides with the real part of
\begin{align*}
\Delta_n:=\frac{1}{2n+1}\sum_{j=1}^n\frac{1}{\lambda^n_j(\sigma)}\exp\left[\sqrt{-1}\pi\tau_n\left(2j-1\right)\right],
\end{align*}
it suffices to prove $\Delta_n=O(1)$. Summation by parts yields
\begin{align*}
\Delta_n=\frac{1}{2n+1}\left[\sum_{j=1}^{n-1}\left(\frac{1}{\lambda^n_j(\sigma)}-\frac{1}{\lambda^n_{j+1}(\sigma)}\right)\frac{1-e^{2\sqrt{-1}\pi\tau_nj}}{e^{-\sqrt{-1}\pi\tau_n}-e^{\sqrt{-1}\pi\tau_n}}
+\frac{1}{\lambda^n_n(\sigma)}\frac{1+e^{-\sqrt{-1}\pi\tau_n}}{e^{-\sqrt{-1}\pi\tau_n}-e^{\sqrt{-1}\pi\tau_n}}\right].
\end{align*}
Therefore, noting that $\lambda^n_1(\sigma)<\cdots<\lambda^n_n(\sigma)$, we have
\begin{align*}
|\Delta_n|\leq\frac{4}{2n+1}\frac{1}{\left|\sin\left(\pi\tau_n\right)\right|}\frac{1}{\lambda^n_1(\sigma)}.
\end{align*}
Since $\lambda^n_1(\sigma)\geq\frac{\sigma^2}{n}$ and $\sin\left(\pi\tau\right)\neq0$, we obtain $\Delta_n=O(1)$.
\end{proof}

\begin{proof}[\upshape{\bfseries{Proof of Proposition \ref{LAN}}}]
Let $h_n=(h_n^{(k)})_{0\leq k\leq K}$ be $(K+1)$-column vectors tending to $h=(h^{(k)})_{0\leq k\leq K}\in\mathbb{R}^{K+1}$ as $n\to\infty$, and set $\delta^n_i=\lambda^n_i(\sigma+n^{-\frac{1}{4}}h^{(0)}_n)/\lambda^n_i(\sigma)-1$. Then, noting that $U_n$ does not depend on $\sigma$, the log-likelihood ratio is given by
\begin{multline*}
\log\left(\frac{\mathrm{d}P_{n,\vartheta+n^{-\frac{1}{4}}h_n}}{\mathrm{d}P_{n,\vartheta}}\right)
=-\frac{1}{2}\sum_{i=1}^n\left\{\log(1+\delta^n_i)-\left(Z_i'\right)^2\frac{\delta^n_i}{1+\delta^n_i}\right\}\\
+\sum_{k=1}^K\left\{n^{-\frac{1}{4}}h_n^{(k)}\sum_{j=1}^n\frac{U^{i(k)j}_n}{(1+\delta^n_j)\sqrt{\lambda^n_j(\sigma)}}Z_j'-\frac{1}{2}\left(n^{-\frac{1}{4}}h_n^{(k)}\right)^2\sum_{j=1}^n\frac{\left(U^{i(k)j}_n\right)^2}{(1+\delta^n_j)\lambda^n_j(\sigma)}\right\}.
\end{multline*}
Similarly to the proof of Eq.(3.2) from \citet{GJ2001a}, we can deduce
\[
\sup_{1\leq i\leq n}\left|\delta^n_i\right|\to0,\qquad
\sum_{i=1}^n(\delta^n_i)^2\to2(h^{(0)})^2(2\sigma\sqrt{\Upsilon})^{-1}.
\]
Therefore, noting that $Z'_i\overset{i.i.d.}{\sim}\mathbf{N}(0,1)$ under $P_{n,\vartheta}$ (especially $Z'_i\{(Z_i')^2-1\}$ is centered under $P_{n,\vartheta}$), it is enough to prove
\begin{equation}\label{fisher}
n^{-\frac{1}{2}}\sum_{j=1}^n\frac{U^{i(k)j}_nU^{i(l)j}_n}{\lambda^n_j(\sigma)}=(2\sigma\sqrt{\Upsilon})^{-1}1_{\{k=l\}}
+O\left(\frac{1}{\sqrt{n}}\right)
\end{equation}
as $n\to\infty$ for any $k,l=1,\dots,K$ in order to derive the desired result.

By a trigonometric identity we can decompose the target quantity as
\begin{align*}
n^{-\frac{1}{2}}\sum_{j=1}^n\frac{U^{i(k)j}_nU^{i(l)j}_n}{\lambda^n_j(\sigma)}&=\frac{2n^{-\frac{1}{2}}}{2n+1}\sum_{j=1}^n\frac{1}{\lambda^n_j(\sigma)}\left\{\cos\left[\pi\phi_n^+(k,l)\left(2j-1\right)\right]+\cos\left[\pi\phi_n^-(k,l)\left(2j-1\right)\right]\right\}\\
&=:\mathbb{I}_n+\mathbb{II}_n,
\end{align*}
where $\phi_n^+(k,l)=\frac{i(k)+i(l)-1}{2n+1}$ and $\phi_n^-(k,l)=\frac{i(k)-i(l)}{2n+1}$. Since $\phi_n^+(k,l)\to \frac{S_k+S_l}{2}\in(0,1)$, Lemma \ref{lemma:RL} yields $\mathbb{I}_n=O(n^{-1/2})$. Similarly we have $\mathbb{II}_n=O(n^{-1/2})$ if $k\neq l$. Now consider the case that $k=l$. 
Applying a standard approximation argument for Riemann sums by the corresponding integral, we obtain $\mathbb{II}_n=n^{-\frac{1}{2}}J_n+O(n^{-1/2})$, where
\[
J_n=\frac{2}{\pi}\int_{\frac{\pi}{2}\frac{1}{2n+1}}^{\frac{\pi}{2}}\frac{1}{\frac{\sigma^2}{n}+4\Upsilon\sin^2(z)}\mathrm{d}z
\]
(cf.~Eq.(8.8) of \citet{GJ2001a}).  
Since $J_n=\frac{\sqrt{n}}{\sigma}(\frac{\sigma^2}{n}+4\Upsilon)^{-1/2}+O(1)$, we obtain $\mathbb{I}_n=(2\sigma\upsilon)^{-1}+O(n^{-1/2})$. This completes the proof of \eqref{fisher}.
\end{proof}

\subsubsection{Proof of Proposition \ref{negative}}

Assume that there is a function $g$ satisfying \ref{hypo:W} and $v_J(g,\theta)=8\sigma\sqrt{\Upsilon}$ for some $\theta>0$. Without loss of generality we may assume $\psi(g)_2=\int_{-\infty}^\infty g(x)^2\mathrm{d}x=1$. Since the minimizer of $\theta\mapsto v_J(g,\theta)$ is $\theta^*=\sqrt{\Phi(g)_{12}\Upsilon/\Phi(g)_{11}\sigma^2}$, we obtain $2\sqrt{\Phi(g)_{22}\Phi(g)_{12}}=1$. Now, setting $K=\phi_{g,g}$, we have $K(0)=\psi(g)_2=1$ and $\Phi(g)_{12}=\int_0^\infty K'(x)^2\mathrm{d}x$ by integration by parts, hence the Schwarz inequality and integration by parts yield
\[
2\sqrt{\Phi(g)_{22}\Phi(g)_{12}}=2\sqrt{\int_0^\infty K(x)^2\mathrm{d}x\int_0^\infty K'(x)^2\mathrm{d}x}
\geq 2\left|\int_0^\infty K(x)K'(x)\mathrm{d}x\right|=1.
\] 
In our case the equality holds true, so there is a constant $c$ such that $K'(x)=cK(x)$ for all $x\geq0$. Since $K(0)=1$ and $K(x)\to0$ as $x\to\infty$, we have $K(x)=e^{cx}$ for all $x\geq0$ and $c<0$. This gives a contradiction because $K'(0)=-\phi_{g',g}(0)=0$ due to integration by parts.\hfill$\Box$

\subsubsection{Proof of Proposition \ref{MLE}}

We use the same notation as in Section \ref{proof:LAN}. Noting that the left side of Eq.\eqref{fisher} is equal to $n^{-\frac{1}{2}}\mathbf{e}_{i(k)}V_n(\sigma)^{-1}\mathbf{e}_{i(l)}$, we have $E\left[\widehat{\gamma}^n_k\right]=n^{-\frac{1}{2}}\cdot2\sigma\sqrt{\Upsilon}\mathbf{e}_{i(k)}V_n(\sigma)^{-1}\sum_{l=1}^K\gamma_l\mathbf{e}_{i(l)}
=\gamma_k+O(n^{-1/2})$. Therefore, it suffices to prove $n^{1/4}(\widehat{\gamma}^n-E\left[\widehat{\gamma}^n\right])\xrightarrow{d}\mathbf{N}(0,2\sigma\sqrt{\Upsilon}E_K)$. $\widehat{\gamma}^n_k-E\left[\widehat{\gamma}^n_k\right]$ can be rewritten as $\widehat{\gamma}^n_k-E\left[\widehat{\gamma}^n_k\right]
=n^{-\frac{1}{2}}\cdot2\sigma\sqrt{\Upsilon}\mathbf{e}_{i(k)}U_n\Lambda_n(\sigma)^{-\frac{1}{2}}\mathbf{z}'_n$. Since $\mathbf{z}'_n\sim \mathbf{N}(0,E_n)$, it is enough to prove $\sqrt{n}\Cov(\widehat{\gamma}^n_k,\widehat{\gamma}^n_l)\to2\sigma\sqrt{\Upsilon}1_{\{k=l\}}$ for all $k,l=1,\dots,K$, which follows from \eqref{fisher}.\hfill$\Box$

\section*{Acknowledgements}

\addcontentsline{toc}{section}{Acknowledgements}

The author is grateful to Teppei Ogihara who pointed out a problem on the mathematical construction of the noise process in a previous version of this paper.  
This work was partly supported by Grant-in-Aid for JSPS Fellows, the Program for Leading Graduate Schools, MEXT, Japan and CREST, JST.

{\small
\renewcommand*{\baselinestretch}{1}\selectfont
\addcontentsline{toc}{section}{References}

\begin{thebibliography}{48}
\expandafter\ifx\csname natexlab\endcsname\relax\def\natexlab#1{#1}\fi

\bibitem[{A\"{i}t-Sahalia \emph{et~al.}(2010)A\"{i}t-Sahalia, Fan and
  Xiu}]{AFX2010}
A\"{i}t-Sahalia, Y., Fan, J. and Xiu, D. (2010).
\newblock High-frequency covariance estimates with noisy and asynchronous
  financial data.
\newblock \emph{J. Amer. Statist. Assoc.} \textbf{105}, 1504--1517.

\bibitem[{A\"{i}t-Sahalia and Jacod(2014)}]{AJ2014}
A\"{i}t-Sahalia, Y. and Jacod, J. (2014).
\newblock \emph{High-frequency financial econometrics}.
\newblock Princeton University Press.

\bibitem[{Barndorff-Nielsen \emph{et~al.}(2008)Barndorff-Nielsen, Hansen, Lunde
  and Shephard}]{BNHLS2008}
Barndorff-Nielsen, O.~E., Hansen, P.~R., Lunde, A. and Shephard, N. (2008).
\newblock Designing realised kernels to measure the ex-post variation of equity
  prices in the presence of noise.
\newblock \emph{Econometrica} \textbf{76}, 1481--1536.

\bibitem[{Barndorff-Nielsen \emph{et~al.}(2011)Barndorff-Nielsen, Hansen, Lunde
  and Shephard}]{BNHLS2011}
Barndorff-Nielsen, O.~E., Hansen, P.~R., Lunde, A. and Shephard, N. (2011).
\newblock Multivariate realised kernels: Consistent positive semi-definite
  estimators of the covariation of equity prices with noise and non-synchronous
  trading.
\newblock \emph{J. Econometrics} \textbf{162}, 149--169.

\bibitem[{Bibinger(2011)}]{B2011a}
Bibinger, M. (2011).
\newblock Efficient covariance estimation for asynchronous noisy high-frequency
  data.
\newblock \emph{Scand. J. Stat.} \textbf{38}, 23--45.

\bibitem[{Bibinger(2012)}]{Bibinger2012}
Bibinger, M. (2012).
\newblock An estimator for the quadratic covariation of asynchronously observed
  {I}t\^{o} processes with noise: Asymptotic distribution theory.
\newblock \emph{Stochastic Process. Appl.} \textbf{122}, 2411--2453.

\bibitem[{Bibinger \emph{et~al.}(2014)Bibinger, Hautsch, Malec and
  Rei{\ss}}]{BHMR2014}
Bibinger, M., Hautsch, N., Malec, P. and Rei{\ss}, M. (2014).
\newblock Estimating the quadratic covariation matrix from noisy observations:
  local method of moments and efficiency.
\newblock \emph{Ann. Statist.} \textbf{42}, 80--114.

\bibitem[{Bibinger and Mykland(2014)}]{BM2014}
Bibinger, M. and Mykland, P.~A. (2014).
\newblock Inference for multi-dimensional high-frequency data: Equivalence of
  methods, central limit theorems, and an application to conditional
  independence testing.
\newblock Preprint, Available at arXiv: http://arxiv.org/abs/1301.2074.

\bibitem[{Bibinger and Winkelmann(2015)}]{BW2015}
Bibinger, M. and Winkelmann, L. (2015).
\newblock Econometrics of co-jumps in high-frequency data with noise.
\newblock \emph{J. Econometrics} \textbf{184}, 361--378.

\bibitem[{Christensen \emph{et~al.}(2010)Christensen, Kinnebrock and
  Podolskij}]{CKP2010}
Christensen, K., Kinnebrock, S. and Podolskij, M. (2010).
\newblock Pre-averaging estimators of the ex-post covariance matrix in noisy
  diffusion models with non-synchronous data.
\newblock \emph{J. Econometrics} \textbf{159}, 116--133.

\bibitem[{Christensen \emph{et~al.}(2013)Christensen, Podolskij and
  Vetter}]{CPV2013}
Christensen, K., Podolskij, M. and Vetter, M. (2013).
\newblock On covariation estimation for multivariate continuous {I}t\^{o}
  semimartingales with noise in non-synchronous observation schemes.
\newblock \emph{J. Multivariate Anal.} \textbf{120}, 59--84.

\bibitem[{Cinlar and Agnew(1968)}]{CA1968}
Cinlar, E. and Agnew, R.~A. (1968).
\newblock On the superposition of point processes.
\newblock \emph{J. R. Stat. Soc. Ser. B Stat. Methodol.} \textbf{30}, 576--581.

\bibitem[{Cl\'{e}ment \emph{et~al.}(2014)Cl\'{e}ment, Delattre and
  Gloter}]{CDG2014}
Cl\'{e}ment, E., Delattre, S. and Gloter, A. (2014).
\newblock Asymptotic lower bounds in estimating jumps.
\newblock \emph{Bernoulli} \textbf{20}, 1059--1096.

\bibitem[{Fukasawa(2010)}]{Fu2010b}
Fukasawa, M. (2010).
\newblock Realized volatility with stochastic sampling.
\newblock \emph{Stochastic Process. Appl.} \textbf{120}, 829--852.

\bibitem[{Gloter and Jacod(2001)}]{GJ2001a}
Gloter, A. and Jacod, J. (2001).
\newblock Diffusions with measurement errors. {I}. {L}ocal asymptotic
  normality.
\newblock \emph{ESAIM Probab. Stat.} \textbf{5}, 225--242.

\bibitem[{Hautsch(2012)}]{Hautsch2012}
Hautsch, N. (2012).
\newblock \emph{Econometrics of financial high-frequency data}.
\newblock Springer.

\bibitem[{Hautsch and Podolskij(2013)}]{HP2012}
Hautsch, N. and Podolskij, M. (2013).
\newblock Pre-averaging based estimation of quadratic variation in the presence
  of noise and jumps: Theory, implementation, and empirical evidence.
\newblock \emph{J. Bus. Econom. Statist.} \textbf{31}, 165--183.

\bibitem[{Hayashi \emph{et~al.}(2011)Hayashi, Jacod and Yoshida}]{HJY2011}
Hayashi, T., Jacod, J. and Yoshida, N. (2011).
\newblock Irregular sampling and central limit theorems for power variations:
  The continuous case.
\newblock \emph{Ann. Inst. Henri Poincar\'{e} Probab. Stat.} \textbf{47},
  1197--1218.

\bibitem[{Hayashi and Yoshida(2005)}]{HY2005}
Hayashi, T. and Yoshida, N. (2005).
\newblock On covariance estimation of non-synchronously observed diffusion
  processes.
\newblock \emph{Bernoulli} \textbf{11}, 359--379.

\bibitem[{Hayashi and Yoshida(2011)}]{HY2011}
Hayashi, T. and Yoshida, N. (2011).
\newblock Nonsynchronous covariation process and limit theorems.
\newblock \emph{Stochastic Process. Appl.} \textbf{121}, 2416--2454.

\bibitem[{Jacod \emph{et~al.}(2010)Jacod, Podolskij and Vetter}]{JPV2010}
Jacod, J., Podolskij, M. and Vetter, M. (2010).
\newblock Limit theorems for moving averages of discretized processes plus
  noise.
\newblock \emph{Ann. Statist.} \textbf{38}, 1478--1545.

\bibitem[{Jacod and Protter(2012)}]{JP2012}
Jacod, J. and Protter, P. (2012).
\newblock \emph{Discretization of processes}.
\newblock Springer.

\bibitem[{Jacod and Shiryaev(2003)}]{JS}
Jacod, J. and Shiryaev, A.~N. (2003).
\newblock \emph{Limit theorems for stochastic processes}.
\newblock Springer, 2nd edn.

\bibitem[{Koike(2013)}]{Koike2013hit}
Koike, Y. (2013).
\newblock Central limit theorems for pre-averaging covariance estimators under
  endogenous sampling times.
\newblock Unpublished paper, Available at arXiv:
  http://arxiv.org/abs/1305.1229.

\bibitem[{Koike(2014)}]{Koike2014phy}
Koike, Y. (2014).
\newblock Limit theorems for the pre-averaged {H}ayashi-{Y}oshida estimator
  with random sampling.
\newblock \emph{Stochastic Process. Appl.} \textbf{124}, 2699--2753.

\bibitem[{Koike(2015{\natexlab{a}})}]{Koike2015pthy}
Koike, Y. (2015{\natexlab{a}}).
\newblock Estimation of integrated covariances in the simultaneous presence of
  nonsynchronicity, microstructure noise and jumps.
\newblock \emph{Econometric theory (forthcoming),
  doi:10.1017/S0266466614000954} .

\bibitem[{Koike(2015{\natexlab{b}})}]{Koike2014jclt}
Koike, Y. (2015{\natexlab{b}}).
\newblock Quadratic covariation estimation of an irregularly observed
  semimartingale with jumps and noise.
\newblock \emph{Bernoulli (forthcoming), Available at arXiv:
  http://arxiv.org/abs/1408.0938v2} .

\bibitem[{Kunitomo and Sato(2013)}]{KS2013}
Kunitomo, N. and Sato, S. (2013).
\newblock Separating {I}nformation {M}aximum {L}ikelihood estimation of the
  integrated volatility and covariance with micro-market noise.
\newblock \emph{The North American Journal of Economics and Finance}
  \textbf{26}, 282--309.

\bibitem[{Li \emph{et~al.}(2014{\natexlab{a}})Li, Todorov and
  Tauchen}]{LTT2014}
Li, J., Todorov, V. and Tauchen, G. (2014{\natexlab{a}}).
\newblock Jump regressions.
\newblock Working paper.

\bibitem[{Li \emph{et~al.}(2014{\natexlab{b}})Li, Mykland, Renault, Zhang and
  Zheng}]{LMRZZ2014}
Li, Y., Mykland, P.~A., Renault, E., Zhang, L. and Zheng, X.
  (2014{\natexlab{b}}).
\newblock Realized volatility when sampling times are possibly endogenous.
\newblock \emph{Econometric Theory} \textbf{30}, 580--605.

\bibitem[{Li \emph{et~al.}(2013)Li, Zhang and Zheng}]{LZZ2012}
Li, Y., Zhang, Z. and Zheng, X. (2013).
\newblock Volatility inference in the presence of both endogenous time and
  microstructure noise.
\newblock \emph{Stochastic Process. Appl.} \textbf{123}, 2696--2727.

\bibitem[{Liu and Tang(2014)}]{LT2014}
Liu, C. and Tang, C.~Y. (2014).
\newblock A quasi-maximum likelihood approach for integrated covariance matrix
  estimation with high frequency data.
\newblock \emph{J. Econometrics} \textbf{180}, 217--232.

\bibitem[{Mancini(2001)}]{M2001}
Mancini, C. (2001).
\newblock Disentangling the jumps of the diffusion in a geometric jumping
  {B}rownian motion.
\newblock \emph{Giornale dell'Istituto Italiano degli Attuari} \textbf{64},
  19--47.

\bibitem[{Mykland and Zhang(2009)}]{MZ2009}
Mykland, P.~A. and Zhang, L. (2009).
\newblock Inference for continuous semimartingales observed at high frequency.
\newblock \emph{Econometrica} \textbf{77}, 1403--1445.

\bibitem[{Mykland and Zhang(2012)}]{MZ2012}
Mykland, P.~A. and Zhang, L. (2012).
\newblock The econometrics of high-frequency data.
\newblock In M.~Kessler, A.~Lindner and M.~S{\o}rensen, eds., \emph{Statistical
  methods for stochastic differential equations}, chap.~3. CRC Press.

\bibitem[{Ob\l\'{o}j(2004)}]{Obloj2004}
Ob\l\'{o}j, J. (2004).
\newblock The {S}korokhod embedding problem and its offspring.
\newblock \emph{Probab. Surv.} \textbf{1}, 321--392.

\bibitem[{Ogihara(2014)}]{Ogihara2014noise}
Ogihara, T. (2014).
\newblock Parametric inference for nonsynchronously observed diffusion
  processes in the presence of market microstructure noise.
\newblock Available at arXiv: http://arxiv.org/abs/1412.8173.

\bibitem[{Podolskij and Vetter(2009)}]{PV2009}
Podolskij, M. and Vetter, M. (2009).
\newblock Estimation of volatility functionals in the simultaneous presence of
  microstructure noise and jumps.
\newblock \emph{Bernoulli} \textbf{15}, 634--658.

\bibitem[{Podolskij and Vetter(2010)}]{PV2010}
Podolskij, M. and Vetter, M. (2010).
\newblock Understanding limit theorems for semimartingales: a short survey.
\newblock \emph{Stat. Neerl.} \textbf{64}, 329--351.

\bibitem[{Potiron and Mykland(2015)}]{PM2015}
Potiron, Y. and Mykland, P.~A. (2015).
\newblock Estimation of integrated quadratic covariation between two assets
  with endogenous sampling times.
\newblock Working paper, Available at arXiv: http://arxiv.org/abs/1507.01033.

\bibitem[{Rei{\ss}(2011)}]{Reiss2011}
Rei{\ss}, M. (2011).
\newblock Asymptotic equivalence for inference on the volatility from noisy
  observations.
\newblock \emph{Ann. Statist.} \textbf{39}, 772--802.

\bibitem[{Renault and Werker(2011)}]{RW2011}
Renault, E. and Werker, B.~J. (2011).
\newblock Causality effects in return volatility measures with random times.
\newblock \emph{J. Econometrics} \textbf{160}, 272--279.

\bibitem[{Resnick and Tomkins(1973)}]{RT1973}
Resnick, S.~I. and Tomkins, R.~J. (1973).
\newblock Almost sure stability of maxima.
\newblock \emph{J. Appl. Probab.} \textbf{10}, 387--401.

\bibitem[{Robert and Rosenbaum(2012)}]{RR2012}
Robert, C.~Y. and Rosenbaum, M. (2012).
\newblock Volatility and covariation estimation when microstructure noise and
  trading times are endogenous.
\newblock \emph{Math. Finance} \textbf{22}, 133--164.

\bibitem[{Von~Schelling(1954)}]{VonSchelling1954}
Von~Schelling, H. (1954).
\newblock Coupon collecting for uneqal probabilities.
\newblock \emph{Amer. Math. Monthly} \textbf{61}, 306--311.

\bibitem[{Xiu(2010)}]{Xiu2010}
Xiu, D. (2010).
\newblock Quasi-maximum likelihood estimation of volatility with high frequency
  data.
\newblock \emph{J. Econometrics} \textbf{159}, 235--250.

\bibitem[{Zhang(2006)}]{Z2006}
Zhang, L. (2006).
\newblock Efficient estimation of stochastic volatility using noisy
  observations: a multi-scale approach.
\newblock \emph{Bernoulli} \textbf{12}, 1019--1043.

\bibitem[{Zhang(2011)}]{Zhang2011}
Zhang, L. (2011).
\newblock Estimating covariation: {E}pps effect, microstructure noise.
\newblock \emph{J. Econometrics} \textbf{160}, 33--47.

\end{thebibliography}

}

\end{document}